\documentclass[preprint]{imsart}

\input{packages}

\usepackage[margin=1in]{geometry}

\startlocaldefs
\numberwithin{equation}{section}

\newcommand{\NAME}{ScreeNOT\,}

\newcommand{\SIname}{{\em \NAME}: Exact MSE-Optimal Singular Value
Thresholding in Correlated Noise}
\newcommand{\shortname}{{\em \NAME}: Exact MSE-Optimal Singular Value Thresholding}
\newcommand{\SIshortname}{{\em \NAME}}
\newcommand{\arxivnumber}{2009.12297}

\theoremstyle{theorem} \newtheorem{thm}{Theorem}
\theoremstyle{theorem} 
\theoremstyle{theorem} \newtheorem{defn}{Definition}
\theoremstyle{theorem} \newtheorem{lemma}{Lemma}
\theoremstyle{theorem} \newtheorem{prop}{Proposition}
\newtheorem*{theorem*}{Theorem}
\newtheorem*{corollary*}{Corollary}

\newcommand{\R}{\ensuremath{\mathbb{R}}}
\newcommand{\C}{\ensuremath{\mathbb{C}}}

\newcommand{\E}{\ensuremath{\mathbb{E}}}

\newcommand{\V}[1]{\ensuremath{\mathbf{#1}}}

\newcommand{\aslim}{\stackrel{a.s.}{\longrightarrow}}
\newcommand{\dlim}{\stackrel{d}{\longrightarrow}}

\newcommand{\iid}{\stackrel{\text{iid}}{\sim}}

\newcommand{\Pd}[3]{\ifthenelse{\equal{#3}{1}}{\frac{\partial #1}{\partial #2}}{\frac{\partial^{#3} #1}{\partial #2^{#3}}}}

\newcommand{\Ind}[1]{\mathds{1}_{\left\{{#1}\right\}}}

\newcommand{\prob}{\mathbb{P}}

\newcommand{\T}{\top}

\newcommand{\rank}{\mathrm{rank}}

\newcommand{\m}{\mathcal}

  \newcommand{\Cc}{\mathcal{C}}

  \newcommand{\Yc}{\mathcal{Y}}

  \newcommand{\optThresh}{T_{\gamma}}

\newcommand{\optIntervU}{\overline{\Theta}(\V{x})}
\newcommand{\optIntervL}{\underline{\Theta}(\V{x})}

\newcommand{\bulkEdge}{\Zplus}
\newcommand{\bulkedge}{\zplus}
\newcommand{\BBP}{\Xplus}
\newcommand{\bbp}{\xplus}

\newcommand{\ASE}{\mathrm{ASE}}
\newcommand{\SEn}{\mathrm{SE}_n}

\newcommand{\FZ}{F_Z}

\newcommand{\FZn}{F_{Z_n}}
\newcommand{\fZ}{f_Z}

\newcommand{\FYn}{F_{Y_n}}
\newcommand{\FY}{F_{Y}}

\newcommand{\Xplus}{{\cal X}_+}
\newcommand{\Zplus}{{\cal Z}_+}
\newcommand{\xplus}{{x}_+}
\newcommand{\zplus}{{z}_+}

\newcommand{\FS}{F_S}
\newcommand{\FSn}{F_{S_n}}
\newcommand{\UpperEdgeS}{{\lambda}_+}
\newcommand{\LowerEdgeS}{{\lambda}_-}

\newcommand{\TCritBare}{\Psi}
\newcommand{\TCrit}{\Psi_\gamma}
\newcommand{\TCritGamma}{\Psi_\gamma}
\newcommand{\TCritPoverN}{\Psi_{p/n}}

\newcommand{\thresh}{\theta}
\newcommand{\threshopt}{\theta_{\mbox{opt}}}

\endlocaldefs

\arxiv{arXiv:\arxivnumber}

\begin{document}

\begin{frontmatter}

%\title{\longname}
\title{ \SIname }  
\runtitle{\shortname}

\begin{aug}
  %%%%%%%%%%%%%%%%%%%%%%%%%%%%%%%%%%%%%%%%%%%%%%%
  %% Only one address is permitted per author. %%
  %% Only division, organization and e-mail is %%
  %% included in the address.                  %%
  %% Additional information can be included in %%
  \author[A]{\fnms{David}~\snm{Donoho}\ead[label=e1]{donoho@stanford.edu}},
  \author[B]{\fnms{Matan}~\snm{Gavish}\ead[label=e2]{gavish@cs.huji.ac.il}}
  \and
  \author[A]{\fnms{Elad}~\snm{Romanov}\ead[label=e3]{eromanov@stanford.edu}}
  %%%%%%%%%%%%%%%%%%%%%%%%%%%%%%%%%%%%%%%%%%%%%%
  %% Addresses                                %%
  %%%%%%%%%%%%%%%%%%%%%%%%%%%%%%%%%%%%%%%%%%%%%%
  \address[A]{Department of Statistics, Stanford University\printead[presep={,\ }]{e1,e3}}
  
  \address[B]{School of Computer Science and Engineering, Hebrew University of Jerusalem\printead[presep={,\ }]{e2}}
  \end{aug}

% \begin{aug}
% \author{\fnms{David}
% \snm{Donoho}\thanksref{nsf,bsf}\ead[label=e1]{donoho@stanford.edu}}
% ,
% \author{\fnms{Matan}
% \snm{Gavish}\thanksref{bsf}\ead[label=e2]{gavish@cs.huji.ac.il}}
% \and
% \author{\fnms{Elad}
% \snm{Romanov}\thanksref{isf}\ead[label=e3]{elad.romanov@mail.huji.ac.il}}

% \thankstext{nsf}{partially supported by NSF DMS 1407813, 1418362, and 1811614}
% \thankstext{bsf}{partially supported by United States – Israel Binational Science Foundation (BSF)
%  2016201}
% \thankstext{isf}{partially supported by Israel Science Foundation 1523/16}
% \runauthor{D. Donoho, M. Gavish and E. Romanov}

% \affiliation{Stanford University and Hebrew University}

% \address{
% D. Donoho\\
%   Department of Statistics\\
% Sequoia Hall\\
% 390 Serra Mall\\ 
% Stanford University\\
% Stanford, CA 94305-4065\\
% USA\\
% \printead{e1}
% %\phantom{E-mail:\ }\printead*{e1,e3}
% }

% \address{
%   M. Gavish and E. Romanov\\
% School of Computer Science and Engineering\\
% Hebrew University of Jerusalem\\
% Edmond J. Safra Campus\\
% Jerusalem 91904\\ 
% Israel\\
% \printead{e2}\\
% %\printead{e3}\\
% \phantom{E-mail:\ }\printead*{e3}
% }

% \end{aug}

\begin{abstract}
  \noindent  We derive a formula for  optimal hard thresholding of the singular value decomposition in the presence of correlated additive noise; although it nominally involves unobservables, we show how to apply it even where the noise covariance structure is not a-priori known or is not independently estimable.  The proposed method, which we call {\bf \NAME}, is a mathematically solid alternative to Cattell's ever-popular but vague Scree Plot heuristic from 1966.  \NAME has a surprising oracle property: it typically achieves {\it exactly}, in large finite samples, the lowest possible MSE for matrix recovery, on each given problem instance -- i.e. the specific threshold it selects gives exactly the smallest achievable MSE loss among all possible threshold choices for {\it that} noisy dataset and {\it that} unknown underlying true low rank model.  The method is computationally efficient and robust against perturbations of the underlying covariance structure.  Our results depend on the assumption that the singular values of the noise have a limiting empirical distribution of compact support; this property, which is standard in random matrix theory, is satisfied by many models exhibiting either cross-row correlation structure or cross-column correlation structure, and also by many situations with more general, inter-element correlation structure. Simulations demonstrate the effectiveness of the method even at moderate matrix sizes. The paper is supplemented by ready-to-use software packages implementing the proposed algorithm: package \texttt{ScreeNOT} in \texttt{Python} (via {\tt PyPI}) and \texttt{R} (via {\tt CRAN}).

\end{abstract}



\end{frontmatter}

\section{Introduction}

Across a wide variety of scientific and technical fields, practitioners have found many valuable applications of  {\it singular value thresholding} (SVT).  This procedure starts from the singular value decomposition (SVD), which represents the data matrix $Y$ as 
\begin{equation} \label{eq:SVDdef} Y=\sum_{i=1}^{\min(n,p)} y_i \cdot \V{u}_i \V{v}_i^\T, \end{equation} 
using the  empirical singular values $\{y_i\}_{i=1}^{\min(n,p)}$, and the empirical left- and right- singular vectors of $Y$, denoted here $\V{u}_i$ and  $\V{v}_i$. 

In such applications, it is generally
claimed that the small singular values
represent `noise' and the large singular values
`signal';  practitioners attempt to
separate signal from noise by 
setting a threshold $\theta$ (say), and 
using, in place of $Y$, the partial
reconstruction containing only would-be signal components:
\begin{eqnarray} \label{eq:SVTdef} 
  \hat{X}_\thresh = \sum_{i} y_i\,\Ind{i: y_i> \thresh} \cdot \V{u}_i\V{v}_{i}^\T \,. 
\end{eqnarray}

How do practitioners determine the threshold $\thresh$?  Often, by eye. They
plot the ordered singular values and spot `elbows'.  Sometimes, they give this a
scholarly veneer by saying they are using the `scree-plot method'; they might
even formally cite the originator of this folk-tradition \cite{Cattell1966},
which still gets more than 1000 citations yearly. According to the method
prescribed in that paper, the practitioner plots the values $\{y_i\}$ and uses
her {\em eyes} to distinguish between `signal' and `noise' singular values of
$Y$. 

How {\it should} they determine the threshold?
Relevant theory and methodology literature 
spans multiple disciplines over multiple decades; we mention only
a few entry points, including:
\cite{Wold1978,Jackson1993,Lagerlund1997,Edfors1998,Alter2000,Achlioptas2001,
Azar2001,Jolliffe2005,Price2006,Hoff2006,Bickel2008,Owen2009,Perry2009,
Chatterjee2010,Donoho2013b,owenDobriban}. 
Progress has been made in our understanding of the underlying problem, and many valuable
quantitative approaches have been 
developed \-- to which we here add one more.
Our contribution relies on recent advances 
in random matrix theory which
point, we think convincingly,
to the method introduced here. 
This method typically offers the exact optimal loss 
available on each specific,
finite dataset $Y$.

Our task formalization supposes that: (a) there is an 
underlying matrix $X$ of fixed rank $r$ \-- though $X$ and
even its rank $r$ are unknown to us; (b) only a potentially loose upper bound on the signal rank $r$ is known;
(c)  the data matrix $Y$ has the signal+noise
form $Y = X + Z$, where $Z$ is a noise matrix 
with a general covariance structure -- also unknown to us;
(d)  we use hard thresholding of singular values, 
exactly as in (\ref{eq:SVTdef}) above\footnote{And not some variant, 
such as soft thresholding or a more general shrinkage.};
(e) we adopt
squared error {\it loss}\footnote{$\| X \|_F^2=
  \sum_{i,j}X_{i,j}^2$ denotes the squared Frobenius norm.}:
\begin{equation} \label{eq:SEdef}
\mathrm{SE}[X|\thresh] =
||\hat{X}_\thresh-X||_F^2.
\end{equation}
As goal, we literally aim to choose a 
loss-minimizing value 
$\threshopt = \threshopt(Y|X)$
solving:
\begin{equation} \label{eq:oracle}
\mathrm{SE}[X| \threshopt] = \min_\thresh \mathrm{SE}[X| \thresh].
\end{equation}
Aiming for $\threshopt(Y|X)$ may seem overambitious,
as we know only the data matrix $Y$,
and not $X$. Wait and see.

Essentially this problem was studied previously 
by two of the authors in the special case of white noise.
 \cite{Donoho2013b} supposed that
the underlying noise $Z$ matrix has i.i.d Gaussian 
zero-mean entries and the problem
is scaled so that the columns of $Z$ have unit Euclidean squared norm in expectation,
and considered a sequence of increasingly large problems.  
In the square case, when $Y$ has as many
rows as columns: the authors found results\footnote{That is, the authors of  \cite{Donoho2013b} adopted a slightly different viewpoint involving asymptotic MSE, and showed that $4/\sqrt{3}$ is optimal, whereas we consider here exact finite sample MSE loss, and show that with eventually overwhelming probability, $4/\sqrt{3}$ is exactly optimal {\it on each typical realization}.} which,
in light of our results below, say
that, with eventually overwhelming probability,
we have $\threshopt = 4/\sqrt{3}$.
Their analysis relied on
then-recent advances in the `Johnstone spiked model'
of random matrix theory \cite{Johnstone2001};   they proposed 
a method
for white noise with unknown variance, 
where the threshold formula became
$\threshopt \approx 4/\sqrt{3} \cdot \frac{y_{\mathrm{med}}}{\sqrt{n \cdot .6528 }}$,
where $y_{\mathrm{med}}$ denotes the median empirical singular value of $Y$.\footnote{$\sqrt{.6528}$ is approximately the median of the standard quarter-circle law; see the original paper.} 

Understanding the white noise case cannot be the end of the story.
Practitioners ordinarily don't know that their
noise is white, and in fact realistic noise models
can include correlations between columns, rows, or even 
general row-column combinations.
Fortunately, a broad range
of noise models can be studied using 
appropriate advances that have been made in 
random matrix theory. In this broader context, as we show,
a more general formula for the optimal threshold 
can be given, which of course reduces to $4/\sqrt{3}$ in
the above `square-matrix in white noise' case, but which is inevitably quite a bit more sophisticated in general. 

Section 2 below describes {\bf \NAME}, our proposed deployment of this formula on
actual data.  The acronym NOT stands for {\em Noise-adaptive Optimal
Thresholding}; `adaptive' refers to the algorithm's optimality across a wide
range of unknown noise covariances.  The prefix `Scree' reminds us that, still
today, in many cases, the alternative would simply be `eyeballing' the Scree
Plot \cite{Cattell1966}.  Cattell and his many followers clearly believed that
{\it something}, some visible feature, in the Scree Plot -- namely, in the
collection of data singular values  $\{y_i\}$ -- could tell us where the noise
stopped and the signal began. But what exactly? In a very concrete sense, the
\NAME algorithm shows that the information needed to separate signal from noise
truly {\it is} there in the distribution of empirical singular values, where Cattell
and his followers all hoped it would be. 
However, 
the \NAME algorithm and the approach we develop here quantitatively
identify this information 
as a specific {\it functional of the CDF of singular values}.

The method, once implemented, surprised us by the
{\it finite-sample} optimality it exhibited;
in simulations at reasonable problem sizes it
typically achieves the exact minimal loss (\ref{eq:oracle})
for the given dataset, even though the method is not entitled to know the 
underlying low-rank model  $X$ or specifics of the noise model on $Z$;
we initially expected a weaker and more `asymptotic' optimality property, perhaps similar to the one shown in \cite{Donoho2013b}.
Our analysis below proves typicality of such exact optimality in finite samples.
This strong optimality is partly due to the penetrating nature of 
random matrix theory; but also to the very specific task: 
minimizing squared error loss (\ref{eq:SEdef}) of 
singular value thresholding (\ref{eq:SVTdef}).

\paragraph{Underlying Analysis}

Hoping to make the paper helpful to prospective users of the proposed method,
we have  made the Introduction and also Section 2 mostly independent of the
analysis to come; however, we now very briefly offer
mathematically-oriented  readers some insight about 
the approach being followed in later sections and the tools being developed there.

At heart, this paper concerns the asymptotic analysis of a 
sequence of matrix recovery problems
where the problem sizes $n$ and $p$  grow to $\infty$
in a proportional fashion. 
We assume that the matrix $X$ has $r$ nonzero singular values $x_1, \dots x_r$
which are fixed independently of $n$ and $p$.
About the sequence of random noise matrices $Z = Z_{n,p}$,  
we assume that the sequence of empirical cumulative distribution functions (CDF's) of
noise singular values converges to a compactly supported distribution $\FZ$
with certain qualitative restrictions at boundary of the support.

Using results of Benaych-Georges and Nadakuditi \cite{benaych2012singular},
we obtain an expression for an asymptotically optimal 
hard threshold, as a {\em functional} $T(\cdot)$ of the limiting CDF of noise singular values $\FZ$.
The functional is continuous and even differentiable in certain senses.

Admittedly, the limiting CDF of noise singular values $\FZ$ is
not observable to the statistician, as we only observe
a sample of the signal+noise singular values mixed together. 
Performing a kind of amputation and prosthetic extension
on the CDF $\FY$ of singular values of $Y$, which we do observe,
we construct a modified empirical CDF $\hat{F}_n$ which consistently estimates the
limiting CDF of noise-only singular values. Applying the
hard threshold selection functional to this modified empirical
CDF $\hat{F}_n$ gives our proposed method, in the form $\hat{\thresh} = T(\hat{F}_n)$. 
As we show in Section 2, there is a quite explicit and
computationally tractable algorithm for  computing $T(\hat{F}_n)$,
which we label {\bf \NAME}. 

Owing to the continuity
of the hard threshold functional $T(\cdot)$, and the consistency of the constructed CDF,
the resulting method is a consistent estimator of the underlying 
asymptotically optimal threshold $T(\FZ)$.
We also prove a finite-sample optimality of the method. 
 Specifically, the \NAME algorithm is shown to be exactly optimal for squared error loss
  with high probability, in large-enough finite samples,
  under very general model assumptions. 
  For generic configurations of signal singular values $(x_i)_{i=1}^r$,
  there is, in large finite samples, an {\em optimal interval} of thresholds, 
  all achieving the optimal MSE at that realization;
the consistency of the optimal threshold estimator implies
  that eventually for large-enough $n$, with overwhelming probability,
   the proposed method achieves the exact optimal MSE loss.

\paragraph{Contributions}

The approach we develop selects an optimal threshold for singular values, and thus selects the ``signal'' singular values, based on the principle of minimizing SE. 
Indeed, minimizing squared error loss is a ubiquitous goal in 
statistical theory, and we are not
the first to consider it as a goal for threshold selection
in the context of singular values.
In addition to our own just-cited
work \cite{Donoho2013b}, prior citable work
on squared-error-loss includes Perry
\cite{Perry2009},
Shabalin and Nobel \cite{Shabalin2013}
and Nadakuditi \cite{Nadakuditi}, although
much of this concerns singular value shrinkage 
rather than thresholding\footnote{Singular value shrinkage is considerably more involved as it changes the data singular values rather than select them. The best-possible relative improvement of shrinkage over  thresholding was studied in \cite{Donoho2013b} in the white noise case. }.
% see also
%  \cite{gavish2017optimal} 
% for other losses.
%\Revision{
%Nadakuditi \cite{Nadakuditi} in particular considers optimal singular value %shrinkage in the spiked model under an arbitrary compactly-supported noise %spectral distribution, similar to the model considered in the present paper. It %is important to emphasize that in the event where one is interested purely in %estimating the low-rank signal part $X$, singular value shrinkage often yields %better performance in terms of squared error. We stress that in this paper, %squared error minimization is advocated as a \emph{guiding principle} for %principal component selection, a task which is often done as a step in practical %data-processing pipelines. That is, we do not consider SE minimization as a goal %in itself: singular value hard thresholding is a strictly suboptimal tool for %this task.
%% derives optimal singular value shrinkage rules in spiked model, 
%}
While our approach is implemented for SE loss, we note that it could in principle be used to develop optimal thresholding rules for other loss criteria, such as the operator norm.

%Instead of squared error loss, one could  consider minimizing other loss %criteria, such as the operator norm; to that end, one would have to compute the %optimal threshold under the corresponding loss. This is beyond the scope of the %present paper; we remark, however, that optimal singular value shrinkers, under %several different loss criteria beyond squared error, are already available in %the literature for the white-noise model \cite{gavish2017optimal}.
%}

We especially point to \cite{owenDobriban}; in this
work Dobriban and Owen 
mainly study Parallel Analysis \cite{franklin1995parallel} ---
simulation-based significance 
testing for large singular values;
they develop tools from Random Matrix Theory 
to derandomize Parallel Analysis\footnote{Potential users of ScreeNOT should note that in many scientific projects the goal
is determining the number of statistically significant factors,
without particular regard to the quality of squared-error approximation;
it is possible that for such a goal
Parallel Analysis or its derandomized version \cite{owenDobriban}
is preferred.}.
Beyond this, they also mention in a final section that their tools
could be adapted to produce a threshold selector minimizing
the asymptotic mean-squared error of the resulting approximation;
and their equation (6) provides a way to characterize 
such a functional.

In this paper, we make explicit (in Equation \ref{eq:optFunc} below) 
the functional $T$  
for threshold selection with minimal asymptotic SE,
and show that it is well-defined;
we offer (in Section \ref{ssec-Internals}) 
an explicit construction of a modified CDF 
$\hat{F}_n$ to plug in to $T$ \-- the proposal involves
singular value ``ablation and prosthesis''; we
develop a theoretical machinery, involving continuity
properties of $T$ and convergence properties of $\hat{F}_n$,
and use the machinery to prove that $T(\hat{F}_n)$
achieves not just asymptotic optimal loss (Theorem 1),
but also (Theorem 2) that in large finite samples 
it achieves the exact minimal loss with overwhelming probability. 

{\bf Outline.} This paper is organized as follows. In Section
\ref{sec:practical} we offer a practical, succinct description of the \NAME
algorithm,
for the convenience of prospective users. 
%\Revision{The presentation in this section is rather light on mathematical %details, which will all be given in later sections; this is done to maximize %readability for as wide an audience as possible.} 
In Section \ref{sec:setup}
we introduce the signal+noise model used and survey relevant results from random
matrix theory. In Section \ref{sec:results} we state our main results regarding
the optimality and stability properties of \NAME, both in finite matrix size and
asymptotically as the matrix size grows to infinity. In Section
\ref{sec:numerics} we demonstrate the mathematical results in various
simulations and numerical examples; for space considerations only a handful of
figures are shown, with most simulation results deferred to the supplementary
article 
% \cite{SI}
 and available in the code supplement \cite{SDR}. The results
are proved in Section \ref{sec:proofs}, with some proofs referred to 
the supplementary article.
% \cite{SI}.

{\bf Reproducibility advisory.} Implementation of the proposed algorithm, scripts generating all figures in this paper, and many additional simulations have been permanently deposited and are available at the code supplement 
\cite{SDR}.

{\bf Code packages.} Ready-to-use code packages, implementing the \NAME procedure in various language are available. 
In {\tt Python}: package \texttt{ScreeNOT} is available through {\tt PyPI}; 
in {\tt R}: package {\tt ScreeNOT} is available though {\tt CRAN}; and {\tt Matlab} source code.  For details, see the following {\tt GitHub} repository: \url{https://github.com/eladromanov/ScreeNOT}. In addition, the source code has been permanently deposited and is available at the code supplement 
\cite{SDR}. 

\section{The \NAME ~Procedure: User-level description}
\label{sec:practical}

In this section we give a brief self-contained description of our proposed
procedure. 
%Note that {\tt Python}, {\tt R} and {\tt Matlab} packages implementing the \NAME algorithm have been published, see the code supplement 
%\cite{SDR} for more information.

\subsection{Procedure API}

\NAME selects a hard threshold for singular
values, which can in finite samples give the optimal
MSE approximation of a low rank matrix
from a noisy version; the noise may be correlated,
and the threshold will adapt to that appropriately.

\subsubsection{Inputs}
The user provides these inputs to \NAME:
\begin{description}
\item  {\bf y:} \qquad the singular values $y_{1} \dots, y_{\min(n,p)}$
of the data matrix $Y$;
\item  {$n$, $p$:} \quad size parameters of the data matrix $Y$.
\item  {$k$:}  \qquad upper bound on the rank $r$ of the underlying unknown 
signal matrix $X$ which is to be recovered. This upper bound may be very loose.
\end{description}

\subsubsection{Outputs}
\NAME returns $\hat{\theta} = \hat{\theta}(Y)$, the value to be used in
singular value thresholding. 

To use the threshold, the user should  reconstruct 
an approximation to  the underlying signal matrix $X$ 
using the empirical singular values $y_i$ and the empirical singular vectors 
$\V{u_i}$ and $\V{v}_i$ as follows:
\[
  \hat{X} = \sum_i y_i \Ind{ y_i > \hat{\theta} } \cdot \V{u}_i \V{v}_i^\T\,.
\]
In this reconstruction, the singular values smaller than $\hat{\theta}$ are judged to
be noise and the corresponding singular decomposition components
are ignored.

\subsection{Example in a stylized application}
\label{subsec:example}
We next construct a synthetic-data example,
in which we know the ground truth for demonstration purposes.
The synthetic data $Y = X +Z$ 
and the invocation of \NAME are based on these 
ingredients.
 
 \begin{description}
\item {Signal $X$:} The underlying signal matrix, unbeknownst to the
hypothetical user, has rank 10, with singular values $(x_{10},\ldots,x_1)=(1.0, 1.15, 1.3, \ldots, 2.35)$.
\item {Noise $Z$:}  The underlying noise, unbeknownst to the 
hypothetical user,  follows an
$AR(1)$ process in the row index, within each column.
The $AR(1)$ process has parameter $\rho=0.4$, and additionally each entry is divided by $\sqrt{n}$, so to have {variance\footnote{That is, the columns of $Z$ are independent and distributed as $\V{z}/\sqrt{n}$, where the random vector $\V{z}$ has entries: 
$
z_1=\varepsilon_1$ and $z_i=\rho \cdot z_{i-1} + \sqrt{1-\rho^2}\cdot\varepsilon_i$ for $2\le i \le p\,,
$
where ${\varepsilon}_1,\ldots,{\varepsilon}_p \iid \m{N}(0,1)$. }
$1/n$.}
\item {Problem Size: $n=p=1000$}
\item { Rank bound: $k=15$.} The user specifies a
bound of $k=15$ on the possible rank of the signal.
\end{description}

Figure~\ref{fig:IntroFig1} shows a so-called {\it Scree Plot} \cite{Cattell1966} of the first 30 empirical singular values $y_i$. For this particular instance, it is verified (by exhaustive search) that the minimal loss is attained by retaining the first three principal components of $Y$; in other words, thresholding at any point $\theta \in (y_3,y_4)$ is optimal. The threshold $\hat{\theta}$ returned by the \NAME~procedure is indicated by the green horizontal line, and, indeed, it falls inside the optimal interval. The would-be ``elbow'' in the scree-plot,  determined subjectively by the authors, is indicated by the grey (lower) horizontal line; it corresponds to retaining the top $6$ principal components of the data matrix. This rule attains strictly sub-optimal SE: roughly $31.3$, as opposed by $26.4$ attained by ScreeNOT.

\begin{figure}[h]
    \centering
    \includegraphics[width=0.55\textwidth]{{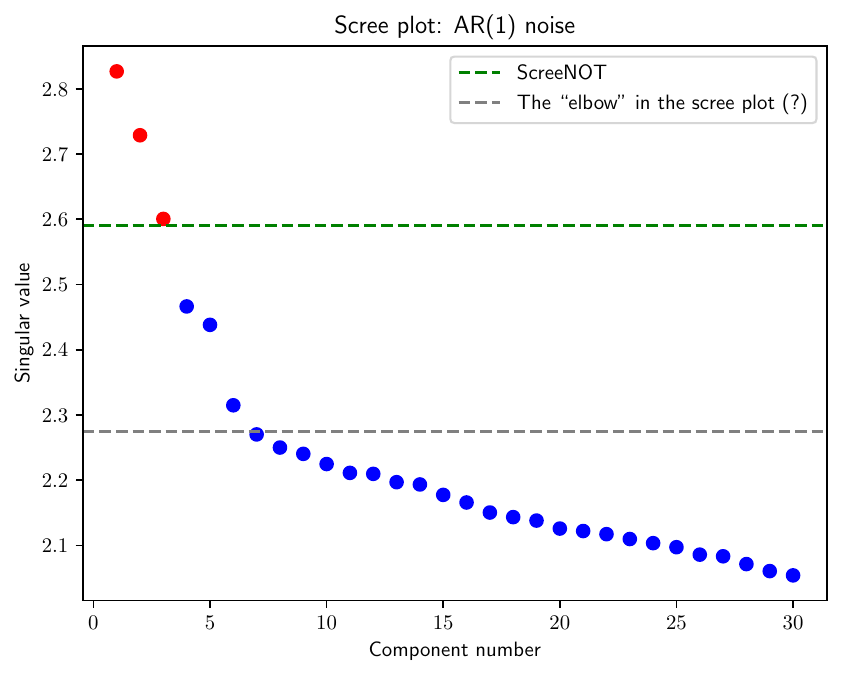}}
    \caption{\small
        Scree Plot for the stylized example of Section \ref{subsec:example}. Horizontal axis: singular value index, where the singular values $y_i$ of the data matrix $Y$ are sorted in decreasing order. Vertical axis: singular values $y_i$. Dashed green (upper) line: the optimal threshold calculated by \NAME. (Color online.)
    }
    \label{fig:IntroFig1}
    \end{figure}

\subsection{Internals of the Procedure}
\label{ssec-Internals}
\noindent We briefly describe the computational task performed by \NAME. 

\begin{itemize}
    
\item[{\bf Step 1.}] Sort the singular values in non-increasing order: $y_1\ge \ldots \ge y_p$. 

\item[{\bf Step 2.}] Compute the ``pseudo singular values'':
\[
\tilde{y}_{i} = y_{k+1} + \frac{1-\left( \frac{i-1}{k} \right)^{2/3}}{2^{2/3}-1}\left(y_{k+1}-y_{2k+1}\right) \quad\textrm{ for }i=1,\ldots,k \,,
\]
and set $\tilde{y}_i=y_i$ for $i=k+1,\ldots,p$.\footnote{We assume $2k+1<p$. Our proposed estimator is expected to perform  poorly when $k$ is large compared to $p$.}

\item[{\bf Step 3.}] Define the four scalar functions $\varphi,\tilde{\varphi},\varphi',\tilde{\varphi}'$ by
\begin{alignat*}{2}
    &\varphi(y) = \frac1p \sum_{i=1}^n \frac{y}{y^2-\tilde{y}_i^2}\,,\quad &&\varphi'(y) = -\frac1p \sum_{i=1}^n \frac{y^2+\tilde{y}_i^2}{(y^2-\tilde{y}_i^2)^2} \,, 
\end{alignat*}
and
\begin{alignat*}{2}
    &\tilde{\varphi}(y) = \gamma\varphi(y) + \frac{1-\gamma}{y}\,,\quad &&\tilde{\varphi}'(y) = \gamma\varphi'(y) - \frac{1-\gamma}{y^2} \,.
\end{alignat*}
 Now define 
\[
\TCritBare(y) = y \cdot \left( \frac{\varphi'(y)}{\varphi(y)} + \frac{\tilde{\varphi}'(y)}{\tilde{\varphi}(y)} \right) \,.
\]

\item[{\bf Step 4.}]
Assuming that $\tilde{y}_1,\ldots,\tilde{y}_p$ are not all zero, the function $y\mapsto \TCritBare(y)$ can be shown to be continuous and strictly increasing for $y>\tilde{y}_1$. Moreover, 
$\lim_{y\searrow \tilde{y}_1}\TCritBare=-\infty$ and $\TCritBare(\infty)=-2$. The computed hard threshold
 is the unique value $\hat{\thresh}$  satisfying 
\begin{eqnarray} \label{eq:master}
\TCritBare(\hat{\thresh})=-4 \,.
\end{eqnarray}
This equation is then solved numerically, for example by binary search.\footnote{We remark that the computational cost of this search is not considerable. E.g., binary search requires a number of iterations which is only {logarithmic} in the required precision. }

\item[{\bf Step 5.}] The algorithm returns the value $\hat{\thresh}$.
\end{itemize}

Evidently, the procedure as stated costs $O(n \log(n))$ flops; the dominant cost is sorting the singular values; ordinarily of course,
sorting is performed anyway as part of a standard SVD.
In that situation, the additional computational effort is $O(n)$,
which is unimportant compared to the cost of the underlying SVD.

\subsection{How the procedure works on the stylized application}

Figure~\ref{fig:IntroFig2and3}(a) shows a plot of $\TCritBare(\thresh)$ as a function of $\thresh$.
The horizontal blue line indicates the desired level $-4$. The vertical green
line indicates the crossing point, $\hat{\thresh}$, which is the value
returned by \NAME.

Figure~\ref{fig:IntroFig2and3}(b) shows a plot of the loss $\mathrm{SE}[\thresh|X]$ versus  $\thresh$. The red horizontal line shows
the optimum achievable loss. The green vertical
line shows the threshold selected by the 
procedure. It intersects the loss curve within the optimal level
and the achieved loss is therefore optimal.

\begin{figure}[H]
     \centering
     \begin{subfigure}[b]{0.45\textwidth}
         \centering
         \includegraphics[width=\textwidth]{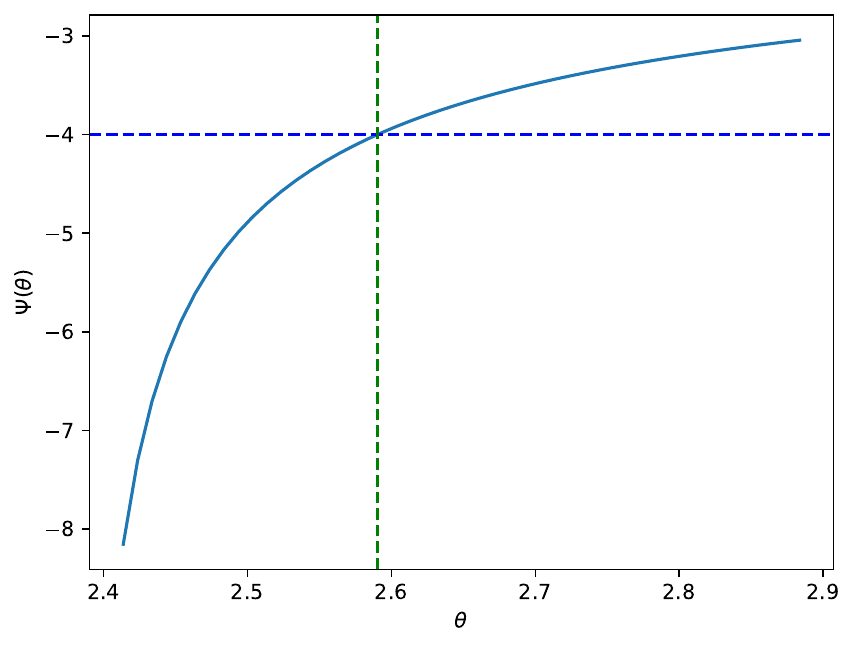}
         \caption{Solving the master equation \eqref{eq:master}}
         \label{fig:IntroFig2}
     \end{subfigure}
     ~
     \begin{subfigure}[b]{0.45\textwidth}
         \centering
         \includegraphics[width=\textwidth]{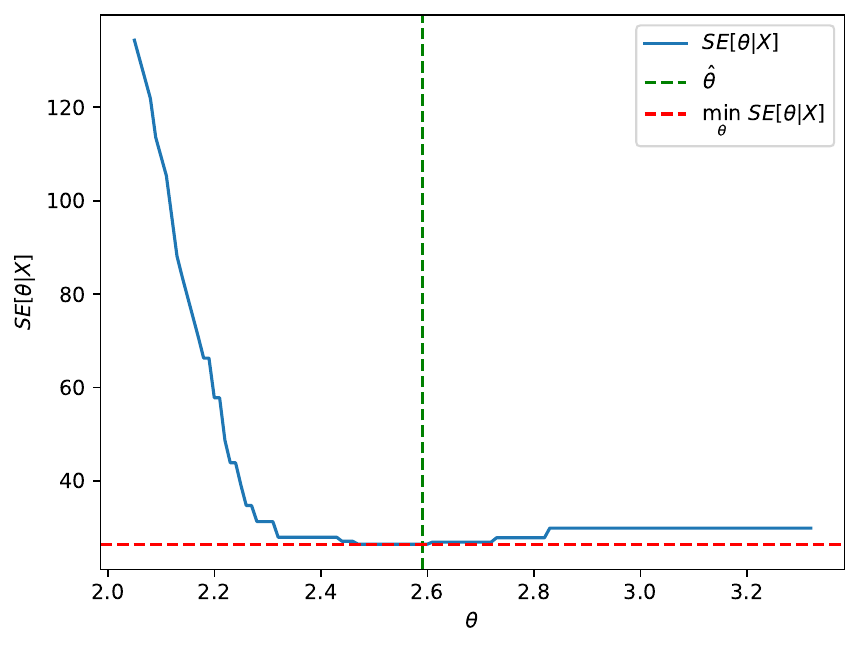}
         \caption{MSE loss $SE[\theta|X]$ as a function of $\theta$}
         \label{fig:IntroFig3}
     \end{subfigure}
     \hfill
        \caption{\small 
        Calculation of the optimal threshold $\hat{\theta}$ on the stylized example of Section \ref{subsec:example}.
        {\it (a) Left panel:}
    Horizontal axis: candidate thresholds $\theta$. Vertical axis: The function $\Psi(\theta)$. The \NAME algorithm solves 
    the equation $\Psi(\theta) = -4$ for $\theta$. Vertical green line shows the solution, denoted by $\hat{\theta}$. This is the value returned by \NAME.
    {\it (b) Right panel:}
    The MSE loss function $SE[\theta|X]$ for the stylized example of Section
    \ref{subsec:example}, plotted over candidate thresholds $\theta$. There is
    an interval of values $\theta$ all achieving the lowest possible loss; The
    threshold $\hat{\theta}$ returned by \NAME (shown  by the vertical green line) is located inside this optimal interval. (Color online.)
        }
        \label{fig:IntroFig2and3}
\end{figure}

\section{Setup and background from random matrix theory}
\label{sec:setup}

The rest of this paper is dedicated to formal analysis of the \NAME algorithm. To that end, we now define a precise signal+noise model and set up the necessary notation.

\paragraph{An asymptotic model for low-rank matrices observed in additive noise}
To recap, let $X_n$ be an unknown $n$-by-$p$ matrix, to be estimated. We observe a noisy measurement of $X_n$, $Y_n = X_n + Z_n$, where $Z_n$ is a noise matrix, which is statistically independent of $X_n$.
Our analysis employs
an asymptotic framework originating in Random Matrix Theory, and considers a sequence of such problems
$n,p\to \infty$, with the following generative assumptions.
\begin{enumerate}
    
    \item {\bf Limiting shape}: the dimensions $n,p$ tend to infinity together at a fixed ratio $p/n\to\gamma$. More concretely, fix $\gamma\in (0,1]$ and set $p=p_n=\lceil \gamma n \rceil$. Denoting $\gamma_n=p_n/n$, of course, $\gamma \le \gamma_n < \gamma + \frac1n$ and $\gamma_n \to \gamma$ as $n\to\infty$.
    
    \item {\bf Fixed signal rank and singular values}: The matrix $X_n$ has fixed rank $r=\rank(X_n)$ and fixed singular values. Specifically, let $r$ be constant, and fix $r$ positive and distinct numbers $x_1>\ldots>x_r>0$.
    $X_n$ is the matrix 
    \[
    X_n = \sum_{i=1}^r x_i \V{a}_{i,n} \V{b}_{i,n}^\T \,,
    \]
    where $\V{a}_{i,n}\in \R^n$ (resp. $\V{b}_{i,n}\in \R^p$) for $i=1,\ldots,r$ are sequences of left (resp. right) singular vectors of $X_n$, obeying a generative assumption as described next. We let $\V{x}=(x_1,\ldots,x_r)$ denote the  vector of singular values and we refer to either the matrix $X$ or just $\V{x}$ as the {\em signal}\,\footnote{
    	The assumption that the $x_i$-s are all distinct is standard in the literature on singular value shrinkage in the spiked model. When there are multiplicities, the SVD of $X_n$ is 
    not uniquely defined
    and, consequently, some framing constructs
    we use in this paper become inapplicable.
    % We don't have space to discuss all the
    The reframings adapted to such degenerate situations
    are beyond our scope.
    At any rate, the distinctness condition is generic in the space of matrices.	
    }.
    % We refer to the vector of singular values $\V{x}=(x_1,\ldots,x_r)$ as the {\em signal}. 
    
    \item {\bf Incoherent signal singular vectors.} The vectors  $\V{a}_{1,n},\ldots,\V{a}_{r,n}$ (resp.  $\V{b}_{i,n}$) 
    constitute a random, uniformly distributed orthonormal $r$-frame in $\R^n$ (resp. in $\R^p$).\footnote{In other words, $\V{a}_{1,n},\ldots,\V{a}_{r,n}$ are sampled from the $O(n)$-invariant distribution on the Stiefel manifold $V_{r}(\R^n)$. Equivalently, one can assume that $\V{a}_{i,n}$ and $\V{b}_{i,n}$ are any arbitrary sequences of orthonormal $r$-frames, and the distribution of $Z_n$ is invariant to multiplication by $O(n)$ to the left and by $O(p)$ to the right.} 
    
    \item {\bf Compactly supported, limiting bulk distribution of noise singular values.} Each matrix $Z_n$ is statistically independent of $X_n$. Let $z_{1,n},\ldots,z_{p,n}$ denote its singular values, with empirical CDF $\FZn$,  $\FZn(z) = p^{-1}\sum_{i=1}^p \Ind{z_{i,n}\le z}$. There is a {\it  limiting empirical CDF} (LECDF) $\FZ$ such that $\FZn \to \FZ$ a.s. at continuity points.

Moreover, we assume that $\FZ$ is compactly supported\footnote{
%	This will seem strange 
This assumption might seem unnatural
	to many statisticians when they first encounter it; but note that
     if $Z_n$ is a standard Gaussian white noise, then even though the distribution of matrix entries
     is not compactly supported,  the limiting bulk distribution of singular values is compactly supported, 
     in $[(1-\sqrt{\gamma}),(1+\sqrt{\gamma})]$. 
     }
     and denote the upper edge of the support (sometimes called the \emph{noise bulk edge}) by 
    \[
   \bulkEdge\equiv \bulkEdge(\FZ) = \sup \left\{ z\,:\,\FZ(z)<1 \right\} \,.
    \]
    We also assume that $\FZ$ is nontrivial ($d\FZ$ is not a single atom at $z=0$), in other words, $\bulkEdge(\FZ)>0$. 
    Note that neither the distribution $\FZ$ nor its bulk edge $\bulkEdge(\FZ)$ are assumed to be known to the statistician.
    
    \item {\bf No outliers straying from the bulk.} Asymptotically, no singular values of $Z_n$ can be found above the bulk edge:
    \[
    z_{1,n} = \|Z_n\| \aslim \bulkEdge(\FZ) \,.
    \]
    
    \item \label{assum:dense}
    {\bf Thickness of the bulk edge.} The following condition holds:
    \[
    \lim_{y\to \bulkEdge(\FZ)} \int (y-z)^{-2}d\FZ(z) = \infty \,,
    \]
    where the limit is taken from the right. That is, $\FZ$ puts ``sufficient'' mass near the upper edge of its support. 
  Under this condition, when the signal singular values $x_i$ are sufficiently small,
    the amount of ``information'' one can obtain about the corresponding singular vectors $\V{a}_{i,n}\V{b}_{i,n}^\T$ from the leading singular vectors of $Y_n$ also vanishes.   
    
    This assumption is by no means esoteric. For example, suppose that $\FZ$ has a continuous density $\fZ$ in a neighborhood of $z_+ = \bulkEdge(\FZ)$, where it behaves like $\fZ(z)\sim C(z-z_+)^{\alpha}$ as $z\to z_+$; here $\alpha>0$ is some exponent. Then this condition holds whenever $\alpha\le 1$. In Section~\ref{sec:correlated-noise}, we mention a broad class of noise matrices $Z_n$ for which this property holds with $\alpha=1/2$.
    
\end{enumerate}

\paragraph{Class of estimators and performance measure} Our goal is to estimate $X_n$. We consider the family of singular value hard-thresholding estimators: $\hat{X}_\thresh = \hat{X}_\thresh(Y_n)$, where
\begin{equation}
    \hat{X}_\thresh = \sum_{i=1}^p y_{i,n}\Ind{y_{i,n}>\thresh}\cdot \V{u}_{i,n}\V{v}_{i,n}^\T \,,
\end{equation}
where $Y_n = \sum_{i=1}^p y_{i,n}\V{u}_{i,n}\V{v}_{i,n}^\T$ is an SVD. 
To ensure that the vectors $\{\V{u}_{i,n},\V{v}_{i,n}\}$ are well-defined, even when there are multiplicities in the spectrum, the right singular vectors $\V{v}_{i,n}$ corresponding to a degenerate singular value of $Y_n$ are chosen to be a random (Haar distributed) orthonormal basis for the corresponding (right) singular subspace. (Accordingly, the corresponding $\V{u}_{i,n}$-s constitute a random orthonormal basis for the corresponding left singular subspace.)
We measure the error with respect to Frobenius norm (squared error)\footnote{Recall that for a matrix $A$, $\|A\|_F^2 = \sum_{i,j}|A_{i,j}|^2$.}, where we denote:
\begin{equation}
    \SEn[\V{x}|\thresh] = \left\|X_n - \hat{X}_\thresh(Y_n)\right\|^2_F \,.
\end{equation}
Our task is to choose $\thresh$, so as to make $\SEn[\V{x}|\thresh]$ as small as possible, in an appropriate sense (note that $\SEn[\V{x}|\thresh]$ is a random variable - we \emph{do not} take the expectation of $X_n$ and $Y_n$). The best possible performance is given by the {\it Oracle Loss} 
\begin{equation}
    \SEn^*[\V{x}] = \min_{\thresh\geq 0} \SEn[\V{x}|\thresh] \,,
\end{equation}
which is the best loss one can achieve over the family of singular value hard-threshold estimators, 
\emph{even knowing the true signal $X_n$}. 
Our goal in this paper is to develop a threshold selector that, ``typically for large $n$'', attains
the oracle loss $\SEn^*[\V{x}]$.
Note that the oracle loss $\SEn^*[\V{x}]$ is also a random variable, and it is not a priori clear how to estimate it.  
An important observation is that the (random) function $\thresh\mapsto \SEn[\V{x}|\thresh]$ is piecewise constant, with finitely many jumps (specifically, these are at the singular values of $Y_n$: $y_{1,n},\ldots,y_{p,n}$). In particular, the minimum of $\SEn[\V{x}|\thresh]$ is attained not strictly at a point, but on an \emph{interval} (or a union of intervals).

\subsection{Background from random matrix theory}

\paragraph{The Spiked Model} 
Our perspective on the matrix denoising problem extends the one proposed by Perry \cite{Perry2009} and
Shabalin and Nobel \cite{Shabalin2013}. In the model they proposed, which was
inspired by Johnstone's Spiked Covariance model \cite{Johnstone2001}, one works under the same model $Y_n=X_n+Z_n$ as described above, but specifically assumes that the
noise matrix $Z_n$ is column-normalized and white, namely, that its entries are
properly scaled $i.i.d$ random variables.  This model's close sibling,
the Spiked Model
for high-dimensional covariance, has been extensively studied in the probability
and statistics literature, to such an extent that we cannot point to all of the
existing literature here. Seminal works such as
\cite{Bai2008,Baik2006,Paul2007} and others have shown that the randomness in
the Spiked Model can be neatly described in terms of the so-called
BBP phase transition,
similar to the one discovered in \cite{Baik2005}; and of the displacement of the
sample eigenvalues relative to the population eigenvalues; and of the rotation
of the sample eigenvectors relative to the populations eigenvectors. 

In the matrix denoising setup we consider here, the model described by 
our assumptions above has been studied in \cite{benaych2012singular},
and the same three underlying phenomena were identified and quantified:
\begin{enumerate}
    \item {\bf BBP phase transition:} 
    Let $z_+ = \bulkEdge(\FZ)$ denote the noise bulk edge.
    There is a functional $\BBP(F_Z,\gamma)$ that depends on
        the LECDF $\FZ$ and the asymptotic shape 
        $\gamma$ that defines an important threshold phenomenon
        in the behavior of limiting empirical singular values.
      Setting $\bbp=\BBP(F_Z,\gamma)$,  then for any $i=1,\ldots,r$ where  $x_i \le x_+$,
        \begin{eqnarray} \label{bbp:eq}
            y_{i,n} \aslim z_+, \,\,\qquad n\to\infty.
        \end{eqnarray}
In short, {\it sufficiently small signal singular values $x_i$ do not produce
outliers beyond the noise bulk edge}. 
As we are about to see, the situation for $x_i > x_+$ is
quite different.
The split between $x_i \gtrless x_+$ is sometimes called the Baik-Ben
Arous-P\'ech\'e (BBP) phase transition,
after the original example of this type \cite{Baik2005}.
    \item {\bf Limiting location of outlier singular values:}
        The limiting value of $y_{i,n}$ is {\it not} its underlying population 
        counterpart $x_i$.
        There is instead a functional $\Yc(x; \FZ,\gamma)$, depending on 
        $\FZ$ and $\gamma$, describing this limiting behavior.
        The function of $x$ obtained by fixing
        $\FZ$, and $\gamma$ \--
        $ \Yc(x) \equiv \Yc(x; \FZ,\gamma)$ \-- explains how the asymptotic limit varies 
        with theoretical singular value $x$. For any $i=1,\ldots,r$    where $x_i\ge x_+ \equiv \BBP$,
        \begin{eqnarray} \label{displacement:eq}
          y_{i,n} \aslim y_{i,\infty} = \Yc(x_i), \,\,\qquad n\to\infty.
        \end{eqnarray}
The function $x\mapsto \m{Y}(x)$ is strictly increasing and one-to-one between $[x_+,\infty)$ and $[z_+,\infty)$. 
    \item {\bf No limiting cross-correlation of non-corresponding principal subspaces: }
    For $i \neq j$, the empirical dyad $\V{u}_{n,i}\V{v}_{n,i}^\T$ ultimately decorrelates from  
        each of the non-corresponding population dyads $\V{a}_{n,j}\V{b}_{n,j}^\T$.
        For any $i,j=1,\ldots,r$ such that $i\ne j$,
        \begin{equation}\label{decoupling:eq}
            \langle \V{a}_{n,i}\,,\,\V{u}_{n,j}\rangle \cdot
        \langle \V{b}_{n,i}\,,\,\V{v}_{n,j}\rangle \aslim 0 , \,\,\qquad n\to\infty \,.
        \end{equation}
        
    \item {\bf Limiting cross-correlation of corresponding principal subspaces:}
        Suppose the signal singular values $(x_i)_{i=1}^r$ are distinct.
        The empirical dyad $\V{u}_{n,i}\V{v}_{n,i}^\T$ {\it does}  correlate with 
        its theoretical counterpart $\V{a}_{n,i}\V{b}_{n,i}^\T$, but not perfectly. 
         The limit is described by a functional $\Cc(x; \FZ,\gamma)$   depending on $x$, 
         $\FZ$ and $\gamma$. Fixing once again $\FZ$ and $\gamma$, we get 
        a function of $x$, $\Cc(x) \equiv \Cc(x; \FZ,\gamma)$,  such that,
        with $\bbp = \BBP(\FZ,\gamma)$,
        \begin{eqnarray} \label{rotation:eq}
        \langle \V{a}_{n,i}\,,\,\V{u}_{n,i}\rangle \cdot
        \langle \V{b}_{n,i}\,,\,\V{v}_{n,i}\rangle 
        \aslim
        \begin{cases}
        \Cc(x_i) & x_i> x_+ \\
        0 & x_i \le  x_+ 
        \end{cases} \,.
        \end{eqnarray}
\end{enumerate}

We now give formulas for $\BBP$ and the mappings $\m{Y}(\cdot)$ and $\m{C}(\cdot)$, as computed in \cite{benaych2012singular}. For a CDF $H$, let
\begin{equation}\label{eq:varphi}
\varphi(y;H) = \int \frac{y}{y^2-z^2}dH(z)\,,
\end{equation}
which defines a smooth function on $y>\bulkEdge(H)$. Its derivative is
\begin{equation}\label{eq:varphi_d}
\varphi'(y;H) = - \int \frac{y^2+z^2}{(y^2-z^2)^2}dH(z) \,.
\end{equation}
Also define
\begin{equation}\label{eq:varphi_tilde}
\tilde{\varphi}_{\gamma}(y;H) = \gamma \varphi(y;H) + \frac{(1-\gamma)}{y}\,,\quad \tilde{\varphi}'_\gamma(y;H) = \gamma\varphi'(y;H) - \frac{1-\gamma}{y^2} \,.
\end{equation}
Note that $\tilde{\varphi}_\gamma(y;H)$ is simply $\varphi(y;\tilde{H}_\gamma)$, where $\tilde{H}_\gamma(z) = \gamma H(z) + (1-\gamma)\Ind{z\ge 0}$. This so-called {\it companion 
CDF} $\tilde{H}_\gamma$ describes the same distribution of nonzero singular values as $H$,
diluted by `zero padding' and has the following interpretation: if $Z_n$ is a sequence of $n$-by-$p$ matrices with a limiting singular value distribution $H$, then $Z_n^\T$ has a limiting singular value distribution $\tilde{H}_\gamma$\footnote{Practitioners will recognize that computer software often offers two options for SVD outputs, a 'fat' output with zero padding and a 'thin' output with those superfluous zeros stripped away. If $H$ denotes the LECDF of the 'thin' output singular values,
then $\tilde{H}$ is the corresponding LECDF of the 'fat' outputs.}. Let
\begin{equation}
    \label{eq:D}
\begin{split}
    &\m{D}_\gamma(y;H) \equiv \varphi(y;H)\cdot \tilde{\varphi}_\gamma(y;H)\,, \\ &\m{D}'_\gamma(y;H) \equiv \varphi'(y;H)\cdot \tilde{\varphi}_\gamma(y;H) + \varphi(y;H)\cdot \tilde{\varphi}'_\gamma(y;H) \,.
\end{split}
\end{equation}
To ease the notation in coming paragraphs, we put for short $\m{D}_\gamma(y) = \m{D}_\gamma(y;\FZ,\gamma)$,
and similarly for $\varphi(y)$, $\tilde{\varphi}_\gamma(y)$.
Let $z_+ =\bulkEdge(\FZ)$ denote the bulk edge.
The BBP phase transition location  $x_+ = \BBP(\FZ,\gamma)$ is given by 
\begin{equation}\label{eq:BBP}
   x_+ = \lim_{y\to z_+} \left( \m{D}_\gamma(y) \right)^{-1/2} \,,
\end{equation}
equivalently, $1/x_+^2 = \lim_{y\to z_+} \m{D}_\gamma(y) $. 
It is easy to verify that $\varphi(y)$, $\tilde{\varphi}_\gamma(y)$ and $\m{D}_\gamma(y)$ are non-negative, strictly decreasing functions of $y > z_+$, each tending to $0$ as $y\to\infty$. Thus, $\m{D}_\gamma(\cdot)$ maps the interval $(z_+,\infty)$ bijectively into $(\bbp,0)$; denote by $\m{D}_\gamma^{-1}(\cdot) \equiv \m{D}_\gamma^{-1}(\cdot; \FZ) $ the inverse mapping.

We finally can give formulas for the 
fundamental phenomenological limits described earlier.
The limiting empirical signal singular value $y_{i,\infty} = \m{Y}(x_i) \equiv \m{Y}(x_i; \FZ,\gamma)$ obeys
\begin{equation}
    \label{eq:Yc}
    \m{Y}(x) = \m{D}^{-1}_\gamma\left( \frac{1}{x^2}\right) \,, \quad\textrm{ for } x>\bbp\,,
\end{equation}
equivalently, $\m{D}_\gamma(\m{Y}(x))=1/x^2$. 
The asymptotic cosine $\m{C}(x) \equiv \m{C}(x;\FZ,\gamma)$ is given by 
\begin{equation}
    \label{eq:Cc}
    \m{C}(x) = -\frac{2}{x^3}\cdot \frac{1}{\m{D}'_\gamma\left(\m{Y}(x)\right)} \,,\quad\textrm{ for }x>\bbp \,.
\end{equation}
One may readily verify that $\m{C}(x)\ge 0$ for all $x>x_+$.\footnote{
Note that in \cite{benaych2012singular}, Eq. (\ref{rotation:eq}) is only stated as $|\langle \V{a}_{n,i}\,,\,\V{u}_{n,i}\rangle \cdot
        \langle \V{b}_{n,i}\,,\,\V{v}_{n,i}\rangle|
        \aslim \m{C}(x_i)$ (assuming $x_i>x_+$), with the absolute value. One may readily verify that the limiting cross-correlation must, in fact, be non-negative: Start with
        \[
        y_{i,n} = \V{u}_{i,n}^\T Y_n \V{v}_{i,n} = \V{u}_{i,n}^\T X_n \V{v}_{i,n} + \V{u}_{i,n}^\T Z_n \V{v}_{i,n} \le \V{u}_{i,n}^\T X_n \V{v}_{i,n} + \|Z_n\| \,.
        \]
        By Eq. (\ref{decoupling:eq}), $\V{u}_{i,n}^\T X_n \V{v}_{i,n} \sim x_i \langle \V{a}_{n,i}\,,\,\V{u}_{n,i}\rangle \cdot
        \langle \V{b}_{n,i}\,,\,\V{v}_{n,i}\rangle$, while $\|Z_n\|\to z_+$, $y_{i,n}\to y_{i,\infty}\ge z_+$. The conclusion follows.
        } 

We sometimes adopt the implicit parameterization of $\m{C}(x)$ in terms of $y=\m{Y}(x)$:
\begin{equation}
    \label{eq:Cc_y}
    \m{C}(x) = -2 \cdot \frac{\left( \m{D}_\gamma(y) \right)^{3/2}}{\m{D}'_\gamma(y)}\,,\quad\textrm{ where }y=\m{Y}(x)\,\textrm{ and } x>\bbp \,.
\end{equation}

\paragraph{Existence of a BBP phase transition} 
Recall that $\bbp = \BBP(\FZ,\gamma)$ gives the threshold such that whenever $x_i \le x_+$, one \emph{does not} observe an outlier singular value away from the bulk of $Y$. Not all noise distributions display this phase transition phenomenon, i.e. they may not exhibit $\bbp>0$: indeed,
by Eq. (\ref{eq:BBP}), $\BBP>0$ if and only if $\lim_{y\to\bulkEdge(\FZ)}\m{D}_\gamma(y;\FZ)<\infty$, equivalently, $\lim_{y \to \bulkEdge(\FZ)} \int (y-z)^{-1} d\FZ(z)<\infty$. This condition entails that 
near its own bulk edge, $\FZ$ is not ``thick''. 
For example, when $\FZ$ has a density in a neighborhood of $\bulkedge=\bulkEdge(\FZ)$ that behaves as $\fZ(z)\sim C(z-\bulkedge)^\alpha$, this condition is satisfied whenever $\alpha>0$. For example, the family of noise distributions described in Section~\ref{sec:correlated-noise} is of this type (with $\alpha=1/2$); they all display a BBP phase transition. Moreover, Assumption~\ref{assum:dense} gives $\lim_{y\to\bulkedge} \m{D}'_\gamma(y;\FZ)=-\infty$. From Eq. (\ref{eq:Cc_y}), this means that if $\bbp \equiv\BBP(\FZ,\gamma)>0$, then $\m{C}(x)=0$ as $x\to\bbp$ from the right. Curiously, when $\bbp=0$, this does not have to be the case. For instance, when $d\FZ = \delta_1$ and $\gamma=1$, an easy computation shows $\bbp=0$ and $\m{C}(x)=\frac{y^3}{y(y^2+1)}$, where $y=\m{Y}(x)$ and $\bulkEdge(\FZ)=1$. We see that $\lim_{x\to\bbp}\m{C}(x)=1/2$: this means that an \emph{arbitrarily small} signal already creates a very strong bias in the direction of the principal singular vectors of $Y_n$.

\paragraph{Notation} Throughout the paper, we use the notation 
\[
y_{i,\infty} = \begin{cases}
\m{Y}(x_i)\quad&\textrm{ when }x_i > \BBP\,,\\
\bulkEdge(\FZ)\quad&\textrm{ when } x_i \le \BBP \,.
\end{cases}
\]
By the results of \cite{benaych2012singular}, the singular values of $Y_n$, $y_{1,n}\ge \ldots \ge y_{p,n}$, satisfy $y_{i,n} \aslim y_{i,\infty}$ for any \emph{fixed} index $i$ (for $i>r$ this is an easy consequence of the interlacing inequality for singular values).

\subsection{Noise matrices with correlated columns}
\label{sec:correlated-noise}

We conclude this section by mentioning an important family of noise matrices satisfying our assumptions, namely, noise matrices with independent rows, having cross-column correlations. We consider noise matrices of the form $Z_n = W_n S_n^{1/2}$, where $(W_n)$ and $(S_n)$ are sequences of matrices obeying:
\begin{itemize}
    \item $W_n$ is an $n$-by-$p$ matrix with i.i.d elements. Specifically, let $W$ denote a random variable with moments
    \[
    \E(W) = 0,\quad \E(W^2)=1,\quad \E(W^4)<\infty \,.
    \]
    The entries of $W_n$ are i.i.d, with law $(W_n)_{ij} \overset{d}{=} n^{-1/2}W$, that is, scaled to variance $1/n$. Finiteness of the fourth moment of $W$ is essential; see \cite{bai1998no}.
    
    \item $(S_n)$ is a sequence of non-random $p$-by-$p$ matrices. Let $\lambda_1(S_n)\ge \ldots \ge \lambda_p(S_n)$ be the eigenvalues of $S_n$, and denote by $\FSn(\lambda) = p^{-1}\sum_{i=1}\Ind{\lambda_i(S_n)\le \lambda}$ the empirical CDF of its eigenvalues. We assume that the sequence  $(\FSn)$ converges to a compactly supported 
    LECDF $\FS$. 
    Moreover, denoting the upper and lower edges of the support by 
    \[
    \UpperEdgeS(\FS) = \sup\{\lambda\,:\,\FS(\lambda)<1\}\,,\quad  \LowerEdgeS(\FS) = \inf\{\lambda\,:\,\FS(\lambda)>0\}\,,
    \]
    we assume that $\lambda_1(S_n)\to \UpperEdgeS(\FS)$ and $\lambda_p(S_n)\to \LowerEdgeS(\FS)$. 
\end{itemize}

We refer to a random matrix ensemble of the form above as a {\bf noise matrix with correlated columns}. They appear, for example, in the following scenario: We observe $n$ i.i.d $p$-dimensional samples $\V{y}_i = \V{x}_i + \V{z}_i$, where $\V{x}_i$ are instances of a signal vector, assumed to be supported in an $r$-dimensional subspace, and $\V{z}_i=S_n^{1/2}\V{w}_i$ is a vector of correlated noise, with covariance $\mathrm{Cov}(\V{w}_i)=S_n$. Let $Y_n$ be the $n$-by-$p$ matrix, whose rows are $n^{-1/2}\V{y}_i^\T$ (define $X_n$, $W_n$ and $Z_n$ similarly). Then $Y_n = X_n + Z_n = X_n + W_n S_n^{1/2}$, where  $\rank(X_n) \le r$, by assumption. For any estimator $\hat{X}=\hat{X}(Y_n)$, let $\hat{\V{x}}_1,\ldots,\hat{\V{x}}_n$ be the rows of the matrix $n\cdot \hat{X}$. Then the Frobenius loss is just the average $L^2$ loss in estimating the signal samples ${\V{x}}_i$ by the vectors $\hat{\V{x}}_i$: $\|X_n-\hat{X}(Y_n)\|_F^2=n^{-1}\sum_{i=1}^n \|\V{x}_i-\hat{\V{x}}_i\|_2^2$. 

Much is known about the singular values of $Z_n$:
\begin{enumerate}
    \item {\bf Limiting singular value distribution:} $\FZn$ converges weakly almost surely to a compactly supported law $\FZ$. This limiting law is defined in terms of its Stietljes transform\footnote{$m(y)$ is in fact the Stieltjes transform of the limiting eigenvalue distribution of $Z^\T Z$: \[m(y)=\int (z-y)^{-1}dF_{Z^\T Z}(z) = \int (z^2-y)^{-1}d\FZ(z) \,.\]}, $m(y)=\int (z^2-y)^{-1}d\FZ(z)$; $m(y)$ is the unique Stieltjes transform
%which satisfies
     satisfying 
%     the equation
    \[
    m(y) = \int \frac{1}{t \left(1-\gamma-\gamma y m(y)\right) - y} d\FS(t)\,,\quad\textrm{ for all } y\in \C\setminus \R \,.
    \]
    
    \item {\bf Extreme singular values:} The largest and smallest singular values of $Z_n$ converge almost surely to the upper and lower edges of the support of the limiting law\footnote{{Below, $\m{Z}_{-}(\FZn)$ denotes the smallest singular value of $Z_n$. Recall that we assume $p\le n$.}} $\FZ$:
    \[
    \bulkEdge(\FZn) \aslim \bulkEdge(\FZ)\,,\quad \m{Z}_{-}(\FZn)\aslim \m{Z}_{-}(\FZ) \,.
    \]
    
    \item {\bf Behavior at the edge of the bulk:} On $\R\setminus\{0\}$, the limiting law $\FZ$ is absolutely continuous with respect to Lebesgue measure. Denoting by $\fZ$ the corresponding density, we have {$\fZ(z)\sim C\cdot\sqrt{\left|z-\bulkEdge(\FZ)\right|}$ as $z\nearrow \bulkEdge(\FZ)$.} This is the same behavior as a Mar\v{c}enko-Pastur law, corresponding to $S_n=I$. 
    This edge behavior will motivate one of our strategies for estimating $\FZn$ from the observed singular values $\FYn$ ({\em imputation}, see Section~\ref{sec:main:algorithm}); this is an important step in the \NAME~algorithm.
    Also, note that in particular, the limiting noise CDF $\FZ$ satisfies Assumption~\ref{assum:dense}. 
    
    \item {\bf CLT for linear spectral statistics:} Denote
    \[
    \underline{y} = (1-\sqrt{\gamma})^2\cdot \LowerEdgeS(\FS)\,,\quad \overline{y} = (1+\sqrt{\gamma})^2\cdot \UpperEdgeS(\FS) \,.
    \]
    Note that $\underline{y}^{1/2} \le \mathcal{Z}_{-}(\FZ) \le \bulkEdge(\FZ) \le \overline{y}^{1/2}$. Let $g$ be analytic on an open domain in $\C$ containing the closed interval $[\underline{y},\overline{y}]$. Set
    \[
    \Phi_n[g] = \int g(z^2) \left( d\FZ-d\FZn \right)(z) \,,
    \]
    which is a random variable.\footnote{A random variable of the form $\int h(z)\FZn(z)=p^{-1}\sum_{i=1}^n h(z_{i,n})$ is called a linear spectral statistic.} Then the sequence $p\cdot \Phi_n[g]$ is tight. If, moreover, $\E(W^4)=3$, then $p\cdot \Phi_n[g]$ converges in law to a Gaussian random variable. 
\end{enumerate}
For properties (1) and (2), we refer to \cite{bai1998no} and the references therein (see also the book \cite{bai2010spectral}). Property (3) is proved in \cite{silverstein1995analysis}. Property (4) is proved in \cite{bai2004}.

\begin{figure}[h]
    \centering
    \includegraphics[width=0.32\textwidth]{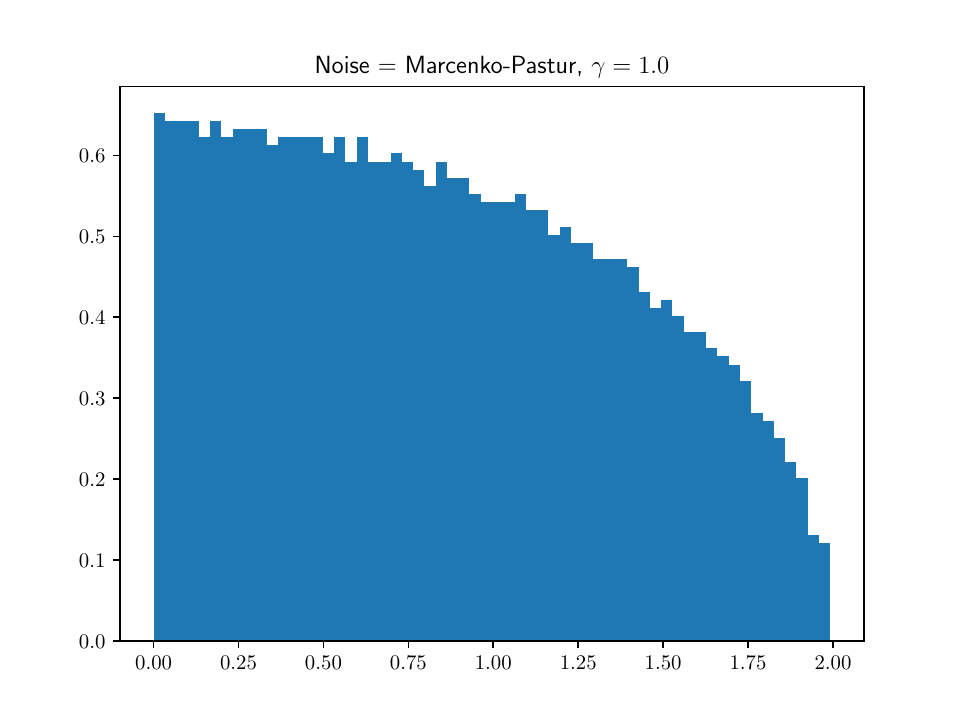}
    \includegraphics[width=0.32\textwidth]{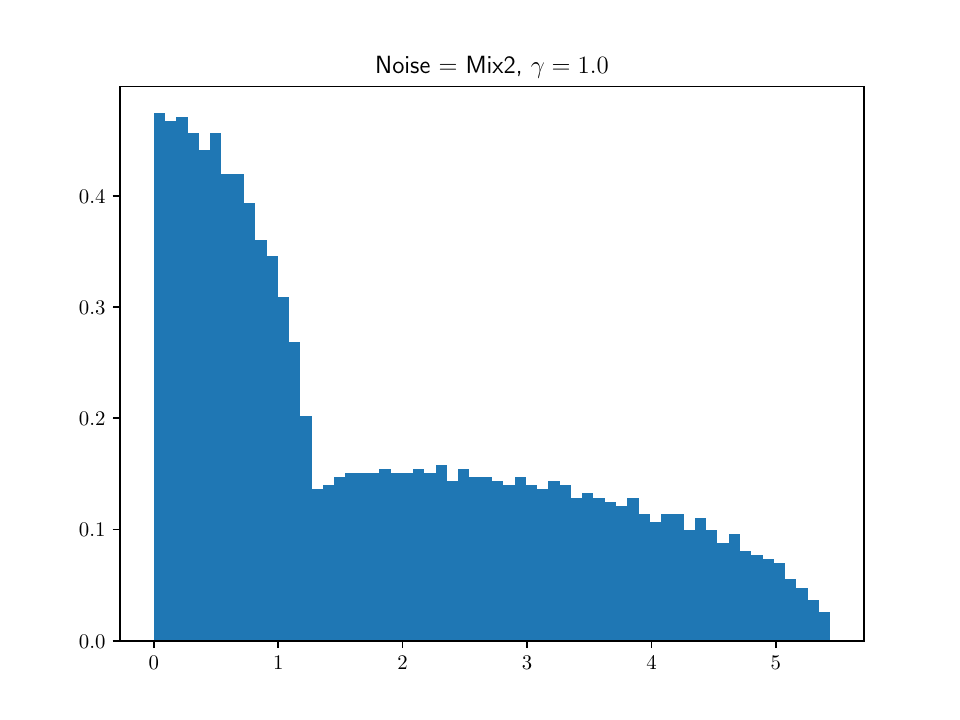}
    \includegraphics[width=0.32\textwidth]{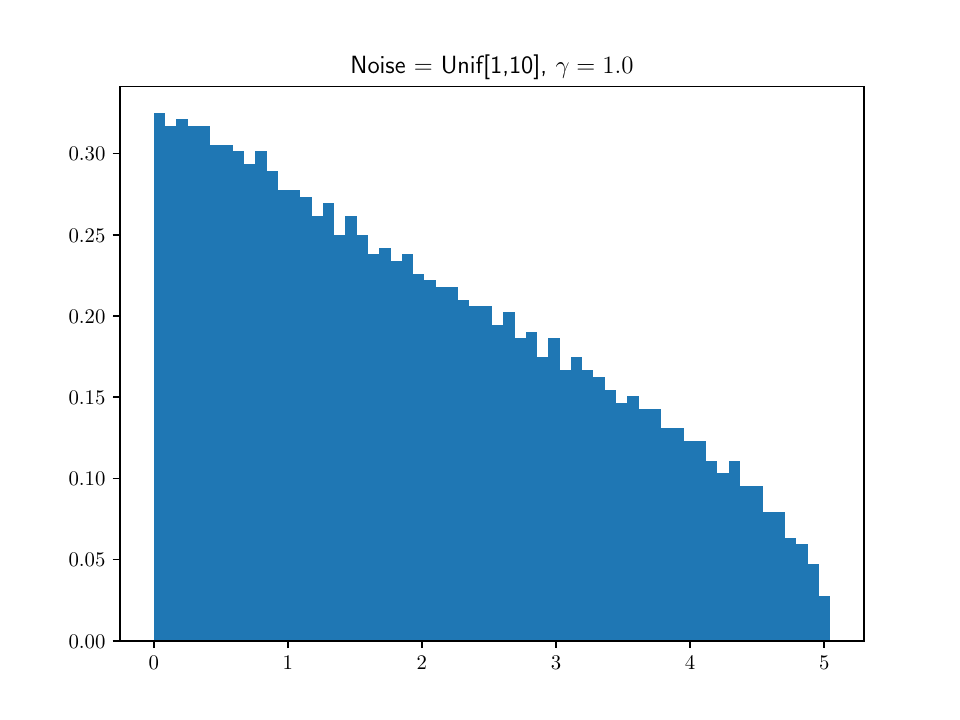}
    \newline
    \includegraphics[width=0.32\textwidth]{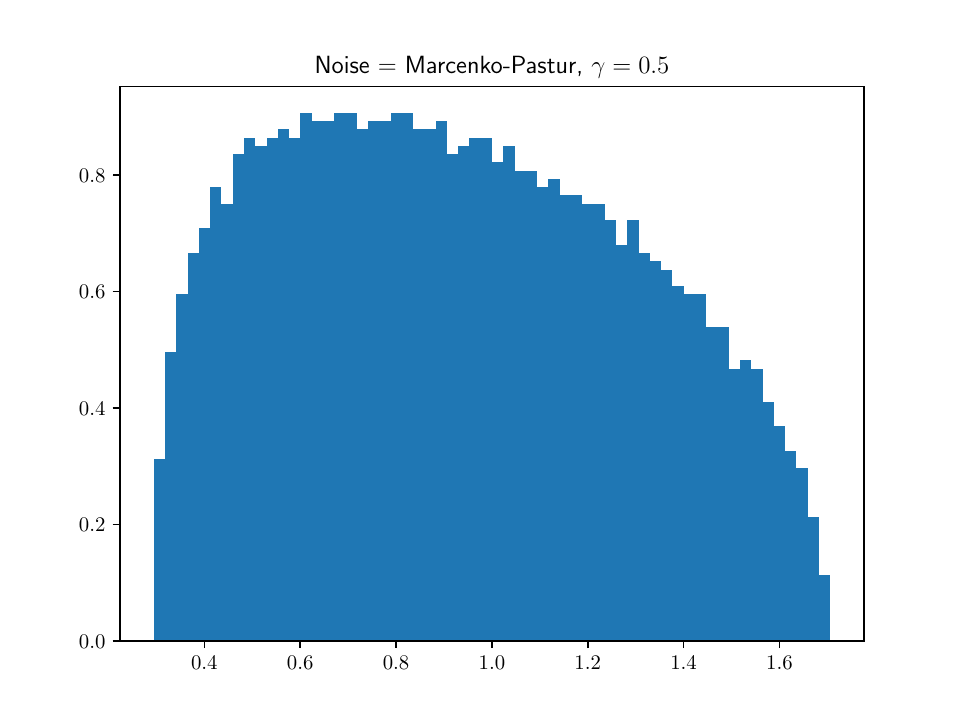}
    \includegraphics[width=0.32\textwidth]{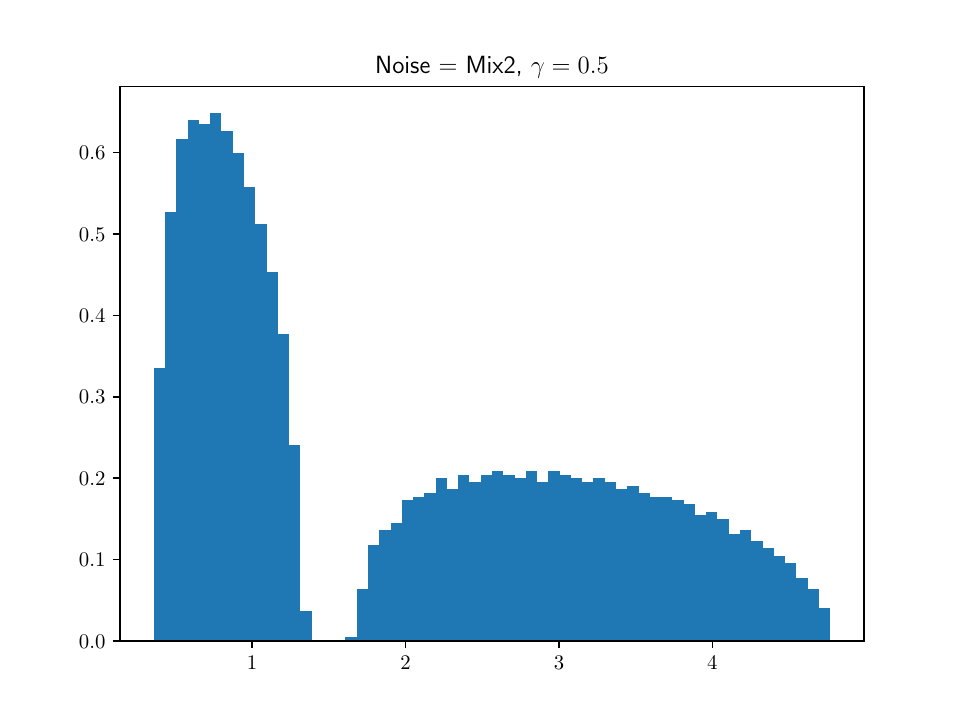}
    \includegraphics[width=0.32\textwidth]{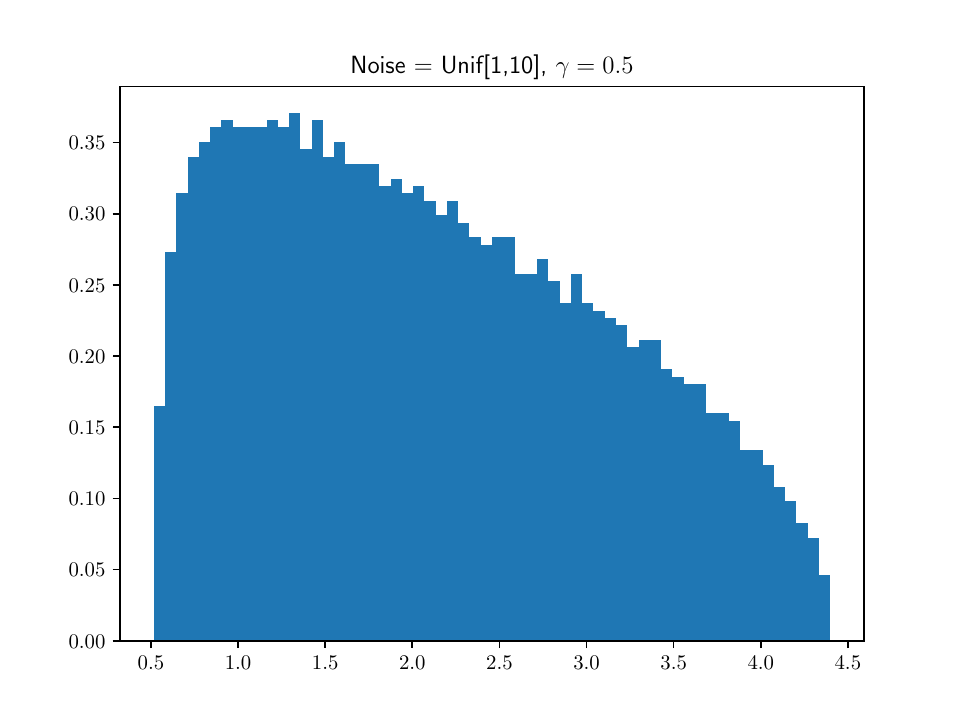}
    \caption{Several empirical noise singular value distributions that come from the model in Section~\ref{sec:correlated-noise}. 
    Left to right: covariance eigenvalue distribution: (i) $d\FS=\delta_1$ (giving a Mar\v{c}enko-Pastur bulk); (ii) An equal mix of two atoms, $d\FS = \frac12 \delta_1 + \frac12 \delta_{10}$; (iii) $\FS$ uniform on $[1,10]$. Top: shape $\gamma=1$; bottom: $\gamma=0.5$. Each plot is the histogram of singular values from a single random $n\times p$ matrix, with $p=3000$ and $n=p/\gamma$. 
    }
    \label{fig:Figure1}
\end{figure}

As a final remark, we mention that when the covariance matrix $S_n$ is invertible and known (or can be consistently estimated with respect to the operator norm), estimating $X_n$ using the leading singular vectors of $Y_n$ is sub-optimal. Instead, it is better to first ``whiten'' the noise, that is, compute $Y_n^w = Y_n S_n^{-1/2}=X_n S_n^{-1/2} + W_n$. Letting $\V{u}^w_{i,n}$ and $\V{v}^w_{i,n}$ be respectively the left and right singular vectors of $Y_n$, we can ``recolor'' the right singular vectors, $\V{v}^c_{i,n}=S_n^{1/2}\V{v}_{i,n}^w/\|S_n^{1/2}\V{v}_{i,n}^w\|$. Under a uniform prior on the signal singular vectors, as we assume in this paper, and when $W_n$ is i.i.d Gaussian, the correlations between the signal singular vectors and empirical singular vectors can be shown to be \emph{stronger} in the whiten-then-recolor scheme, see \cite{leeb2018optimal,hong2018asymptotic} for details:
\[
\lim_{n\to\infty}\langle \V{a}_{i,n}, \V{u}^w_{i,n} \rangle \langle \V{b}_{i,n},\V{v}_{i,n}^c \rangle \ge \lim_{n\to\infty} \langle \V{a}_{i,n}, \V{u}_{i,n} \rangle \langle \V{b}_{i,n},\V{v}_{i,n} \rangle \,.
\]

% \newpage
\section{Results}
\label{sec:results}

\paragraph{Outline} We start by developing a theory for optimal hard thresholding, under the assumption that the noise singular value distribution $\FZ$ is known. We show that there is an asymptotically uniquely admissible hard threshold $\optThresh(\FZ)$, which is given as a certain functional $T_\gamma$  of the asymptotic aspect ratio $\gamma$ and the limiting noise CDF $\FZ$. Relying on the fact that $\SEn[\V{x}|\theta]$ can only take a finite number of values as $\theta$ varies, we show that thresholding at $\optThresh(\FZ)$ has rather strong optimality properties: it in fact attains oracle loss, at finite $n$, with probability increasing to $1$ as $n\to\infty$. 
We then move on to the practical setting of interest, in which $\FZ$ is unknown. We propose a method for consistently estimating $\optThresh(\FZ)$ from the observed data $Y_n$. We do this by applying the optimal threshold functional $T_{p/n}(\cdot)$ on a judiciously transformed version of $\FYn$, the empirical singular value distribution of $Y_n$. The continuity of the functional with respect to the CDF and the shape parameter then implies that the resulting quantity is a consistent estimator for $\optThresh(\FZ)$; the optimality properties of the adaptive algorithm then follow from the previously developed theory.  
Unless otherwise stated, we always operate under assumptions (1)-(6) of Section~\ref{sec:setup}.

\subsection{A theory for optimal singular value thresholding}

Consider the  function $\thresh\mapsto \ASE[\V{x}|\thresh]$ defined for $\thresh > 0$ by
\begin{equation}\label{eq:ASEt}
	\ASE[\V{x}|\thresh] = \sum_{i=1}^r R(x_i|\thresh),\quad\textrm{where} \quad R(x_i|\thresh) = \Ind{y_{i,\infty}\le \thresh}\cdot R_0(x_i) + \Ind{y_{i,\infty}> \thresh}\cdot R_1(x_i) \,,
\end{equation}
and
\begin{equation}\label{eq:R0_R1}
	R_0(x)=x^2,\quad R_1(x)=
	\begin{cases}
	x^2 + \m{Y}(x)^2 -2x\m{Y}(x)\m{C}(x) \quad&\textrm{if }x>\bbp,\\
	x^2 + \bulkedge^2 \quad&\textrm{if }x\le \bbp.
	\end{cases} 
\end{equation} 
Recall that $y_{i,\infty}=\bulkedge$ $(\equiv\bulkEdge(\FZ))$ 
when $x_i\le \bbp $ $(= \BBP(\FZ,\gamma))$ and $y_{i,\infty}=\m{Y}(x_i)$ when $x>\bbp$.
Define also
\begin{equation}\label{eq:ASEopt}
	\ASE^*[\V{x}] = \sum_{i=1}^r R^*(x_i),\quad \textrm{where }\quad R^*(x_i)=
    \min\{R_0(x_i),R_1(x_i)\} .
% 	\begin{cases}
% 	R_0(x_i)\quad&\textrm{when }x_i< \bbp,\\
% 	\min\left\{ R_0(x_i),R_1(x_i) \right\}\quad&\textrm{when }x_i\ge \bbp.
% 	\end{cases}
\end{equation}
Clearly, $R^*(x)\le R(x|\theta)$ for any $x$ and $\thresh$, which means $\ASE^*[\V{x}]\le \ASE[\V{x}|\thresh]$. 

It is easy to verify that for almost every $\thresh>\bulkEdge(\FZ)$, the loss of thresholding at the fixed point $\thresh$ converges: $\lim_{n\to\infty}\SEn[\V{x}|\thresh]=\ASE[\V{x}|\thresh]$ almost surely; see Lemma~\ref{lem:as-limit} for a precise statement. We start by finding the threshold that attains minimum asymptotic loss. 

\begin{defn}[{\bf Optimal threshold functional}]
    For a compactly supported CDF $H$ and $\gamma\in (0,1]$, let
    \begin{equation}
        \label{eq:F}
       \TCrit(y;H) = y\cdot \frac{\m{D}_\gamma'(y;H)}{\m{D}_\gamma(y;H)} = y \cdot \left( \frac{\varphi'(y;H)}{\varphi(y;H)} + \frac{\tilde{\varphi}_\gamma'(y;H)}{\tilde{\varphi}_\gamma(y;H)} \right) \,;
    \end{equation}
    this is well-defined for $y>\bulkEdge(H)$. Define the functional of $H$
        \begin{equation}
        \label{eq:optFunc}
        \optThresh(H) = \inf \left\{ y\,:\,y>\bulkEdge(H)\textrm{ and }\TCrit(y;H)\ge -4 \right\}\,. 
    \end{equation}
    We call this the {\bf optimal threshold functional}. 
\end{defn}

\begin{lemma}\label{lem:main:F-props}
The following holds:
\begin{enumerate}
    \item For any $H$ and $\gamma$, $y\mapsto\TCrit(y;H)$ is negative and increasing, with $\TCritGamma(\infty)=-2$. 
    \item Assume that $H$ is compactly supported and satisfies $\lim_{y\to\bulkEdge(H)}\int (y-z)^{-2}dH(z)=\infty$  (note that, by assumption, $H=\FZ$ satisfies this). Then $\optThresh(H)$ is the unique number $>\bulkEdge(H)$ satisfying $\TCritGamma( \optThresh(H); H) = -4$. 
    \item Thresholding at $\thresh^*=\optThresh(\FZ)$  minimizes the asymptotic loss:
    \[
    \ASE[\V{x}|\thresh^*] = \min_{\theta}\ASE[\V{x}|\thresh] = \ASE^*[\V{x}] \,.
    \]
    Moreover, $\thresh^*$ is the unique threshold for which the above holds {\bf universally}, for all signals $\V{x}$.   
\end{enumerate}
\end{lemma}
Lemma~\ref{lem:main:F-props} is proved in Section~\ref{sec:proofs:fixed-threshold}.

Note that $\thresh\mapsto \ASE[\V{x}|\thresh]$ is piecewise constant, with jumps at $y_{1,\infty},\ldots,y_{r,\infty}$. This means that its minimum is actually attained on an \emph{interval}:
\begin{defn}[The asymptotic optimal interval]
	Let
	\begin{equation}
		\optIntervL = \max\left\{ y_{i,\infty} \,:\, y_{i,\infty}< \optThresh(\FZ)\right\},\quad 
		\optIntervU = \min\left\{ y_{i,\infty} \,:\, y_{i,\infty}> \optThresh(\FZ)\right\} \,.
	\end{equation}
	Note that since $\optThresh(\FZ)> \bulkedge = \bulkEdge(\FZ,\gamma) $ and $y_{r+1,\infty}=\bulkedge$, we always have, by definition, $\optIntervL\ge \bulkedge$. Moreover, if $y_{1,\infty}\le \optThresh(\FZ)$ then we define $\optIntervU=\infty$.  
\end{defn} 

\begin{lemma}\label{lem:main:opt-interval}
	\begin{enumerate}
		\item Throughout the interval $\thresh\in (\optIntervL,\optIntervU)$,
		$\ASE[\V{x}|\thresh]$ is constant.  Moreover, it attains its minimum there;
		if $\theta_0 \in (\optIntervL,\optIntervU)$, then
		\[
		\ASE[\V{x}| \theta_0]=\min_{\thresh\ge 0}\ASE[\V{x}|\thresh]=\ASE^*[\V{x}] \,.
		\]
		\item Any $\theta_1>\bulkEdge(\FZ)$ {\bf outside} $[\optIntervL,\optIntervU]$ has 
		\[
		\ASE[\V{x}|\theta_1] > \ASE^*[\V{x}] \,.
		\]
		\item {\bf Unique asymptotic admissibility:} $\optThresh(\FZ)$ is in the interior of the asymptotic optimal interval. In fact, it is the only threshold which has optimal asymptotic loss simultaneously for all signals $\V{x}$:
		\[
		\bigcap_{\V{x}\textrm{ signal}} (\optIntervL,\optIntervU) = \{\optThresh(\FZ)\} \,.
		\]
	\end{enumerate}
\end{lemma}
Lemma~\ref{lem:main:opt-interval} is proved in Section~\ref{sec:proofs:fixed-threshold}.

\paragraph{The Scree Plot heuristic: a quantitatively interpretation} 
Under our signal model, one could think of a ``natural'' quantification of Cattell's Scree Plot heuristic. Roughly, it hopes to threshold the data 
singular values slightly above the pure-noise bulk edge. 
%The excess ASE this incurs, compared to the minimal attainable ASE is
If this hope is fulfilled, the excess ASE incurred,
compared to the minimal attainable ASE, is
\[
\lim_{\delta\to 0} \ASE[\V{x}|\bulkedge+\delta]-\ASE^*[\V{x}] = \sum_{i\,:\,y_{i,\infty} \in (\bulkedge,T_\gamma(\FZ)) } \left( R_1(x_i)-R_0(x_i) \right)\,.
\]
The excess ASE is proportional to the number of 
barely/moderately emergent
%``moderate'' 
signal singular values, namely, such that $y_{i,\infty}>\bulkedge$ (so that they can be observed as outliers in the spectrum of $Y_n$) but $y_{i,\infty}<T_\gamma(F_Z)$ (meaning that the corresponding empirical singular vectors are too ``noisy'' so to be useful in estimating $X_n$). 
%Clearly, for the worst signal configuration $\V{x}$, the excess ASE is proportional to the rank $r$.
Clearly, in the worst-case scenario, the signal $\V{x}$ consists entirely of barely emergent singular values, so that the excess ASE is proportional to $r=\rank(X_n)$. For a concrete example, consider $\V{x}$  that consists of $r$ distinct singular values, located {\it just slightly } above $\bbp$; in that case, the excess ASE is $r\cdot (R_1(\bbp)-R_0(\bbp))=r \cdot \bulkedge^2$.

It is clear at this point that thresholding at any point in the interior of the asymptotic optimal interval achieves the best asymptotic loss, among all other fixed hard thresholds. Our main result states that, remarkably, one {\bf cannot} come up with a consistently better thresholding strategy, even if given access to the true unknown signal $X_n$:

\begin{thm}\label{thm:lim-oracle-risk}
	\begin{enumerate}
		\item
		Almost surely,
		\[
		\lim_{n\to\infty} \SEn^*[\V{x}] = \ASE^*[\V{x}]\,.
		\]
		
		\item  Let $\thresh \in (\optIntervL,\optIntervU)$ be in the interior of the asymptotic optimal interval, and $\thresh_n$ be any sequence of thresholds (possibly depending on $Y_n$) such that  $\thresh_n \aslim \thresh$. Then
		\[
		\SEn[\V{x}|\thresh_n] \aslim \ASE^*[\V{x}] \,.
		\] 
	\end{enumerate}
	
\end{thm} 

Our next result states that thresholding inside the asymptotic optimal interval in fact 
achieves oracle risk with high probability, {\bf for finite $n$}:

\begin{thm}
	\label{thm:oracle-risk-attained}
	Suppose that $\optThresh(\FZ)\notin \left\{y_{1,\infty},\ldots,y_{r,\infty}\right\}$.\footnote{We need to exclude the case $\optThresh(\FZ)\in \left\{y_{1,\infty},\ldots,y_{r,\infty}\right\}$ for this reason: If $y_{i,\infty}=\optThresh(\FZ)$ for some $i$, then thresholding either slightly above or below $y_{i,n}$ (but still inside the asymptotic optimal interval) will achieve the same (optimal) asymptotic risk. However, we cannot deduce that for finite $n$, one of those options is, necessarily, consistently better than the other, thereby achieving oracle risk exactly.} 
	Then:
	\begin{enumerate}
		\item Let $\theta_0 \in (\optIntervL,\optIntervU)$ and $\theta_n$ be a sequence with $\theta_n \aslim \theta_0$. Then
		\[
		\prob \left\{ \exists N \textrm{ s.t. }\forall n\ge N\,:\, \SEn[\V{x}|\theta_n] = \SEn^*[\V{x}] \right\} = 1 \,.
		\] 
		\item Let $\theta_1 \notin [\optIntervL,\optIntervU]$ and $\theta_n\aslim \theta_1$.
		There exists $\delta > 0$, $\delta = \delta(\V{x};\FZ,\gamma)$ such that
		\[
		\prob \left\{ \exists N \textrm{ s.t. }\forall n\ge N\,:\, \SEn[\V{x}|\theta_n] > \SEn^*[\V{x}] + \delta \right\} = 1 \,.
		\]
	\end{enumerate}
	
\end{thm}
Theorems \ref{thm:lim-oracle-risk} and \ref{thm:oracle-risk-attained} are proved in Section~\ref{sec:oracle}.

\subsection{The \NAME algorithm}
\label{sec:main:algorithm}

In practice, the noise distribution $\FZ$ is generally unknown to the statistician. Theorems~\ref{thm:lim-oracle-risk} and \ref{thm:oracle-risk-attained}, along with the unique admissibility property of Lemma~\ref{lem:main:opt-interval}, tell us that our goal should be to estimate the optimal threshold $\optThresh(\FZ)$. 

We start by showing that the functional $(\gamma,H)\mapsto T_{\gamma}(H)$ is continuous with respect to weak convergence of CDFs, with the additional requirement that the edge of the support converges as well:

\begin{lemma}[Continuity of the optimal threshold functional]
\label{lem:main:continuity}
	Suppose that $H$ is compactly supported and satisfies the condition $\lim_{y\to\bulkEdge(H)} \int (y-z)^{-2}dH(z)=\infty$. Let $H_n$ be a sequence of CDFs such that 
	\begin{enumerate}
		\item $H_n$ converges weakly to $H$, denoted $H_n\dlim H$.
		\item $\bulkEdge(H_n)\to \bulkEdge(H)$. 
	\end{enumerate}
	Then $T_{p/n}(H_n)\to \optThresh(H)$.
\end{lemma}
The proof of
Lemma~\ref{lem:main:continuity} appears in the supplementary material, Section~\ref{sec:proof:estimating}.

Recall that the empirical singular value distribution of the noise matrix, $\FZn$, converges, by assumption, weakly almost surely to $\FZ$, with $\bulkEdge(\FZn)\aslim \bulkEdge(\FZ)$. The matrix noise $Z_n$, and consequently $\FZn$, is of course unknown to the statistician. However, since $Y_n$ is a rank-$r$ additive perturbation of $Z_n$, the interlacing inequalities for singular values imply for example the convergence of CDF's in Kolmogorov-Smirnov distance
$\|\FYn - \FZn \|_{KS} \to 0$ 
{(see, for example, the statement and proof of Lemma~\ref{lem:main:approx-FZ} below)}
% (by the same arguments used in
%  Lemma \ref{lem:main:approx-FZ} below)  
 and hence also in weak convergence. 
The obstacle preventing the would-be use of $T_{p/n}(\FYn)$ to estimate $\optThresh(\FZ)$ lies with the fact
 that $T$ is not continuous in
Kolmogorov-Smirnov metric convergence or other topologies involving CDF convergence such as weak convergence.
More concretely,  $T_{p/n}(\FYn)$ can be very different than $\optThresh(\FZ)$ because the top masspoints of $\FYn$ do not converge to the bulk edge $\FZ$. \footnote{Of course, this does not prevent convergence of ECDFs. Recall that $\FYn \dlim \FZ$ means that for {\bf bounded and continuous} functions $f$, $\int f(z)d\FYn(z) \to \int f(z)d\FZ(z)$.} Indeed, recall that $\bulkEdge(\FYn)=y_{1,n}\aslim y_{1,\infty}$, which is $>\bulkEdge(\FZ)$ when $x_1>\BBP$. 

To get a reasonable simulacrum of $\FZn$ built from knowledge only of $\FYn$
we perform ``surgery''  on $\FYn$,  ``amputating'' the top $k$ masspoints and
fitting a ``prosthesis'' to replace them. Post-surgery, we get
an estimate for the unknown empirical noise CDF $\FZn$.

As indicated in Section 2 above, the user of our proposed procedure
supplies an upper bound (which can be potentially very loose) $k\ge r$
on the rank of the unknown low-rank matrix. 

We could, in principle, propose any one of the following ``pseudo-noise'' CDFs, derived from $\FYn$:
\begin{itemize}
    \item {\bf Transport to zero:} We construct a CDF, $F_{n,k}^0$, obtained by removing the $k$ largest singular values of $Y_n$, and adding $k$ additional zeros. That is,
    \[
    F_{n,k}^0(y) = \frac1p \sum_{i=k+1}^p \Ind{y_{i,n}\le y} + \frac{k}{p}\Ind{y\ge 0} \,.
    \]
    \item {\bf ``Winsorization'' (clipping):} As in the previous construction, we remove the leading $k$ singular values. Instead of adding $k$ zeroes, we add $k$ copies of $y_{k+1,n}$. Equivalently, we ``clip'' the large singular values of $Y_n$ to be at most the size of $y_{n,k+1}$. That is, 
    \[
    F_{n,k}^w(y) = \frac1p \sum_{i=k+1}^p \Ind{y_{i,n}\le y} + \frac{k}{p}\Ind{y_{k+1,n}\le y} \,.
    \]
    \item {\bf ``Imputation'' (reconstruction of the missing upper tail):} After removing the top $k$ singular values of $Y_n$, we try to construct the noise tail in a principled way. Recall that when $Z_n$ is a noise matrix with correlated columns, as described in Section~\ref{sec:correlated-noise}, $\FZ$ has a density near $\bulkedge = \bulkEdge(\FZ)$ that behaves as $\fZ(z)\sim C(\bulkedge-z)^{1/2}$ as $z\to \bulkedge$ (\cite{silverstein1995analysis}). Using the heuristic\footnote{The exponent $\alpha=1/2$ was chosen as typical of bulk-edge distributions in random matrix theory. If there is reason to believe that the behavior at the bulk edge follows a different power law $f_Z(z)\sim (\bulkedge-z)^\alpha$, a correspondingly different exponent can be used instead.  }
    \[
    \frac{\ell-1}{p} \approx \int_{z_{\ell,n}}^{\bulkEdge(\FZ)} \fZ(z)dz \approx \int_{z_{\ell,n}}^{\bulkedge} C(\bulkedge-z)^{1/2}dz = C' (\bulkedge-z_{\ell,n})^{3/2} \,,
    \]
    we can estimate the distance between singular values in the upper tail as \[
    z_{\ell,n}-z_{t,n} \approx C'' \left[ \left( \frac{t-1}{p} \right)^{2/3} - \left( \frac{\ell-1}{p} \right)^{2/3} \right] \,.
    \]
    Taking $y_{\ell,n}\approx z_{\ell,n}$ for $\ell\ge r+1$, we propose to estimate the unknown constant as:
    \[
    C'' = \frac{y_{2k+1,n}-y_{k+1}}{(2k/p)^{2/3}-(k/p)^{2/3}}\,,
    \]
    assuming $2k+1<p$ (when $k$ is not very small compared to $p$, there is no reason to believe this heuristic should give good results). We ``reconstruct'' the missing upper tail as 
    \[
    \tilde{y}_{i,n} = y_{k+1,n} + C''\left[ \left(\frac{k}{p}\right)^{2/3}-\left(\frac{i-1}{p}\right)^{2/3} \right] = y_{k+1,n} + \frac{1-\left(\frac{i-1}{k}\right)^{2/3}}{2^{2/3}-1}\left( y_{2k+1,n}-y_{k+1,n} \right) \,.
    \]
    The CDF we use is then
    \[
    F^{i}_{n,k}(y) = \frac1p \sum_{i=k+1}^p \Ind{y_{i,n}\le y} + \frac1p \sum_{i=1}^{k} \Ind{\tilde{y}_{i,n} \le y} \,.
    \]
\end{itemize}
The label $i$ on $F^{i}_{n,k}$ stands for `imputation', a standard terminology in statistical practice
for filling in utterly missing data with plausible pseudo-data.
Numerical results in Section~\ref{sec:numerics} suggest that in many cases, the `imputation' method gives significantly 
better results than truncation or Winsorization at finite $n$. It is also more psychologically ``supportive'', which is why we 
recommended it to practitioners in Section 2 above. However, our formal results hold for all three methods. Importantly, 
the list of strategies above is by no means exhaustive. Indeed, any sequence of CDFs $F_n^\star$ that satisfies the conditions of Lemma~\ref{lem:main:continuity} may be used instead, yielding an asymptotically consistent estimate of the optimal threshold. Furthermore, our `imputation' procedure 
is not claimed to be optimal; possibly, other strategies will outperform those we suggest in finite problem sizes.

\begin{lemma}\label{lem:main:approx-FZ}
Suppose that $k=k_n$ satisfies $k_n\ge r$ and $k_n/p\to 0$ (in particular, $k$ can be any constant $\ge r$). Then for any choice $\star\in \{0,w,i\}$:
\begin{enumerate}
    \item Almost surely, $F_{n,k}^\star \dlim \FZ$.
    \item $\bulkEdge(F_{n,k}^\star) \aslim \bulkEdge(\FZ)$.
    \item We have the following bound on the Kolmogorov-Smirnov distance between $F_{n,k}^\star$ and $\FZn$: 
    \[
    \left\|F_{n,k}^\star-\FZ\right\|_{\mathrm{KS}} = \sup_{z} \left| F_{n,k}^\star(z)-\FZn(z) \right| \le \frac{k}{p} \,.
    \]
\end{enumerate}
\end{lemma}
The proof of Lemma~\ref{lem:main:approx-FZ} is deferred to the supplementary material, Section~\ref{sec:proof:estimating}.

The following theorem states the optimality properties of the proposed \NAME algorithm. It is an immediate corollary of Theorems~\ref{thm:lim-oracle-risk}, \ref{thm:oracle-risk-attained} and Lemma~\ref{lem:main:approx-FZ}:

\begin{thm}\label{thm:adaptive-guarantee}
Suppose that $k=k_n$ satisfies $k_n\ge r$ and $k_n/p\to 0$. For any $\star\in\{0,w,i\}$, $\hat{\theta}_n = T_{p/n}(F_{n,k}^\star)$ satisfies:
\begin{enumerate}
    \item $\hat{\theta}_n \aslim \optThresh(\FZ)$.
    \item $\SEn[\V{x}|\hat{\theta}_n] \aslim \ASE^*[\V{x}]$.
    \item Assume that $\optThresh(\FZ)\notin \left\{y_{1,\infty},\ldots,y_{r,\infty}\right\}$. Then 
    \[
    \prob \left\{ \exists N \textrm{ s.t. }\forall n\ge N\,:\, \SEn[\V{x}|\theta_n] = \SEn^*[\V{x}] \right\} = 1 \,.
    \]
\end{enumerate}
\end{thm}

Regarding the assumption in item 3 above, we note the following.
\begin{lemma}\label{lem:main:generic}
The condition $\optThresh(\FZ)\notin \left\{y_{1,\infty},\ldots,y_{r,\infty}\right\}$
is {\bf generic}, i.e., in the space of possible singular value $r$-vectors $\V{x}$,  this 
condition holds on an open dense set.
\end{lemma}

\begin{proof}
    Fix the noise bulk $\FZ$; then 
$\theta^*=T_\gamma(\FZ)$ is a constant
not varying as the underlying signal $\V{x}$ changes. Moreover,
it always strictly exceeds the bulk edge $\bulkEdge(\FZ)$. 
So $x^* = \m{Y}^{-1}(\theta^*; \FZ,\gamma)$ is a 
uniquely defined constant which exceeds $\BBP(\FZ,\gamma)$.
The set of vectors $\V{x}$ with all entries distinct from $x^*$ is open and dense. 
\end{proof}

\subsection{Stability of \NAME}

One wonders how fast $T_{p/n}(F_{n,k}^\star)$ converges to the limit $T_\gamma(\FZ)$. We show that for noise matrices with correlated columns, the model described in Section~\ref{sec:correlated-noise}, the typical deviations are of order $\m{O}(k/p)$. 
We start with a ``quantitative'' version of Lemma~\ref{lem:main:continuity}:

\begin{lemma}\label{lem:main:continuity-quantitive}
Adopt the setting of Lemma~\ref{lem:main:continuity}. 
Set
	\[
	\Delta_{1,n} = |\varphi(\optThresh(H);H)-\varphi(\optThresh(H);H_n)|,\quad \Delta_{2,n} = |\varphi'(\optThresh(H);H)-\varphi'(\optThresh(H);H_n)| \,,
	\]
	where $\varphi$ and $\varphi'$ are given in Eqs (\ref{eq:varphi}) and (\ref{eq:varphi_d}). Then 
	\[
	\left|\optThresh(H) - T_{p/n}(H_n)\right| = \m{O}\left( \Delta_{1,n} + \Delta_{2,n} + \left|\frac{p}{n} - \gamma\right| \right) \,.
	\]
\end{lemma}

Lemma~\ref{lem:main:continuity-quantitive}, along with the Kolmogorov-Smirnov distance bound from Lemma~\ref{lem:main:approx-FZ} and the tightness result for linear spectral statistics from \cite{bai2004} (see Section~\ref{sec:correlated-noise}), gives the following:
\begin{prop}\label{prop:main:perturb}
    Suppose that $(Z_n)$ is a  sequence of noise matrices 
    with correlated columns, as described in Section~\ref{sec:correlated-noise}; and let $\FS$ denote the LECDF of eigenvalues of the cross-column covariances $S_n$. Assume, in addition, that $\optThresh(\FZ) > \overline{y}^{1/2} = (1+\sqrt{\gamma})\cdot \sqrt{\UpperEdgeS(\FS)}$.\footnote{This additional assumption is used due to a technical requirement in the results of \cite{bai2004}. We suspect that it can be removed.}. Suppose that $k\ge r$ with $k/p\to 0$. Then for any $\star \in \{0,w,i\}$,
    \[
    \left|\optThresh(\FZ) - T_{p/n}(F_{n,k}^\star)\right| = \m{O}_{\prob}\left( \frac{k+1}{p} \right) \,.
    \]
\end{prop}
Lemma~\ref{lem:main:continuity-quantitive} and Proposition~\ref{prop:main:perturb} are proved in the supplementary material, Section~\ref{sec:proof:estimating}.

\section{Numerical experiments}
\label{sec:numerics}

The supplementary article 
% \cite{SI} 
contains comprehensive experiments conducted on a large variety of noise distributions. 
For space constraints, we include just a sample of these results - specifically, for white noise (Mar\v{c}enko-Pastur LECDF) with $\gamma = 0.5$. 
Simulation results and code reproducing all figures here and in the
supplementary article is 
% permanently 
available at 
\cite{SDR}. 
See Section \ref{app:numerics}  of the supplementary article
% \cite{SI} 
for full details on each experiment reported here.

In Figure~\ref{fig:A}, we plot the function $\thresh\mapsto \SEn[\V{x}|\thresh]$ for a single fixed problem instance. The vertical lines correspond to thresholds $\thresh$, taken to be either the true optimal threshold $\thresh=T_\gamma(\FZ)$, its estimated versions $\thresh=T_\gamma(F_{n,k}^\star)$, $\star\in \{0,w,i\}$, or the noise (asymptotic) bulk edge, $\thresh=\bulkEdge(\FZ)$, which is the ``natural'' implementation of Cattell's scree-plot heuristic in the spiked model. The error landscape $\SEn[\V{x}|\thresh]$ is seen to be a step function, and on this particular instance, all the proposed thresholding strategies fall inside the interval where it attains its global minimum; hence, they attain the oracle risk. In constrast, thresholding at the bulk edge results in a strictly suboptimal squared error.

Figures~\ref{fig:B} and~\ref{fig:C} compare the relative efficacy of the proposed threshold estimation strategies. In both experiments, thresholding at the exact optimal threshold $\thresh=T_\gamma(\FZ)$ (which is a priori unknown) yields the best results; among the proposed strategies, ``imputation'' $\thresh=T_\gamma(F_{n,k}^{i})$ appears to give the best results for finite problem dimensions. Figure~\ref{fig:D} demonstrates 
the superior finite-$n$ error of ``imputation'' in estimating the optimal threshold $T_\gamma(\FZ)$.
% the relative superiority of ``imputation'' in estimating the unknown optimal threshold $T_\gamma(\FZ)$, with regard to the finite-$n$ scaling of the approximation error.

\begin{figure}
    \centering
    
\begin{subfigure}{0.48\textwidth}
    \includegraphics[width=\textwidth]{{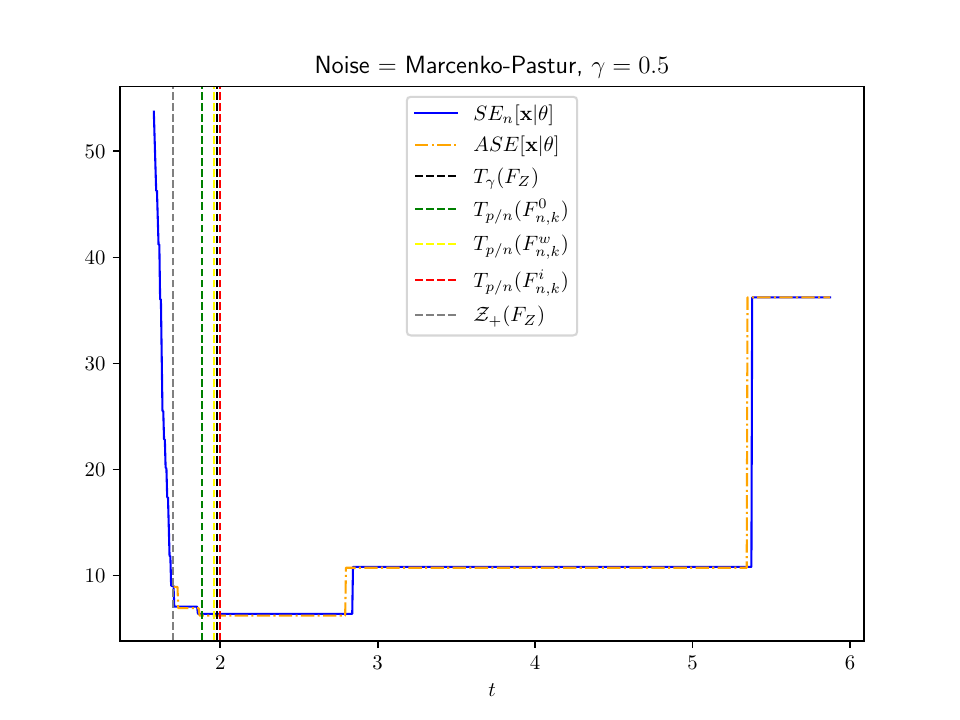}}
    \caption{\small
     A single problem instance, corresponding to the rank $r=5$ signal $\V{x}=(0.5,1.0,1.3,2.5,5.2)$. Shown are  the functions $\SEn[\V{x}|\theta]$ and $\ASE[\V{x}|\thresh]$ on top of each other.
     Here $\gamma=0.5$, $p=500$ and $n=p/\gamma$. We indicate the locations of
     $\bulkEdge$, $\optThresh(\FZ)$ and the estimates $T_{p/n}(F_{n,k}^\star)$ for $\star\in \{ 0,w,i\}$, with $k=4r=20$. 
    }
    \label{fig:A}
\end{subfigure}
% \hfill
~
~
\begin{subfigure}{0.48\textwidth}
    \includegraphics[width=\textwidth]{{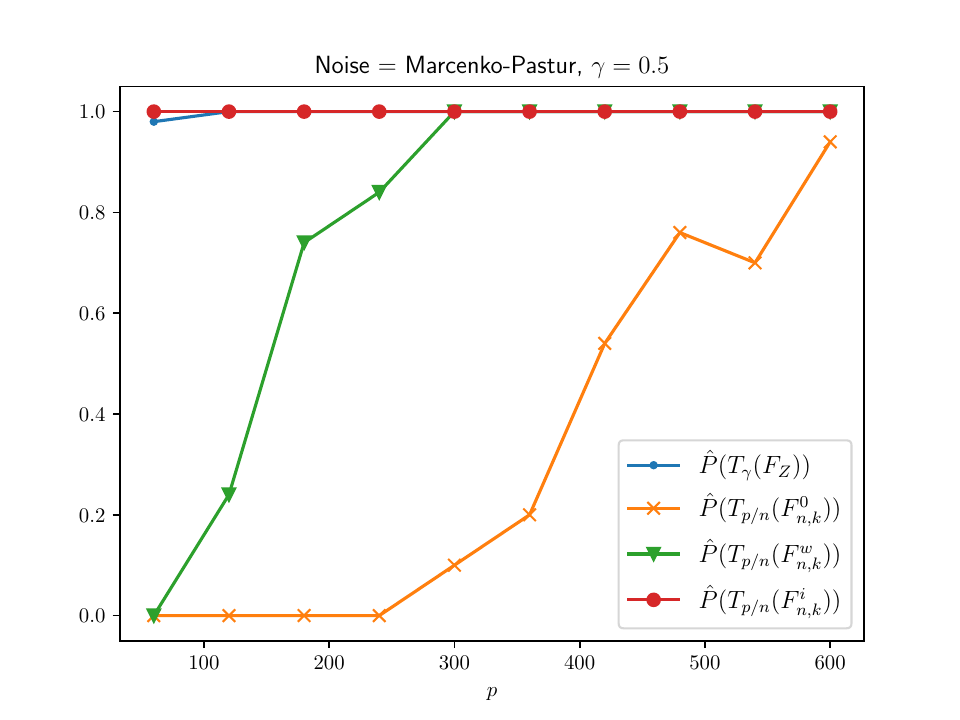}}
    \caption{\small For various choices of $p$, and $50$ denoising experiments, shown are  the fraction of experiments where each threshold $\in \left\{ T_{\gamma}(\FZ), T_{p/n}(F_{n,k}^0),T_{p/n}(F_{n,k}^w),T_{p/n}(F_{n,k}^i) \right\}$ attains oracle loss.
    Here, $\V{x}=(0.5,1.0, 1.3, 2.5, 5.2)$ so that $r=5$
    and  $k=4r=20$.}
    \label{fig:B}
\end{subfigure}
% \hfill

% \end{figure}

% \begin{figure}{H}\ContinuedFloat
\centering

\begin{subfigure}{0.48\textwidth}
    \includegraphics[width=\textwidth]{{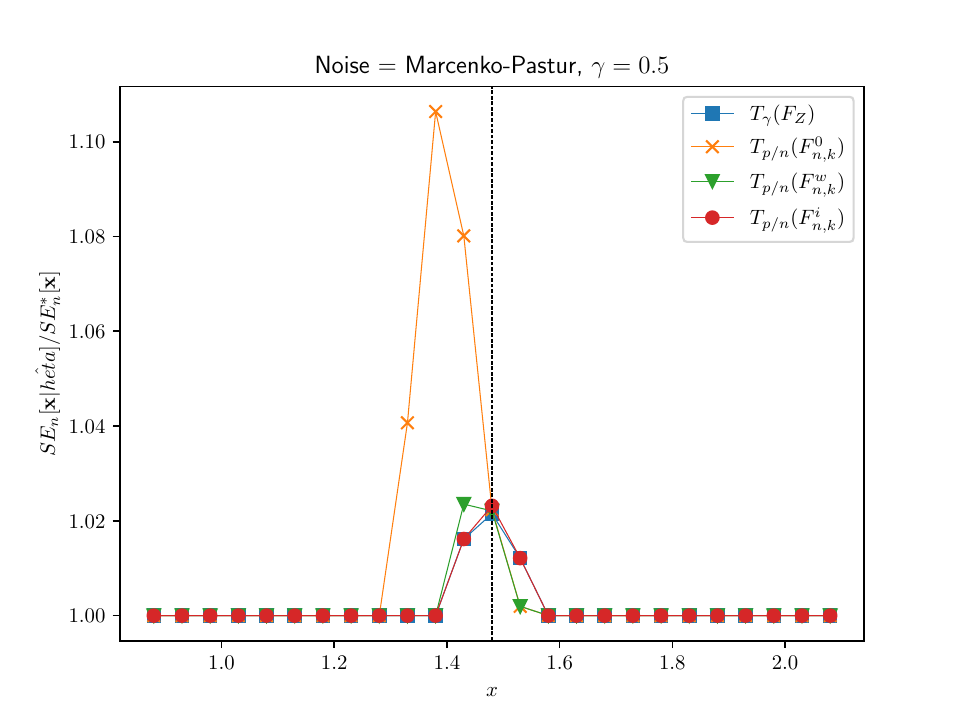}}
    \caption{\small Oracle loss $\SEn^*[\V{x}]$ compared with $\SEn[x|\hat{\theta}]$ for the choices $\hat{\thresh}\in \left\{ \optThresh(\FZ),T_{p/n}(F_{n,k}^0),T_{p/n}(F_{n,k}^w),T_{p/n}(F_{n,k}^i) \right\}$, for a  single spike.
    We let the spike intensity $x$ vary and plot $\SEn[\V{x}|\hat{t}]/\SEn^*[\V{x}]$ for each choice of estimator. 
 The ratios shown are averages across $20$ experiments.}
    \label{fig:C}
\end{subfigure}
% \hfill
~
~
\begin{subfigure}{0.48\textwidth}
    \includegraphics[width=\textwidth]{{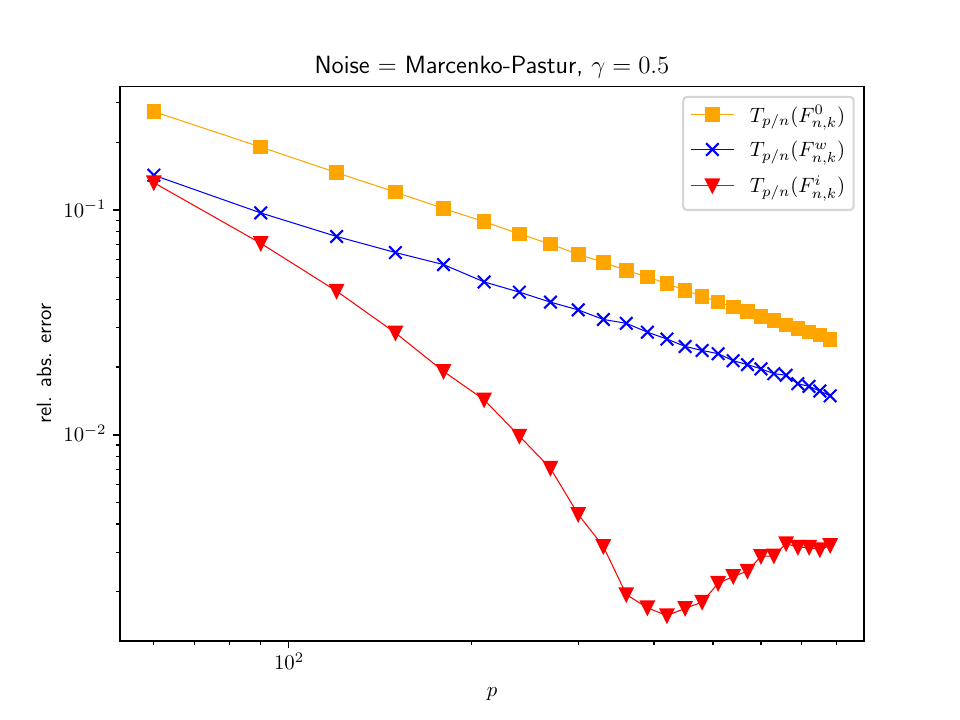}}
    \caption{\small Convergence rate of $T_{p/n}(F_{n,k}^\star)$ towards $\optThresh(\FZ)$.
    Here, $r=10$, $x=(1,\ldots,10)$, $k=20$.
    Shown is the relative absolute error $\left| T_{p/n}(F^\star_{n,k})-\optThresh(\FZ)\right|/\optThresh(\FZ)$, plotted in logarithmic scale, as $p$ increases and $n=p/\gamma$.
    % A logarithmic scale is used to observe the polynomial rate of decay in $p$. 
     Each point corresponds to the the average error across $50$ experiments.}
     \label{fig:D}
\end{subfigure}

    \caption{Monte-Carlo simulation results}
\end{figure}

\section{Proofs}
\label{sec:proofs}

\subsection{The asymptotic loss at a fixed threshold}
\label{sec:proofs:fixed-threshold}

To simplify notation, throughout this section $F_Z$ and $\gamma$ will be held fixed,
%and we will leave them 
and left
implicit in the notation where possible.
%so as to make the notation less cumbersome. 
% so to lighten notation.
In particular, 
we set $z_+ = \bulkEdge(F_Z)$, $x_+ = \BBP(F_Z,\gamma)$ throughout, and 
we suppress mention of $F_Z$ and $\gamma$ 
in entities like $\m{C}$, $\m{Y}$, $\m{D}$.

We start by investigating $\lim_{n\to\infty}\SEn[\V{x}|\thresh]$ for fixed $\thresh$. Note that when $\thresh<\bulkedge$, it is clear that $\lim_{n\to\infty} \SEn[\V{x}|\thresh]=\infty$; the reason being that, for small enough $\epsilon > 0$,
with probability $1$, $(1-\FZ(\thresh +\epsilon)) \cdot n  = \Omega(n)$ empirical singular values 
$y_{i,n}$ exceed the threshold $\thresh$, so that $\rank(\hat{X}_{\thresh}(Y_n))$ increases indefinitely. (This argument will be made more precise later.)

The following is an easy calculation:
\begin{lemma}\label{lem:as-limit}
	For any $\thresh> \bulkedge$, almost surely:
	\begin{enumerate}
		\item 
		\[
		\liminf_{n\to\infty} \SEn[\V{x}|\thresh] \ge \ASE^*[\V{x}] \,.
		\]
		\item If, in addition, $\thresh\notin \{y_{1,\infty},\ldots,y_{r,\infty}\}$, then
		\[
		\lim_{n\to\infty} \SEn[\V{x}|\thresh] = \ASE[\V{x}|\thresh] \,.
		\]
	\end{enumerate}
	The quantities $\ASE[\V{x}|\thresh]$ and $\ASE^*[\V{x}]$ appear in Eqs. (\ref{eq:ASEt}) and (\ref{eq:ASEopt}) respectively.
\end{lemma}
\begin{proof}
	Since $\thresh>\bulkedge$ and $y_{r+1,n}\aslim y_{r+1,\infty}=\bulkedge$, we see that with probability $1$, for large enough $n$, $\hat{X}_t(Y_n) = \sum_{i=1}^r y_{i,n}\Ind{y_{i,n}>\thresh}\cdot \V{u}_{i,n}\V{v}_{i,n}^\T$. Thus, for large enough $n$,
	\begin{align*}
		\SEn[\V{x}|\thresh] 
		&= \left\| \sum_{i=1}^r x_i\cdot \V{a}_{i,n}\V{b}_{i,n}^\T - \sum_{i=1}^r y_{i,n}\Ind{y_{i,n}>\thresh}\cdot \V{u}_{i,n}\V{v}_{i,n}^\T  \right\|_F^2 \\
		&= \sum_{i=1}^2 (x_i^2+y_{i,n}^2\Ind{y_{i,n}>\thresh}) +  \sum_{i=1}^r\sum_{j=1}^r x_iy_{i,n}\Ind{y_{i,n}>\thresh}\cdot \langle \V{a}_{i,n},\V{u}_{j,n}\rangle\langle \V{b}_{i,n},\V{v}_{j,n}\rangle\,.
	\end{align*}
	The lemma follows by recalling  (i) that $y_{i,n}\aslim y_{i,\infty}$ for all $i=1,\ldots,r$, where $y_{i,\infty}=\m{Y}(x_i)$ if $x_i>\bbp$ and $y_{i,\infty}=\bulkedge<\thresh$ whenever $x_i\le \bbp$, (ii) that   
	\[
	\langle \V{a}_{i,n},\V{u}_{j,n}\rangle\langle \V{b}_{i,n},\V{v}_{j,n}\rangle \to \begin{cases}
	\m{C}(x_i)\quad&\textrm{when }i= j \textrm{ and } x_i> \bbp  \\
	0\quad&\textrm{otherwise}
	\end{cases} \,\,
	\]
	and (iii) that $\Ind{y_{i,n}>\thresh}\aslim \Ind{y_{i,\infty}>\thresh}$ whenever $\thresh\ne y_{i,\infty}$. 
\end{proof}

Our goal for the moment is to characterize the minimum of $\ASE[\V{x}|\thresh]$ with respect to thresholds $\thresh$ strictly above the noise bulk edge, $\thresh>\bulkedge$. This will give us the optimal \emph{fixed} threshold, in the sense of minimal asymptotic loss (though, at this point, we cannot exclude the possiblity that thresholding precisely at $\thresh=\bulkedge$ might achieve better asymptotic risk).

Recall, by Eqs. (\ref{eq:ASEt}) and (\ref{eq:ASEopt}), that the asymptotic loss decouples across the signal spikes as 
\[
\ASE[\V{x}|\thresh] = \sum_{i=1}^r R(x_i|\thresh)\,,\quad \ASE^*[\V{x}] = \sum_{i=1}^r R^*(x_i)\,.
\]
Assuming that $\thresh>\bulkedge$, we have $R(x|\thresh)=R^*(x)=x^2$ when {$x\le \bbp$}, while for {$x>\bbp$},
\[
R(x|\thresh)=\Ind{\m{Y}(x)\le \thresh}\cdot R_0(x) + \Ind{\m{Y}(x)> \thresh}\cdot R_1(x)\,,\quad R^*(x) = \min\{R_0(x),R_1(x)\} \,,
\]
with 
\[
	R_0(x)=x^2,\quad R_1(x)=x^2 + \m{Y}(x)^2 -2x\m{Y}(x)\m{C}(x)\,.
\]
If we were able to find $\thresh>\bulkedge$ such that $R(x|\thresh)=R^*(x)$ for all $x>\bbp$, then, clearly, it achieves minimal asymptotic loss. To do that, it is convenient to introduce a re-parameterization $y=\m{Y}(x)$, where recall that $\m{Y}(\cdot)$ is an increasing bijection, mapping $(\bbp,\infty)$ to $(\bulkedge,\infty)$. Using Eqs. (\ref{eq:Yc}) and (\ref{eq:Cc}), assuming $x>\bbp$, we get
\[
x^2 = \left( \m{D}(y) \right)^{-1}\,,\quad \m{C}(x) = -2\frac{\left(\m{D}(y)\right)^{3/2}}{\m{D}'(y)} \,,
\]
so that 
\[
R_1(x) - R_0(x) = y^2 - 2xy\m{C}(x) = y^2 + 4y\cdot \frac{\m{D}(y)}{\m{D}'_\gamma(y)} = y^2\left( 1 + \frac{4}{\TCritGamma(y)} \right)\,,
\]
where 
\[
\TCritGamma(y) = y\cdot \frac{\m{D}'(y)}{\m{D}(y)}
\]
is as defined in Eq. (\ref{eq:F}). Since $\TCritGamma(\cdot)$ is negative ($\m{D}$ is positive and decreasing), we conclude that 
\begin{equation}\label{eq:RstarRt}
    R^*(x) = \Ind{\TCritGamma(y) \le -4}\cdot R_0(x) + \Ind{\TCritGamma(y)>-4}\cdot R_1(x) \,.
\end{equation}

The next lemma establishes some essential properties of $\TCritGamma(y)$:

\begin{lemma}\label{lem:crossing}
Let $H$ be a compactly supported CDF, with $\bulkEdge(H)>0$. Let $\gamma\in (0,1]$, and let $\TCritGamma(y;H)$ be defined as in Eq. (\ref{eq:F}). Then
\begin{enumerate}
    \item The function $y\mapsto\TCrit(y;H)$ is strictly increasing on $y\in (\bulkEdge(H),\infty)$, with $\lim_{y\to\infty}\TCritGamma(y;H)=-2$. 
	\item Assume that
	\[
	\lim_{y\to\bulkEdge(\FZ)} \int (y-z)^{-2} dH(z) = \infty \,.
	\]
	(This is Assumption~\ref{assum:dense} for $H=\FZ$). Then $\lim_{y\to\bulkEdge(H)}\TCritGamma(y;H)=-\infty$, and there is a unique point $y^* \in (\bulkEdge(\FZ),\infty)$ such that $\TCritGamma(y;H)=-4$.
\end{enumerate}
\end{lemma}
The proof of Lemma~\ref{lem:crossing} appears in 
Section \ref{sec:proof-lem:crossing}
of the supplementary article
% \cite{SI}
. An illustration of this Lemma and its consequences appears in Figure
\ref{fig:figure9} below.

\paragraph{Proof of Lemma~\ref{lem:main:F-props}} 
The lemma follows as a straightforward corollary of Lemma~\ref{lem:crossing}. Lemma~\ref{lem:crossing} implies that there is a unique number $\optThresh(\FZ)>\bulkEdge(\FZ)$ such that $\TCritGamma(\optThresh(\FZ);\FZ)=-4$.  Since $y\mapsto\TCrit(y;\FZ)$ is increasing, plugging into Eq. (\ref{eq:RstarRt}),
\begin{equation}
    R^*(x) = R(x|\optThresh(\FZ)) = \Ind{y \le \optThresh(\FZ)}\cdot R_0(x) + \Ind{y>\optThresh(\FZ)}\cdot R_1(x) \,,
\end{equation}
where $x>\BBP$ and $y=\m{Y}(x)$. We conclude that $\ASE^*[\V{x}]=\ASE[\V{x}|\optThresh(\FZ)]$. This is the minimum of $\ASE[\V{x}|\thresh]$ over all $\thresh\ge 0$ since, clearly, $\ASE[\V{x}|\thresh]\ge \ASE^*[\V{x}]$ by definition. Moreover, for any $\thresh\ne \optThresh(\FZ)$, we can find some $y>\bulkEdge(\FZ)$ such that either $\thresh < y < \optThresh(\FZ)$ or $\optThresh(\FZ)<y<\thresh$. Taking $x=\m{Y}^{-1}(x)$, we find that $R(x|\thresh)>R(x|\optThresh(\FZ))=R^*(x)$, since there is a {\bf unique} crossing point $x>\BBP$ with $R_0(x)=R_1(x)$ (because $y\mapsto\TCrit(y;\FZ)$ is strictly increasing). Thus, we can construct a signal $\V{x}$ for which $\ASE[\V{x}|\thresh]>\ASE^*[\V{x}]$, and therefore $\optThresh(\FZ)$ is the unique threshold which minimizes $\ASE[\V{x}|\thresh]$ universally for all $\V{x}$.  

\paragraph{Proof of Lemma~\ref{lem:main:opt-interval}}
Part (1) of Lemma~\ref{lem:main:opt-interval} follows from Lemma~\ref{lem:main:F-props}, along with the observation that if $\m{Y}(x)=\optThresh(\FZ)$, then $R_0(x)=R_1(x)=R^*(x)$; this means that regardless of whether we threshold slightly above or below $\m{Y}(x)$, we get the same asymptotic loss. Part (2) follows by the same argument as in the proof of Lemma~\ref{lem:main:F-props}, in the paragraph above. Finally, part (3) follows right from the definition of $\optIntervL$ and $\optIntervU$. 

\begin{figure}
    \centering
    \includegraphics[width=0.45\textwidth]{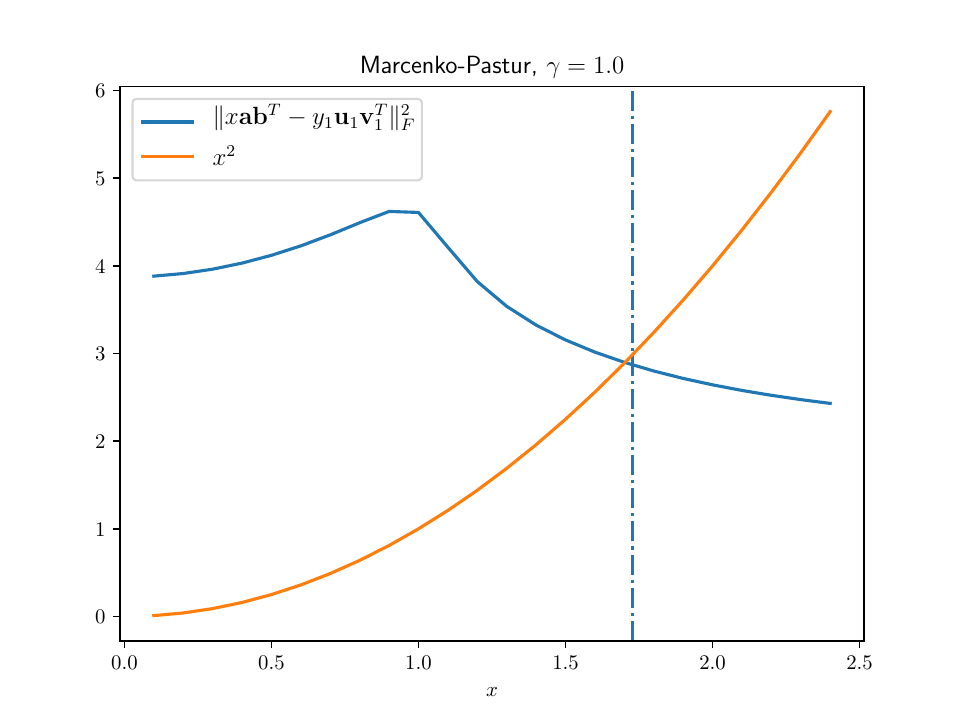}
    \includegraphics[width=0.45\textwidth]{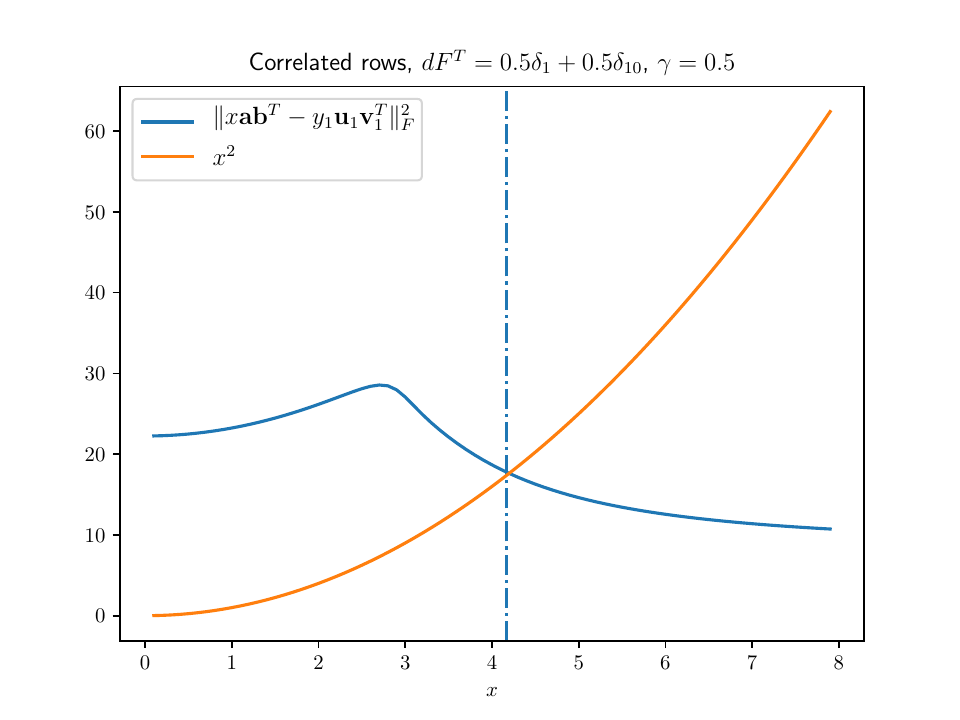}
    \includegraphics[width=0.45\textwidth]{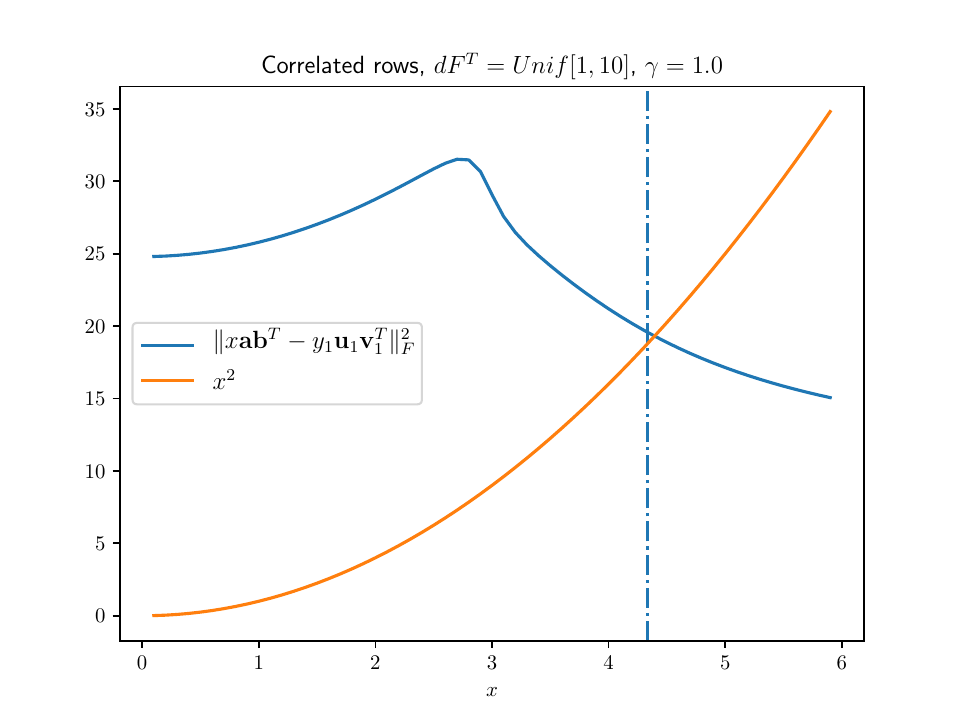}
    \includegraphics[width=0.45\textwidth]{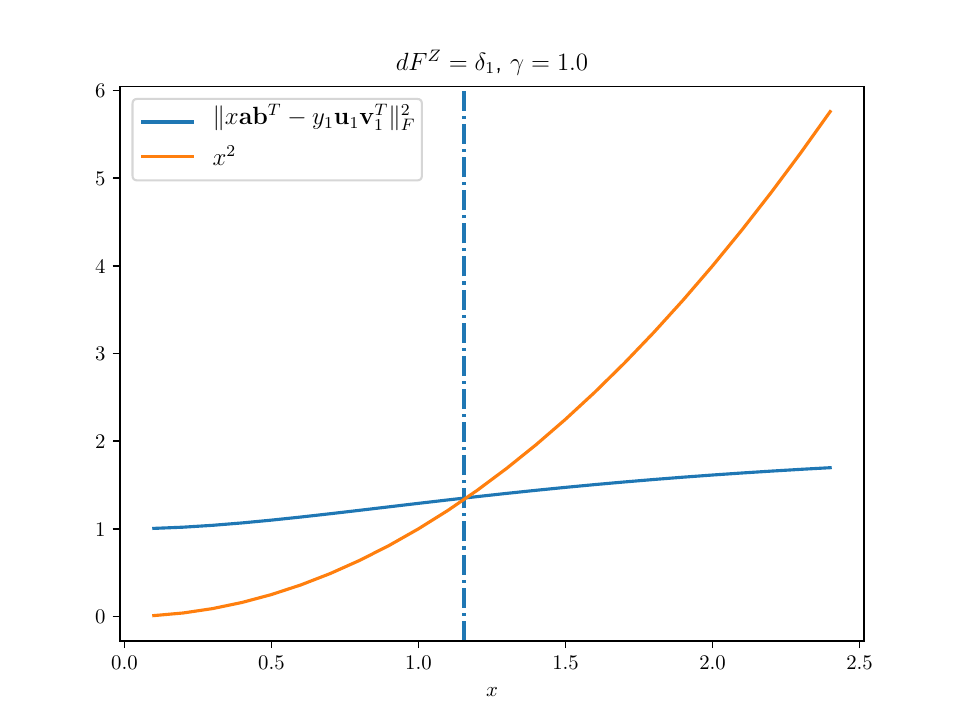}
    
    \caption{A numerical illustration of Lemma~\ref{lem:crossing} and its consequences. Assuming a rank-$1$ signal $X=x\cdot \V{a}_{1,n}\V{b}_{1,n}^\T$, set $R_0(x)=\|X\|_F^2=x^2$ and $R_1(x)=\lim_{n\to\infty}\|X-y_{1,n}\V{a}_{1,n}\V{b}_{1,n}^\T\|_F^2$. The point $x^*= \m{Y}^{-1}\left( \optThresh(\FZ) \right)$ is the unique crossing point $R_0(x^*)=R_1(x^*)$. When $x<x^*$, the principal components of $Y$ are ``too noisy'', so that estimating $\hat{X}=0$ gives better squared error; when $x>x^*$, the situation reverses. At each plot, $R_0(x)$ and $R_1(x)$ are plotted as $x$ varies, for finite $n$, fixed signal components $\V{a}_{1,n}\V{b}_{1,n}^\T$ and a single instance of $Z_n$. The dashed vertical line is an estimate of $x^*$, obtained by applying the functional $T_{p/n}(\cdot)$ on $\FZn$, as well as computing the inverse map $\m{Y}^{-1}(T_{p/n}(\FZn))$ numerically from $\FZn$. In all cases, $p=500$ and $n=p/\gamma$. 
    From left to right, top to bottom: (i) Mar\v{c}enko-Pastur law with shape $\gamma=1$; (ii) Noise matrix with correlated columns, with $\FS=\frac12\delta_1+\frac12\delta_{10}$ and $\gamma=0.5$; (iii) Likewise, with $\FS=\mathrm{Unif}[1,10]$; (iv) $\FZ=\delta_1$ and $\gamma=1$ (specifically, $Z_n=I$). 
    }
    \label{fig:figure9}
\end{figure}

\subsection{Achieving oracle loss}
\label{sec:oracle}

We move on to study the oracle loss $\SEn^*[\V{x}]$. This random variable depends on both 
$X_n$ and the noise $Z_n$.
Denote
\begin{equation}
	\hat{X}_{[k]} = \sum_{i=1}^k y_{n,i}\V{u}_{i,n}\V{v}_{i,n}^\T \,,\quad k=0,\ldots,p\,.
\end{equation}
That is, $\hat{X}_{[k]}$ is obtained from $Y$ by keeping only the top $k=0,\ldots,p$ singular values (equivalently, hard thresholding at $t=y_{k+1,n}$, in case one has $y_{k+1,n}<y_{k,n}$). Any hard thresholding estimator $\hat{X}_{\thresh}$ obviously corresponds to some $\hat{X}_{[k]}$ (however if there are multiplicities, possibly not every $\hat{X}_{[k]}$ is representable by some threshold $\thresh$);
% specifically, $y_{k,n}>\thresh\ge y_{k+1,n}$ (for notational consistency, set $y_{0,n}=\infty$). 
thus,
\[
\SEn^*[\V{x}] \ge \min_{0\le k \le p} \left\|X-\hat{X}_{[k]}\right\|_F^2 \,. 
\]

We first show that keeping too many singular values is consistently sub-optimal:
\begin{lemma}\label{lem:finite-rank}
	Set $M =  r + 1 + \left\lceil \frac{\ASE^*[\V{x}]}{\bulkedge^2} \right\rceil$. Then
	\[
	\prob \left\{  \exists N \textrm{ s.t. }\forall n\ge N\,:\,\SEn^*[\V{x}] < \min_{k\ge M} \|X-\hat{X}_{[k]}\|_F^2 \right\} = 1\,.
	\]
\end{lemma}
The proof of Lemma~\ref{lem:finite-rank} is the deferred to the supplementary article, Section~\ref{sec:proof:lem:finite-rank}.
%\begin{proof}
%	Recall that for any matrix $A\in \R^{n\times p}$ with SVD $A=\sum_{i=1}^p \sigma_i \V{u}_i\V{v}_i^\T$, its best rank-$r$ approximation with respect to Frobenius norm is obtained by taking its $r$ leading principal components. Since $X$ has rank $r$, for any $k\ge M$,
%	\[
%	\|X-\hat{X}_{[k]}\|_F^2 \ge \min_{\mathrm{rank}(B)=r}\|B-\hat{X}_{[k]}\|_F^2 = \sum_{i=r+1}^k y_{i,n}^2 \ge \sum_{i=r+1}^M y_{i,n}^2 \,.
%	\]
%	Recall that any fixed $i\ge r+1$ satisfies $y_{i,n}\aslim y_{i,\infty}=\bulkedge$. Since $M$ is constant, and satisfies $M > r + \ASE^*[\V{x}]/\bulkedge^2$, we obtain that 
%	\[
%	\sum_{i=r+1}^M y_{i,n}^2 \aslim (M-r)\bulkedge^2 > \ASE^*[\V{x}] \,.
%	\]
%	Since $\SEn^*[\V{x}]\le \SEn[\V{x}|\optThresh(\FZ)]\aslim \ASE^*[\V{x}]$, we conclude that almost surely, for all large enough $n$, $\SEn^*[\V{x}] < \min_{k\ge M}\|X-\hat{X}_{[k]}\|_F^2$. 
%\end{proof}
Lemma~\ref{lem:finite-rank} tells us that to study the oracle loss $\SEn^*[\V{x}]$ as $n\to\infty$, we only need, essentially, to study the risk of a \emph{fixed} collection of estimators, the number of whom does not depend on $n$; specifically, $X_{[k]}$ for $0\le k < M$. 
 obtain formulas for $\lim_{n\to\infty}\| X_n-\hat{X}_{[k]} \|_F^2$, we need to compute the limiting correlations between the underlying signal dyads, $\V{a}_{1,n}\V{b}_{1,n}^\T,\ldots,\V{a}_{r,n}\V{b}_{r,n}^\T$ and the 
corresponding empirical dyads $\V{u}_{i,n}\V{v}_{i,n}^\T$, for all $1\le i < M$.  
For empirical spikes up to $i=r$, these limiting correlations are computed in \cite{benaych2012singular} (recall Eq. (\ref{rotation:eq})).  

The next result shows that, as one would expect, the $j$-th singular vectors of $Y$, for any bounded $j\ge r+1$, are asymptotically uncorrelated with the signal singular vectors:

\begin{prop}\label{prop:correlations}
	For any $1\le i \le r$ and \emph{fixed} $j \ne i$ (not necessarily $j\le r$), one has
	\[
	\langle \V{a}_{n,i}\,,\,\V{u}_{n,j}\rangle \cdot
	\langle \V{b}_{n,i}\,,\,\V{v}_{n,j}\rangle \aslim 0 \,.
	\]
\end{prop}
%We provide a relatively short proof of Proposition~\ref{prop:correlations}, assuming that $r=1$.  
The proof of Proposititon~\ref{prop:correlations} is deferred to
Section~\ref{sec:proof-prop:correlations} of the supplementary article
% \cite{SI}
.
Note that Proposition~\ref{prop:correlations} implies that for any \emph{fixed} $M$ (meaning $M$ cannot depend on $n$), the event 
\[
\left\{ \forall 1\le i \le r,\,1\le j \le M,\,j\ne i,\quad :\quad \langle \V{a}_{n,i}\,,\,\V{u}_{n,j}\rangle \cdot
	\langle \V{b}_{n,i}\,,\,\V{v}_{n,j}\rangle \longrightarrow 0 \right\}
\]
holds with probability $1$.
The following Lemma is an immediate corollary:
% of Proposition~\ref{prop:correlations}:

\begin{lemma}\label{lem:more-svs-dont-help}
	For any $k\ge r$, 
	\[
	\left\|X-\widehat{X}_{[k]}\right\|_F^2 \aslim \sum_{i=1}^r \left[ x_i^2 + y_{i,\infty}^2 - 2x_i \cdot y_{i,\infty}\cdot \m{C}(x_i) \right] + (k-r)\bulkedge^2 \,,
	\]
	where, by way of notation, we use $\m{C}(x_i)=0$ for $x_i\le \bbp$.
	In particular,
	\[
	\prob \left\{ \exists N\textrm{ s.t. }\forall n\ge N\,:\,\SEn^*[\V{x}] < \min_{k\ge r+1} \|X-\hat{X}_{[k]}\|_F^2 \right\} = 1\,.
	\]
\end{lemma}
\begin{proof}
	The calculation is straightforward, as in the proof of Lemma~\ref{lem:as-limit}. For the last part, simply recall that 
	\[
	\SEn[\V{x}|\optThresh(\FZ)] \aslim \ASE^*[\V{x}] \le \sum_{i=1}^r \left[ x_i^2 + y_{i,\infty}^2 - 2x_i \cdot y_{i,\infty}\cdot \m{C}(x_i) \right] \,.
	\]
\end{proof}

We are ready to prove Theorems~\ref{thm:lim-oracle-risk} and \ref{thm:oracle-risk-attained}.

\paragraph{Proof of Theorem~\ref{thm:lim-oracle-risk}}
By Lemmas~\ref{lem:finite-rank} and \ref{lem:more-svs-dont-help}, almost surely, there exists $N$ such that $\forall n \ge N$, 
	\[
	\SEn^*[\V{x}] \ge \min_{0\le k\le r} \|X-\hat{X}_{[k]}\|_F^2 \,.
	\]
	Part (1) then follows from the observation that $\min_{0\le k\le r} \|X-\hat{X}_{[k]}\|_F^2 \aslim \ASE^*[\V{x}]$, as can be deduced from the calculations of Section~\ref{sec:proofs:fixed-threshold}, together with $\SEn[\V{x}|\optThresh(\FZ)] \aslim \ASE^*[\V{x}]$.
	We now prove (2). Let's assume, for ease of notation, that $T_\gamma(\FZ) \notin \{ y_{1,\infty},\ldots,y_{r,\infty} \}$. In that case, the asymptotic optimal interval is just the interval between two consecutive spikes, say, 
	\[
	\optIntervL = y_{k^*+1,\infty}\,,\quad \optIntervU = y_{k^*,\infty} \,,
	\]
	where $y_{0,\infty}=\infty$. Note that if $T_\gamma(\FZ)=y_{k,\infty}$ for some $k\ge 1$, then $\optIntervL = y_{k+1,\infty}\,,\quad \optIntervU = y_{k-1,\infty}$. With probability one, for large enough $n$, $X_{\theta_n}=\hat{X}_{[k^*]}$. Now, recall that $\left\| X_n- \hat{X}_{[k^*]}\right\|_F^2 \aslim \ASE^*[\V{x}]$.

\paragraph{Proof of Theorem~\ref{thm:oracle-risk-attained}} 
Let $k^*$ be as in the proof of Theorem~\ref{thm:lim-oracle-risk}. We know, from Lemmas~\ref{lem:finite-rank}, \ref{lem:more-svs-dont-help} and the definition of the asymptotic optimal interval, that $\|X-\hat{X}_{[k^*]}\|_F^2\aslim \ASE^*[\V{x}]$ and that almost surely, $\liminf_{n\to\infty} \min_{k\ne k^*}\|X-\hat{X}_{[k]}\|_F^2 > \ASE^*[\V{x}]$ (here is where we assume that there is no $y_{i,\infty}$ that equals $\optThresh(\FZ)$). Thus,
	\[
	\prob\left\{ \exists N \textrm{ s.t. } \forall n\ge N\,:\, \SEn^*[\V{x}]=\|X-\hat{X}_{[k^*]}\|_F^2 < \min_{k\ne k^*} \|X-\hat{X}_{[k]}\|_F^2 \right\} = 1\,.
	\] 
	The proof follows by noting that: (i) If $\theta\in (\optIntervL,\optIntervU)$, then with probability $1$, for all large enough $n$, $\hat{X}_{\theta_n} = \hat{X}_{[k^*]}$; (ii) If $\theta \notin [\optIntervL,\optIntervU]$ then with probability $1$, for large enough $n$, $\hat{X}_{\theta_n} \ne \hat{X}_{[k^*]}$.

\begin{acks}[Acknowledgments]
  We are grateful to the anonymous reviewers for their thoughtful comments, which have helped improve this manuscript considerably.

  \end{acks}

\begin{funding}
  DD was supported in part by NSF DMS 1407813, 1418362, and 1811614. 
This work was made possible by United States – Israel Binational Science Foundation (BSF) Grant 2016201 ``Frontiers of Matrix Recovery''.
%  and was partially supported by NSF DMS 1407813, 1418362, and 1811614.
ER was affiliated with the School of Computer Science and Engineering, the Hebrew University of Jerusalem, and supported in part by Israel Science Foundation grant no. 1523/16 and an Einstein-Kaye Fellowship from the Hebrew University of Jerusalem.
  \end{funding}

\bibliographystyle{imsart-number} % Style BST file (imsart-number.bst or imsart-nameyear.bst)
\bibliography{ref}       % Bibliography file (usually '*.bib')

% this needs \usepackage{xr,refcount} in sourcing doc's preamble

\makeatletter
\newcommand*{\storecounter}[2]{%
  \edef\@currentlabel{\csname the#1\endcsname}% Store current counter value in \@currentlabel
  \label{#2}% Store label
}
\makeatother

\storecounter{thm}{extthm}
\storecounter{cor}{extcor}
\storecounter{defn}{extdefn}
\storecounter{lemma}{extlemma}
\storecounter{prop}{extprop}
\storecounter{figure}{extfigure}

\newpage

\appendix

\begin{frontmatter}
	
	\title{Supplementary Article to \SIname }
	\runtitle{\SIshortname (Supplementary Article)}

\end{frontmatter}

% this needs \usepackage{xr,refcount} in sourcing doc's preamble

% counter names need to match counters stored in storecounters.tex

% set this to match counters in main file
\newcommand*{\getcounter}[2]{%
  \setcounterref{#1}{#2}% Retrieve label value and store it in a counter
}

\getcounter{thm}{extthm}
\getcounter{cor}{extcor}
\getcounter{defn}{extdefn}
\getcounter{lemma}{extlemma}
\getcounter{prop}{extprop}
\getcounter{figure}{extfigure}

\section{Proof of Lemma~\ref{lem:crossing}}
\label{sec:proof-lem:crossing}

Define the probability distribution $d\tilde{H}=\gamma dH + (1-\gamma)\delta_0$, so that, by definition, $\tilde{\varphi}_\gamma(y;H)=\varphi(y;\tilde{H})$.
Since
\[
\TCritBare(y;H) = y\cdot \frac{\m{D}'(y;H)}{\m{D}(y;H)} = y \cdot \left( \frac{\varphi'(y;H)}{\varphi(y;H)} + \frac{\varphi'(y;\tilde{H})}{\varphi(y;\tilde{H})} \right)\,,
\]
to show that $\TCritBare$ is increasing, it suffices to show that $y \mapsto y \cdot \frac{\varphi'(y;H)}{\varphi(y;H)}$ is increasing for any CDF $H$ and $y>\bulkEdge(H)$.  

Introduce a change of variables $w=\log(y)$ and set $\psi(w)=\phi(e^w;H)$. We have 
\begin{align*}
	y \cdot \frac{\varphi'(y;H)}{\varphi(y;H)}
	= y \cdot \frac{d}{dy}\left( \log \varphi(y;H) \right) = y \cdot \frac{d}{dw}\left( \log \varphi(e^w;H) \right) \cdot \frac{dw}{dy} = \frac{d}{dw} \left(\log \psi(w) \right) \,.
\end{align*}
Since $w$ is strictly increasing in $y$, it remains to show that $w \mapsto \left(\log \psi(w)\right)'$ is increasing, equivalently, that $w\mapsto \psi(w)$ is log-convex. Write 
\[
\psi(w) = \int \psi_{z}(w) dH(z)\,,\quad\textrm{ where }\quad \psi_z(w) = \frac{e^w}{e^{2w}-z^2} \,.
\]
Since a convex combination of log-convex functions is log-convex,
it suffices to verify that each $\psi_z(w)$ is log-convex, whenever $y^2=e^{2w} > z^2$. A straightforward calculation gives:
\begin{align*}
	\left( \log \psi_z(w) \right)' &= 1 - \frac{2e^{2w}}{e^{2w}-z^2} = -\frac{z^2+e^{2w}}{e^{2w}-z^2} = -1 - \frac{2z^2}{e^{2w}-z^2}\,,
\end{align*}
which is negative and clearly increasing in $w$. Thus, $\psi_z(w)$ is log-convex, and so we conclude that $y\mapsto \m{D}(y;H)$ is strictly increasing. For the limit as $y\to\infty$, write
\[
\varphi(y;H)=\int \frac{y}{y^2-z^2}dH(z) = \frac{1}{y} + o\left( \frac{1}{y^2}\right)\,,\quad \varphi(y;H)=-\int \frac{y^2+z^2}{(y^2-z^2)^2}dH(z) = -\frac{1}{y^2} + o\left( \frac{1}{y^3}\right)\,,
\]
as $y\to\infty$. Thus, $\TCritGamma(y;H)=-2 + o(1)$ as $y\to\infty$. 

For part (2), observe that the additional assumption on $H$ implies that $\frac{\varphi'(y;H)}{\varphi(y;H)}\to -\infty$ as $y\to\bulkEdge(H)$ from the right, hence $\TCritBare(y)\to-\infty$. Now, since $\TCritBare(y)$ is continuous on $y\in (\bulkEdge(H),\infty)$, it must attain $\TCritBare(y^*)=-4$ for some $y^*$.  This $y^*$ must be unique since, as we have proved in (1), $\TCritBare(y;H)$ is strictly increasing.

\section{Proof of Lemma~\ref{lem:finite-rank}}
\label{sec:proof:lem:finite-rank}
		
%\begin{proof}
	Recall that for any matrix $A\in \R^{n\times p}$ with SVD $A=\sum_{i=1}^p \sigma_i \V{u}_i\V{v}_i^\T$, its best rank-$r$ approximation with respect to Frobenius norm is obtained by taking its $r$ leading principal components. Since $X$ has rank $r$, for any $k\ge M$,
	\[
	\|X-\hat{X}_{[k]}\|_F^2 \ge \min_{\mathrm{rank}(B)=r}\|B-\hat{X}_{[k]}\|_F^2 = \sum_{i=r+1}^k y_{i,n}^2 \ge \sum_{i=r+1}^M y_{i,n}^2 \,.
	\]
	Recall that any fixed $i\ge r+1$ satisfies $y_{i,n}\aslim y_{i,\infty}=\bulkedge$. Since $M$ is constant, and satisfies $M > r + \ASE^*[\V{x}]/\bulkedge^2$, we obtain that 
	\[
	\sum_{i=r+1}^M y_{i,n}^2 \aslim (M-r)\bulkedge^2 > \ASE^*[\V{x}] \,.
	\]
	Since $\SEn^*[\V{x}]\le \SEn[\V{x}|\optThresh(\FZ)]\aslim \ASE^*[\V{x}]$, we conclude that almost surely, for all large enough $n$, $\SEn^*[\V{x}] < \min_{k\ge M}\|X-\hat{X}_{[k]}\|_F^2$. 
%\end{proof}

\section{Proof of Proposition~\ref{prop:correlations}}
\label{sec:proof-prop:correlations}

This proposition asserts 
the asymptotic de-cross-correlation 
between principal ($j \leq r$) population singular vectors
and non-principal ($j > r$) empirical (sample) singular vectors. 
We assume distinct population
principal singular values.
Results on the asymptotic de-cross-correlation between 
``off diagonal'' combinations of 
population principal and empirical principal  singular vectors
have already been discussed in the main text near (\ref{decoupling:eq}), where they 
follow for example \cite{benaych2012singular}, and also in several
earlier works on the spiked covariance model. 
In contrast, in Proposition~\ref{prop:correlations} only one of the two
vectors being compared is principal, and the other is sub-principal.

%Our argument relies on an idea due to \cite{nadler2008finite}. 
Our argument relies on the ``arrowhead representation'' of
the spiked covariance eigenproblem; see (\ref{eq:SI:arrowhead}) below.
We learned of this representation from \cite{nadler2008finite}.

Recall that our data matrix is $Y_n = \sum_{\ell=1}^r x_\ell \V{a}_{\ell,n}\V{b}_{\ell,n}^\T + Z_n$. Our goal is to show that for every fixed $1\le i \le r$ and $j\ne i$, one has 
\[
\langle \V{a}_{n,i}\,,\,\V{u}_{n,j}\rangle \cdot
\langle \V{b}_{n,i}\,,\,\V{v}_{n,j}\rangle \aslim 0 \,,
\]
with $\V{u}_{n,j},\V{v}_{n,j}$ being, respectively, the left and right $j$-th singular vectors of $Y_n$. For notational convenience, let us assume, without loss of generality, that $i=1$. We will show that $\langle \V{b}_{n,i}\,,\,\V{v}_{n,j}\rangle \aslim 0$, which (since the inner products are bounded) clearly suffices. Recall also that $\V{v}_{n,j}$ is the $j$-th eigenvector of the $p$-by-$p$ matrix $Y_n^\T Y_n$.
Note that we may also assume without loss of generality that the eigenvalues of $Y_n^\T Y_n$ (equivalently, the $p$ singular values of $Z_n$) are all distinct.\footnote{
Otherwise, one could add to $Y_n$ an orthogonally-invariant but very weak independent perturbation $W_n$ (e.g., an i.i.d. Gaussian matrix); doing so will only infinitesimally change the corresponding singular value correlations, and in the limit $\|W_n\|\to 0$ the resulting SVD of $Y_n+W_n$ will produce singular values with the same distribution at those of $Y_n$ (we defined that $\V{v}_{i,n}$/$\V{u}_{i,n}$-s that correspond to multidimensional singular spaces are uniformly random on these subspaces). 
}
We may further assume throughout the proof that the distribution of $Z_n$ is orthogonally invariant (both from the left and right); this is because $(\V{a}_{n,1},\ldots,\V{a}_{n,r})$ and $(\V{b}_{n,1},\ldots,\V{b}_{n,r})$ have an orthogonally-invariant distribution and are independent of $Z_n$ (and, of course, the inner products we would like to compute are invariant to a global orthogonal transformation applied to both the population spikes and $Y_n$).

Let $P_n=I-\V{b}_{1,n}\V{b}_{1,n}^\T$ be the projection onto the orthogonal complement of $\V{b}_{1,n}$, and set 
\[
\tilde{Z}_n = \sum_{\ell=2}^r x_\ell  \V{a}_{\ell,n}\V{b}_{\ell,n}^\T + Z_n,
\]
so that $Y_n = x_1 \V{a}_{1,n}\V{b}_{1,n}^\T + \tilde{Z}_n$. Let $\V{q}_{2,n},\ldots,\V{q}_{p,n}$ be an orthonormal basis of $\mathrm{Range}(P_n)$, that diagonalizes the linear operator $\left. P_n \tilde{Z}_n^\T \tilde{Z}_n P_n \right|_{\mathrm{Range}(P_n)}$ (that is, the restriction of the matrix $P_n \tilde{Z}_n^\T \tilde{Z}_n P_n$ onto the linear subspace $\mathrm{Range}(P_n)$). 
Importantly, observe that the vectors $\V{q}_{2,n},\ldots,\V{q}_{p,n}$ do not depend on the population left singular vectors $\V{a}_{1,n},\ldots,\V{a}_{r,n}$; moreover, they remain unchanged when $\tilde{Z}_n$ is multiplied from the left by any orthogonal matrix.
Let $\mu_{2,n},\ldots,\mu_{p,n}$ be the corresponding eigenvalues, that is, 
\[
(P_n \tilde{Z}_n^\T \tilde{Z}_n P_n) \V{q_{\ell,n}} = \mu_{\ell,n}\V{q}_{\ell,n},\quad 2\le \ell\le p\,.
\]
Denote the $p$-by-$p$ orthogonal matrix $\m{U}=[\V{b}_{1,n},\V{q}_{2,n},\ldots,\V{q}_{p,n}]$, whose columns consists of the aforementioned orthonormal basis. 
The change of basis $\mathcal{U}$ has been explicitly chosen
so that, in this basis, $Y_n^\T Y_n$ has the very particular form
of a so-called {\it arrowhead matrix}, which have known useful
exact closed-form expressions for eigenvalues and eigenvectors.
% ,
% and which we will exploit.
That is, upon conjugation by $\m{U}$, 
\begin{equation}\label{eq:SI:arrowhead}
	\m{U} (Y_n^\T Y_n)\m{U}^\T = \begin{bmatrix}
		\alpha_n \quad& \V{w}_n^\T \\
		\V{w}_n \quad& \mathrm{diag}(\mu_{2,n},\ldots,\mu_{p,n}) 
	\end{bmatrix},
\end{equation}
where
\begin{equation}
		\alpha_n = \V{b}_{1,n}^\T (Y_n^\T Y_n) \V{b}_{1,n},
\end{equation}
and $\V{w}_n = (w_{2,n},\ldots,w_{p,n})$ is a $(p-1)$-dimensional column vector:
\begin{equation}\label{eq:SI:weights}
	w_{\ell,n} = \V{b}_{1,n}^\T(Y_n^\T Y_n) \V{q}_{\ell,n} \,, \quad 2\le \ell \le p\,.
\end{equation}
Importantly, the bottom-right $(p-1)$-by-$(p-1)$ minor of (\ref{eq:SI:arrowhead}) is diagonal. 
%A matrix of this form is called an \emph{arrowhead matrix}; 
It is known \cite{nadler2008finite} that the eigenvectors of a matrix of the form (\ref{eq:SI:arrowhead}), denoted $\V{p}_{1,n},\ldots,\V{p}_{p,n}$ are, up to normalization:
\begin{equation}
	\label{eq:SI:eigenvector}
	\V{p}_{\ell,n} = \left( 1, \frac{w_{2,n}}{\lambda_\ell - \mu_{2,n}}, \ldots, \frac{w_{n,n}}{\lambda_\ell - \mu_{n,n}} \right) \,,\quad 1\le \ell \le p\,,
\end{equation}
where $\lambda_\ell$ is the corresponding eigenvalue. Recalling Eq. (\ref{eq:SI:arrowhead}), the eigenvalues and eigenvectors of the arrowhead matrix are related to those of $Y_n^\T Y_n$:
\begin{equation}
	\lambda_\ell = y_{\ell,n}^2,\quad \V{v}_{\ell,n}=\pm \m{U} \V{p}_{\ell,n}/\|\m{U} \V{p}_{\ell,n}\|\,.
\end{equation}
In particular, combining with Eq. (\ref{eq:SI:eigenvector}),
\begin{equation}\label{eq:SI:2}
	\left| \langle \V{b}_{1,n}, \V{v}_{j,n} \rangle \right| = \left| (\m{U}^\T \V{v}_{j,n})_1 \right| = \frac{|(\V{p}_{j,n})_1|}{\|\V{p}_{j,n}\|} = \left( 1 + \sum_{\ell=2}^{p} \frac{w_{\ell,p}^2}{(y_{j,n}^2 - \mu_{\ell,n})^2}\right)^{-1/2} \,.
\end{equation}
%Thus, Proposition~\ref{prop:correlations} readily follows from Lemma~\ref{lem:SI:long1} below:
The RHS of (\ref{eq:SI:2}) has an exact closed form that 
allows the argument for Proposition \ref{prop:correlations} to be completed by Lemma \ref{lem:SI:long1} below. \qed 

\begin{lemma}\label{lem:SI:long1}
	One has
	\begin{equation}\label{eq:lem:SI:long1}
	\sum_{\ell=2}^{p} \frac{w_{\ell,p}^2}{(y_{j,n}^2 - \mu_{\ell,n})^2} \aslim \infty \,.
	\end{equation}
\end{lemma}

%The rest of this section is devoted to proving Lemma~\ref{lem:SI:long1}. 

Towards the proof of Lemma~\ref{lem:SI:long1}, we start with a simpler claim:

\begin{lemma}
	\label{lem:SI:long2}
	One has 
	\begin{equation}\label{eq:lem:SI:long2}
		\frac{1}{p-1}\sum_{\ell=2}^{p} \frac{\mu_{\ell,n}}{(y_{j,n}^2 - \mu_{\ell,n})^2} \aslim \infty \,.
	\end{equation}
\end{lemma}
\begin{proof}
	(Of Lemma~\ref{lem:SI:long2}.)
	It is suggestive to write the LHS of (\ref{eq:lem:SI:long2}) as 
	\begin{equation}
		\frac{1}{p-1}\sum_{\ell=2}^{p} \frac{\mu_{\ell,n}}{(y_{j,n}^2 - \mu_{\ell,n})^2} = \int \frac{\mu}{(y_{j,n}^2-\mu)^2} dF_{P_n \tilde{Z}_n^\T \tilde{Z}_n P_n}(\mu)\,,
	\end{equation}
where $dF_{P_n \tilde{Z}_n^\T \tilde{Z}_n P_n}(\mu)=\frac{1}{p-1} \sum \frac{1}{p-1} \delta(\mu-\mu_{\ell,n})$ is the empirical eigenvalue distribution (counting measure) of $P_n \tilde{Z}_n^\T \tilde{Z}_n P_n$. 

We know by \cite{benaych2012singular} that $y_{j,n}^2 \aslim \bulkEdge(F_{Z^T Z})$, where $F_{Z^\T Z}$ denotes the limiting eigenvalue distribution\footnote{Recall: $F_Z$ is the limiting distribution of \emph{singular values} of $Z_n$; they are related to the eigenvalues of $Z_n^\T Z_n$ by $\lambda_{\ell}(Z_n^\T Z_n)=\sigma_\ell^2(Z_n)$.} of $Z_n^\T Z_n$. Moreover, recall that by assumption~\ref{assum:dense}, 
\[
\lim_{t\to \bulkEdge(F_{Z^\T Z})} \int \frac{\mu}{(t-\mu)^{2}} dF_{Z^\T Z}(\mu)=\infty \,.
\]
Consequently, the proof of Lemma~\ref{lem:SI:long2} is concluded once one shows that $F_{P_n \tilde{Z}_n^\T \tilde{Z}_n P_n} \dlim F_{Z^\T Z}$ almost surely. To see this, it suffices to note that:
\begin{enumerate}
	\item $P_n \tilde{Z}_n^\T \tilde{Z}_n P_n$ is a $(p-1)$-by-$(p-1)$ minor of $\tilde{Z}_n^\T \tilde{Z}_n$, and therefore by eigenvalue interlacing, they have the same limiting eigenvalue distribution.
	\item $\tilde{Z}_n = Z_n + \sum_{\ell=2}^r x_\ell \V{a}_{\ell,n}\V{b}_{\ell,n}^\T$ is a finite-rank additive perturbation of $Z_n$; hence by singular value interlacing, it has the same limiting singular value distribution as $Z_n$ (with at most $r-1$ outlying singular values); thus, the limiting eigenvalue distribution of $\tilde{Z}_n^\T \tilde{Z}_n$ is $F_{Z^\T Z}$.
\end{enumerate}

\end{proof}

Equipped with Lemma~\ref{lem:SI:long2}, we now prove Lemma~\ref{lem:SI:long1}. 
% We will provide a formal proof of the Lemma only under the simplifying assumption $r=1$. Having done that, we shall next sketch an argument for the $r>1$ case; the argument is conceptually similar to the $r=1$ case, but involves some cumbersome (but otherwise harmless) technical details, that we do not work out fully.

\begin{proof}
	(Of Lemma~\ref{lem:SI:long1}
    % , assuming $r=1$.
)
	Let us analyze the weights $w_{\ell,n}^2$ in (\ref{eq:lem:SI:long1}). Recalling (\ref{eq:SI:weights}),
	\begin{align*}
		w_{\ell,n} 
		&= \V{b}_{1,n}^\T (\tilde{Z}_n^\T  + x_1 \V{b}_{1,n}\V{a}_{1,n}^\T)( \tilde{Z}_n + x_1 \V{a}_{1,n}\V{b}_{1,n}^\T ) \V{q}_{\ell,n} \\
		&= (\V{b}_{1,n}^\T\tilde{Z}_n^\T  + x_1 \V{a}_{1,n}^\T)\tilde{Z}_n \V{q}_{\ell,n}\\
		&= \V{b}_{1,n}^\T (\tilde{Z}_n^\T\tilde{Z}_n) \V{q}_{\ell,n} + x_1 \V{a}_{1,n}^\T \tilde{Z}_n \V{q}_{\ell,n} \,,
	\end{align*}
	where we used $\V{b}_{1,n} \perp \V{q}_{\ell,n}$ (by construction). 
	Note that we may assume with loss of generality that $\V{b}_{1,n}^\T (\tilde{Z}_n^\T\tilde{Z}_n) \V{q}_{\ell,n} \ge 0$ (since we may a priori replace any $\V{q}$ by $-\V{q}$), 
	and therefore,
	\begin{equation}\label{eq:SI:w}
		w_{\ell,n}^2 \ge x_1^2 (\V{a}_{1,n}^\T \tilde{Z}_n \V{q}_{\ell,n})^2 \Ind{\V{a}_{1,n}^\T \tilde{Z}_n \V{q}_{\ell,n} \ge 0}.
	\end{equation}

    Recall that by their definition, the vectors $\{\tilde{Z}_n\V{q}_{\ell,n}\}_{2\le \ell \le p}$ are orthogonal to one another, with $\|\tilde{Z}_n\V{q}_{\ell,n}\|^2=\mu_{\ell,n}$. 
    Let us, for the moment, make the simplifying assumption $r=1$.
	% Using the simplifying assumption $r=1$, we have in this
    In that case, $\tilde{Z}_n=Z_n$, hence $\{\tilde{Z}_n\V{q}_{\ell,n}\}_{2\le \ell \le p}$ are independent of $\V{a}_{1,n}$.  Since $\V{a}_{1,n}$ is uniformly distributed on the $n$-dimensional unit sphere, we may write $\V{a}_{1,n}=\V{g}/\|\V{g}\|$ for $\V{g}\sim \m{N}(0,I_n)$. Then $\{\V{g}^\T \tilde{Z}_n \V{q}_{\ell,n}\}_{2\le \ell \le p}$ are independent (univariate) Gaussians, being the projections of a Gaussian vector onto orthogonal directions. 
    % Recalling that $\tilde{Z}_n\V{q}_{\ell,n}$ are orthogonal to one another ($2\le \ell \le p$), with $\|\tilde{Z}_n\V{q}_{\ell,n}\|^2=\mu_{\ell,n}$, we can write the joint distribution
	To wit, the following equality in distribution holds:
    \[
	\left( \V{a}_{1,n}^\T \tilde{Z}_n \V{q}_{2,n},\ldots,\V{a}_{1,n}^\T \tilde{Z}_n \V{q}_{p,n} \right) \overset{d}{=} \left( \frac{\mu_{2,n}^{1/2}}{\|\V{g}\|}g_2,\ldots, \frac{\mu_{p,n}^{1/2}}{\|\V{g}\|}g_p \right).
	\]
	Using (\ref{eq:SI:w}), $n^{-1}\|\V{g}\|^2\aslim 1$, and the strong law of large numbers,
	\begin{align*}
		\sum_{\ell=2}^{p} \frac{w_{\ell,p}^2}{(y_{j,n}^2 - \mu_{\ell,n})^2} 
		&\ge \frac{x_1^2}{\|\V{g}\|^2}\sum_{\ell=2}^p 
		\frac{\mu_{\ell,n}}{(y_{j,n}^2 - \mu_{\ell,n})^2} g_\ell^2 \Ind{g_\ell\ge 0} \\
		&\approx x_1^2\cdot \frac{p-1}{n} \cdot \frac{1}{p-1}\sum_{\ell=2}^{p} \frac{\mu_{\ell,n}}{(y_{j,n}^2 - \mu_{\ell,n})^2} \cdot \E[g_\ell^2 \Ind{g_\ell\ge 0}] \aslim \infty\,,
	\end{align*}
	where the last limit follows from Lemma~\ref{lem:SI:long2}. This proves Lemma~\ref{lem:SI:long1} assuming $r=1$.

Let us now adapt the above argument for any fixed $r\ge 1$. The difficulty in doing so lies with the fact that $\V{a}_{1,n}$ is no longer independent of $\{\tilde{Z}_n\V{q}_{\ell,n}\}_{2\le \ell \le p}$, since $\tilde{Z}_n = \sum_{\ell=2}^r x_\ell  \V{a}_{\ell,n}\V{b}_{\ell,n}^\T + Z_n$, and $\V{a}_{1,n}$ is orthogonal to $\V{a}_{2,n},\ldots,\V{a}_{r,n}$. The key fact is that this dependence is rather weak, as we will show. 

Let us condition on $\V{a}_{2,n},\ldots,\V{a}_{r,n}$, and let $\m{P}_A$ be the projection onto their orthogonal complement. Then we can write the conditional distribution of $\V{a}_{1,n}$ as $\V{a}_{1,n}\overset{d}{=} \m{P}_A(\V{g})/\|\m{P}_A(\V{g})\|$, where $\V{g}\sim\m{N}(0,I_n)$. Since $\m{P}_A$ projects onto a subspace of dimension $n-r$, with $r$ constant, $\|\m{P}_A(\V{g})\|^2/n\aslim 1$. Denote $\V{\zeta}_{\ell,n} = \tilde{Z}_n\V{q}_{\ell,n}/\mu_{\ell,n}^{1/2}$, $2\le \ell \le p$, which are $p-1$ orthonormal vectors. We have, as was before,
\begin{align*}
    \sum_{\ell=2}^{p} \frac{w_{\ell,p}^2}{(y_{j,n}^2 - \mu_{\ell,n})^2} 
		&\gtrsim 
        x_1^2\cdot \frac{p-1}{n} \cdot \frac{1}{p-1}\sum_{\ell=2}^{p} \frac{\mu_{\ell,n}}{(y_{j,n}^2 - \mu_{\ell,n})^2} \cdot (\V{\zeta}_{\ell,n}^\T \m{P}_A\V{g})^2\Ind{\V{\zeta}_{\ell,n}^\T \m{P}_A\V{g}\ge 0}\,.
\end{align*}

Since we assumed that the distribution of $\tilde{Z}_n$ is orthogonally invariant from the left (see beginning of this section), and since the eigenvectors $\V{q}_{2,n},\ldots,\V{q}_{p,n}$ are not dependent on such orthogonal left multiplication, the random signs $\{\mathrm{sign}(\V{\zeta}_{\ell,n}^\T \m{P}_A\V{g})\}_{2\le \ell \le p}$ are i.i.d. $\mathrm{Ber}(1/2)$ and independent of the modulii $(\V{\zeta}_{\ell,n}^\T \m{P}_A\V{g})^2$. Consequently,  in light of Lemma~\ref{lem:SI:long2}, it suffices to prove that for any $\varepsilon>0$, there is constant $C_\varepsilon$ such that 
\[
\frac1p\sum_{\ell=2}^{\lfloor\varepsilon p\rfloor} (\V{\zeta}_{\ell,n}^\T \m{P}_A\V{g})^2 \ge C_\varepsilon    
\]
holds asymptotically almost surely as $p\to\infty$. The above sum has a simple geometric interpretation: one first projects $\V{g}$ onto the complement of $\V{a}_{2,n},\ldots,\V{a}_{r,n}$, and then projects the result onto the span of of $ \V{\zeta}_{2,n},\ldots,\V{\zeta}_{\lfloor\varepsilon p\rfloor,n}$; the sum is the remaining energy (squared $L_2$ norm). Let $\m{W}=\mathrm{span}\left(\V{a}_{2,n},\ldots,\V{a}_{r,n}), \{ \V{\zeta}_{2,n},\ldots,\V{\zeta}_{\lfloor\varepsilon p\rfloor,n} \}^\perp \right) \subseteq \R^p$, so that $\mathrm{dim}(\m{W})\le r + p - \lfloor \varepsilon p \rfloor + 1$. Certainly, $\frac1p\sum_{\ell=2}^{\lfloor\varepsilon p\rfloor} (\V{\zeta}_{\ell,n}^\T \m{P}_A\V{g})^2 \ge \frac1p \|\m{P}_{\m{W}^\perp} \V{g}\|^2$, where $\m{P}_{\m{W}^\perp}$ is the projection onto the complement of $\m{W}$. Since $r$ is constant, $\mathrm{dim}(\m{W}^\perp) \ge p(\varepsilon-o(1))$, and consequently, $\frac1p \|\m{P}_{\m{W}^\perp} \V{g}\|^2 \gtrsim \varepsilon$ almost surely as $p\to\infty$. Thus, the proof is concluded.

\end{proof}

\section{Estimating $\optThresh(\FZ)$: Auxiliary Lemmas}
\label{sec:proof:estimating}

\subsection{Proof of Lemma~\ref{lem:main:continuity}}
Let $\epsilon>0$ be small. By assumption, $\optThresh(H)>\bulkEdge(H)$, and it is the unique number satisfying $\TCritGamma(\optThresh(H);H)=-4$. Let $y_1=\optThresh(H)-\epsilon/2$, $y_2=\optThresh(H)+\epsilon/2$, where $\epsilon$ was chosen small enough so that $\bulkEdge(H)<y_1<y_2$. Note that $\TCritGamma(y_1;H)<-4<\TCritGamma(y_2;H)$, as $\Psi(\cdot;H)$ is increasing. Since $H_n\dlim H$ (and $|p/n-\gamma|\le 1/n \to 0$) we find that for all large enough $n$, $\TCritPoverN(y_1;H_n)<-4<\TCritPoverN(y_2;H_n)$, since $\TCritPoverN(y_1;H_n)\to \TCritPoverN(y_1;H)$ and $\TCritPoverN(y_1;H_n)\to \TCritPoverN(y_2;H)$. Since also $\bulkEdge(H_n) \to \bulkEdge(H)<y_1$, we deduce that for large enough $n$, $y_1<T_{p/n}(H_n)<y_2$, that is, $|T_{p/n}(H_n)-\optThresh(H)|<\epsilon$ .

\subsection{Proof of Lemma~\ref{lem:main:approx-FZ}} 
Part (1) will follow from (3), since convergence in KS distance implies weak convergence, and $\FZn\dlim \FZ$ almost surely. For part (3),
denote by $z_{i,n}$, $i=1,\ldots,p$, the singular values of $Z_n$. By Weyl's interlacing inequality, since $Y_n=X_n+Z_n$ and $\mathrm{rank}(X)=r\le k$,
\[
z_{i,n} \le y_{i,n} \le z_{i-k,n}\,,\quad\textrm{ for }i=k+1,\ldots,p \,.
\]
Fix some $y$, and let $j$ be the smallest index $j\ge k+1$ such that $y_{j,n}\le y$ (set $j=0$ if no such index exists). The interlacing inequality states that $z_{j,n}\le y$ as well. If $j=k+1$, then at worst the interval contains all of $z_{k,n},\ldots,z_{1,n}$, and none of the additional ``pseudo-singular values'' we have introduced. Hence, in that case, $|\FZn(y)-F_{n,k}^\star(y)|\le k/p$. Now, suppose that $j>k+1$. Since $y_{j+1,n}>y$, the interlacing inequality gives $z_{j+1-k,n}\ge y_{j+1,n}>y$, hence in addition to $z_{p,n},\ldots,z_{j,n}$, the interval $(-\infty,y]$ contains at most $k$ additional singular values of $Z_n$, specifically $z_{j+1,n},\ldots,z_{j-k}$. In the worst case, we have added no ``pseudo-singular values'' with $\le y$, hence again $|\FZn(y)-F_{n,k}^\star(y)|\le k/p$. Lastly, to prove (2), note that for any $\epsilon>0$, $\FZ\left( \bulkEdge(\FZ)-\epsilon  \right)<1$, which means that almost surely, for large enough $n$, there are $\Omega(n)$ singular values of $Z_n$ bigger than $\bulkEdge(\FZ)-\epsilon$. By interlacing, for any $m_n=o(n)$ with $r<m_n$, $y_{m_n,n}\ge z_{m_n,n}$, and $z_{m_n,n}$ must eventually be among those singular values bigger than $\bulkEdge(\FZ)-\epsilon$. But this is true for any $\epsilon>0$, meaning that $y_{m_n}\aslim \bulkEdge(\FZ)$ whenever $r<m_n=o(n)$. 

\subsection{Proof of Lemma~\ref{lem:main:continuity-quantitive}}
Fix any $a$ satisfying $\bulkEdge(H)<a<\optThresh(H)$. Observe that one can find a neighborhood $\m{N}$ of $\optThresh(H)$ and $c_1,c_2>0$ such that $c_2<\TCritBare_{\gamma'}'(y;G)<c_1$ for all $y\in\m{N}$ and any distribution $G$ supported on $[0,a]$ and $\gamma'$ (to see this, simply follow calculations in the proof of Lemma~\ref{lem:crossing}). 
Now, since $T_{p/n}(H_n) \to \optThresh(H)$, we see that $T_{p/n}(H_n)\in \m{N}$ for all large enough $n$. Consequently, by the mean value theorem,
% Thus, since $T_{p/n}(H_n) \to \optThresh(H)$, for all large enough $n$,
\[
c_2 (T_{p/n}(H_n) - \optThresh(H)) < \TCritPoverN(T_{p/n}(H_n);H_n) - \TCritPoverN(\optThresh(H);H_n) \le c_1(T_{p/n}(H_n) - \optThresh(H)) \,.
\]
Using $\TCritGamma(\optThresh(H);H)=\TCritPoverN(T_{p/n}(H_n);H_n)=-4$, and the fact that $\TCritBare$ is increasing, we conclude that 
\begin{align*}
	|T_{p/n}(H_n)-\optThresh(H)| 
	&\le  \max(1/c_1,1/c_2)\cdot \left| \TCritPoverN(T_{p/n}(H_n);H_n) - \TCritPoverN(\optThresh(H);H_n) \right| \\
	&= \max(1/c_1,1/c_2)\cdot \left| \TCritGamma(T_{\gamma}(H);H) - \TCritPoverN(\optThresh(H);H_n) \right| \,.
\end{align*}
The right-hand-side is now $\m{O}(\Delta_{1,n} + \Delta_{2,n} + \left|  p/n-\gamma\right|)$. 

\subsection{Proof of Proposition~\ref{prop:main:perturb}}
By Lemma~\ref{lem:main:continuity-quantitive}, we need to show that 
\[
|\varphi(\optThresh(\FZ);\FZ)-\varphi(\optThresh(\FZ);F_{n,k}^\star)|,\, |\varphi'(\optThresh(\FZ);\FZ)-\varphi'(\optThresh(H);F_{n,k}^\star)| = \m{O}_{\prob}(k/p) 
\]
(by definition, $|\gamma_n-\gamma|\le 1/n$). Write
\begin{eqnarray*}
	\varphi(\optThresh(\FZ);\FZ)-\varphi(\optThresh(\FZ);F_{n,k}^\star) &=& \left[ \varphi(\optThresh(\FZ);\FZ)-\varphi(\optThresh(\FZ);\FZn)\right]\\ &+& \left[\varphi(\optThresh(\FZ);\FZn) - \varphi(\optThresh(\FZ);F_{n,k}^\star)\right] \,,
\end{eqnarray*}
we bound each bracket separately (the argument when $\varphi$ is replaced by $\varphi'$ is the same). The expression $\varphi(\optThresh(\FZ);\FZ)=\int \frac{\optThresh(\FZ)}{\optThresh(\FZ)^2-z^2}d\FZ(z)$ is a linear spectral statistic, and satisfies the requirement of \cite{bai2004}, by assumption - see Section~\ref{sec:correlated-noise}.  Hence, 
\[
\left| \varphi(\optThresh(\FZ);\FZ)-\varphi(\optThresh(\FZ);\FZn)\right| = \m{O}_{\prob}(1/p) \,.
\]
For the second term, recall that if $F_1$ and $F_2$ are CDFs supported on the interval $I$ and $g:I\to \R$ is bounded and continuously differentiable, then
\[
\left|  \int g(t)( dF_1(t)-dF_2(t) ) \right| \le \left(\|g\|_{L^\infty(I)} + \|g'\|_{L^1(I)}\right)\cdot \|F_1-F_2\|_{\mathrm{KS}} \,. 
\]
Since both $\bulkEdge(\FZn),\bulkEdge(F_{n,K}^\star)\aslim \bulkEdge(\FZ)$, by Lemma~\ref{lem:main:approx-FZ}, almost surely, 
\[\left|\varphi(\optThresh(\FZ);\FZn) - \varphi(\optThresh(\FZ);F_{n,k}^\star)\right|=\m{O}(k/p)\,.\]

\section{Additional numerical experiments}
\label{app:numerics}

In this section we provide extensive numerical experiments validating different aspects of our results under various noise distributions.

\subsection{Noise distributions}

We have conducted experiments using the following noise distributions:

\begin{itemize}
    \item {\bf Mar\v{c}enko-Pastur}: the matrix $Z$ is an i.i.d Gaussian matrix, so that $\FZ$ is a Mar\v{c}enko-Pastur with shape parameter $\gamma$.
    \item {\bf Chi10}: A noise matrix with correlated columns, such that $W$ is i.i.d Gaussian and $\FS$ is the law of a $\chi$-squared random variable with $10$ degrees of freedoms, normalized to have variance $1$: $T=\frac{1}{10}\sum_{i=1}^{10} g_i$ where $g_1,\ldots,g_{10}\sim \m{N}(0,1)$.
    \item {\bf Mix2}: A noise matrix with correlated columns, such that $W$ is i.i.d Gaussian, and $\FS$ is an equal mixture of two atoms: $d\FS = \frac12 \delta_1 + \frac12 \delta_{10}$.
    \item {\bf Unif[1,10]}: A noise matrix with correlated columns, such that $W$ is i.i.d Gaussian, and $\FS$ is the uniform distribution on $[1,10]$.
    \item {\bf Fisher3n}: $Z$ has the form $Z=W_1S_2^{-1/2}$ where $W_1\in \R^{n\times p}$ is i.i.d Gaussian $\m{N}(0,1/n)$, and $S_2=W_2^\T W_2$ where $W_2 \in \R^{3p\times p}$ is an i.i.d Gaussian matrix with entries $\m{N}(0,1/(3n))$. Matrices of this form have been studied in the literature under the name F-matrices (also F-ratios, Fisher matrices). Their limiting singular value distribution is given by Wachter's law \cite{wachter1980limiting}, see also \cite{yin1983limiting,silverstein1985limiting,bai2010spectral}.
    \item {\bf PaddedIdentity:} All the singular values of $Z$ are $1$, that is, $d\FZ=\delta_1$. Specifically, $Z$ is the matrix $Z=[\mathbf{I}_{p\times p},\mathbf{0}_{p\times (n-p)}]^\T $, that is, a $p$-by-$p$ identity matrix, padded by zeros.
\end{itemize}

In the table below, one can find some useful quantities corresponding to the distributions above (with select choices of $\gamma$): the bulk edge $\bulkEdge(\FZ)$, the location of the BBP PT $\BBP$, optimal threshold $\optThresh(\FZ)$ and $x^*=\m{Y}^{-1}\left(\optThresh(\FZ)\right)$. All these quantities are estimated by sampling a large noise matrix $Z$ (specifically, we take $p=3000$ and $n=p/\gamma$) and estimating all the necessary functionals by their plugin estimates, putting $\FZn$ in place of $\FZ$, and rounding all numbers to $2$ digits after the decimal point. Estimating $\BBP$ in this manner is especially problematic, since for any counting measure $H$, in our case $H=\FZn$, $\BBP(H)=\lim_{y\to\bulkEdge(H)}\left( D_{p/n}(y;H) \right)^{-1/2} = 0$, since $\m{D}(\cdot;H)$ has $1/y$ singularity near $\bulkEdge(H)$. For our purposes, we use the heuristic $\BBP \approx \left( D_{p/n}\left(\bulkEdge(\FZn)+0.01;\FZn\right) \right)^{-1/2}$, which may be somewhat off from the true PT location.\footnote{For instance, when the noise is Mar\v{c}enko-Pastur, exact expressions are available, see e.g \cite{Donoho2013b}, and $\BBP=\gamma^{1/4}$. This is quite off from the expression in the table! } In the case of {\bf PaddedIdentity}, it is obvious that $\BBP=0$, and this is what we give below.

\begin{center}
 \begin{tabular}{||c c c c c c||} 
 \hline
 {\bf Distribution:} & $\gamma$ & $\bulkEdge\left(\FZ\right)$ & $\BBP(\FZ)$ &  $T_\gamma\left(\FZ\right)$ & $\m{Y}^{-1}\left(T_\gamma\left(\FZ\right)\right)$ \\
 \hline\hline
{\bf Mar\v{c}enko-Pastur} & 0.5 & 1.7 & 0.91 & 1.98 & 1.48 \\
 & 1.0 & 2.0 & 1.07 & 2.31 & 1.73\\
 \hline
{\bf Chi10} & 0.5 & 2.11 & 1.59 & 2.17 & 1.7\\
 & 1.0 & 2.26 & 1.5 & 2.46 & 1.89\\
\hline
{\bf Fisher3n} & 0.5 & 1.99 & 1.19 & 2.23 & 1.7\\
 & 1.0 & 2.28 & 1.36 & 2.57 & 1.95\\
\hline
{\bf Mix2} & 0.5 & 4.76 & 2.73 & 5.34 & 4.16\\
 & 1.0 & 5.44 & 3.13 & 6.08 & 4.73\\
 \hline
{\bf Unif[1,10]} & 0.5 & 4.4 & 2.48 & 4.96 & 3.79\\
 & 1.0 & 5.04 & 2.72 & 5.7 & 4.36\\
\hline
{\bf PaddedIdentity} & 0.5 & 1.0 & 0.0 & 1.62 & 1.11\\
 & 1.0 & 1.0 & 0.0 & 1.73 & 1.15 \\
 \hline
\end{tabular}
\end{center}

\subsection{Description of experiments}

For each noise distribution we conduct the following experiments:

\begin{itemize}
    \item {\bf Hist}: 
    We give a histogram singular values of a large noise matrix, in the case where the description of $\FZ$ is not trivial. We do this to get a some sense for how the noise bulk looks like. Specifically, the noise matrix we sample has dimensions $p=3000$ and $n=p/\gamma$.   

    \item {\bf R0-vs-R1}: 
    We compare the functions $R0(x)$ and $R1(x)$ from Lemma~\ref{lem:crossing}, and demonstrate that $x^*=\m{Y}^{-1}\left( T_\gamma(\FZ) \right)$ is their unique crossing point $x>\BBP$; this is done in the following way: we consider a rank $1$ signal $X=x\V{a}\V{b}^\T$ and $Y=X+Z$, where the signal directions $\V{a},\V{b}$ and noise matrix $Z$ are sampled once, and we let the intensity $x$ vary. The dimensions used are $p=500$ and $n=p/\gamma$. We plot the quantities $R_0(x)=\|X\|_F^2=x^2$ (the error of the estimator $\hat{X}=0$) and $\widehat{R}_1(x)=\|X-y_1\V{u}\V{v}^\T\|_F^2$, the error obtained by for estimating $X$ using the principal component of $Y$. The theory states that $\widehat{R}_1(x)$ converges to converges to $R_1(x)$, and we also plot this, as well as show that $x^*$ is the unique intersection point of $R_0(x)$ and $R_1(x)$. Thresholding at the intersection of $\overline{R}_1(x)$ and $R_0(x)$ is optimal (for this problem instance), and we see that indeed this intersection point is quite close to $x^*$.
    
    % While, in principal, it is possible to exactly compute $\FZ$ (and hence $\bulkEdge(\FZ)$, $\BBP$ and the mappings $\m{Y}(x)$, $\m{C}(x)$), in practice it is very difficult to obtain closed form expressions. What we instead do is sample a large dimensional noise matrix $Z'$, and estimate all the oracle quantities by replacing $\FZ$ with $F_n^{Z'}$. Estimating $\BBP$ this way is especially difficult, since for any counting measure, in particular $F_n^{Z'}$, $\m{D}_{\gamma}(y;F_n^{Z'})$ has a $1/y$ singularity at $\bulkEdge(F_n^{Z'})$. As a heuristic, we estimate $\BBP$ by $1/\hat{x}_{+}^2 = \m{D}_{\gamma}\left( \bulkEdge(F_n^{Z'})+0.01 \right)$. For the same reason, we expect to see rather strong discrepancy between $R_1(x)$ and $\widehat{R}_1(x)$ for $x$ close $\BBP$, and this is indeed seen in the results. All the experiments use $p=500$ (with $n=p/\gamma$) for the signal $X$, and $p'=3000$ for the noise matrix $Z'$ used to approximate the oracle quantities.
    
    \item 
    {\bf SE-vs-ASE}: 
    Our theory states that as $n,p$ grow, the random function $\thresh\mapsto \SEn[\V{x}|\thresh]$ tends to the deterministic function $\thresh\mapsto \ASE[\V{x}|\thresh]$, when $\thresh>\bulkEdge(\FZ)$. We illustrate this phenomenon. We consider a single problem instance, corresponding to the rank $r=5$ signal $\V{x}=(0.5,1.0,1.3,2.5,5.2)$, and plot the function $\SEn[\V{x}|\thresh]$ and $\ASE[\V{x}|\thresh]$ on top of each other (we plot $\SEn[\V{x}|\thresh]$ for very few thresholds $\thresh\le y_{r+1,n}$; as the rank of $X_t$ grows, the MSE blows up quickly). We use dimensions $p=500$ and $n=p/\gamma$. We indicate the locations of $\bulkedge$, $\optThresh(\FZ)$ and the estimates $T_{p/n}(F_{n,k}^\star)$ for $\star\in \{ 0,w,i\}$, where we use $k=4r=20$. We see that for a ``typical'' problem instance, $\ASE[\V{x}|\thresh]$ is indeed a good proxy for $\SEn[\V{x}|\thresh]$.

    \item {\bf OracleAttainment}: 
    We test Theorems~\ref{thm:oracle-risk-attained} and \ref{thm:adaptive-guarantee}, whereby the probability that thresholding at $T_{p/n}(F_{n,k}^\star)$ attains oracle loss with probability tending to $1$ as $n,p\to \infty$. For various choices of $p$, we run $T=50$ denoising experiments, and report the fraction of experiments where each threshold $\in \left\{ T_{\gamma}(\FZ), T_{p/n}(F_{n,k}^0),T_{p/n}(F_{n,k}^w),T_{p/n}(F_{n,k}^i) \right\}$ attains oracle loss ($T_\gamma(\FZ)$ is computed as described in experiment {\bf R0-vs-R1}). In all experiments, we use the signal $\V{x}=(0.5,1.0, 1.3, 2.5, 5.2)$ which has rank $r=5$ (the same signal as in experiment {\bf SE-vs-ASE}), and the upper bound $k=4r=20$. Note that we have chosen the spikes $x_i$ to be all far from $x^*=\m{Y}^{-1}(\optThresh(\FZ))$, in accordance with the condition in Theorem~\ref{thm:oracle-risk-attained}.
    
    \item {\bf Regret}:
    We compare the oracle loss $\SEn^*[\V{x}]$ with $\SEn[x|\hat{\theta}]$ for the choices 
    \[\hat{\thresh}\in \left\{ \optThresh(\FZ),T_{p/n}(F_{n,k}^0),T_{p/n}(F_{n,k}^w),T_{p/n}(F_{n,k}^i) \right\}\,,\]
     in the single-spiked setup, as described in experiment {\bf R0-vs-R1} above. We let the spike intensity $x$ vary and plot $\SEn[\V{x}|\hat{\thresh}]/\SEn^*[\V{x}]$ for each choice of estimator. We expect the choice of estimator to especially matter when $x$ is close to $x^*=\m{Y}^{-1}(T_\gamma(\FZ))$ (indicated in the plots by a dashed vertial line), and this can indeed be seen in the plots. The ratios we report are averages across $T=20$ experiments.
    
    \item {\bf ConvergenceRate}: 
    We study the rate of convergence of $T_{p/n}(F_{n,k}^\star)$ towards $\optThresh(\FZ)$. To that end, we consider a rank $r=10$ signal $x=(1,\ldots,10)$, set $k=20$ and plots the relative absolute error $\left| T_{p/n}(F^\star_{n,k})-\optThresh(\FZ)\right|/\optThresh(\FZ)$ as $p$ varies and $n=p/\gamma$. We plot the error in a logarithmic scale, to get a sense of its polynomial rate of decay in $p$. Note that by Proposition~\ref{prop:main:perturb}, we expect the slope in most cases to be, roughly, $\lesssim 1$. We find that almost always, $T_{p/n}(F_{n,k}^i)$ (imputation) approximates $\optThresh(\FZ)$ much better than $T_{p/n}(F_{n,k}^w)$ (winsorization) or $T_{p/n}(F_{n,k}^0)$ (transport to zero). All points on the plots are generated by averaging the error of $T=50$ experiments.
    
\end{itemize}   

%\subsection{Numerical results}

\subsection{Distribution: Mar\v{c}enko-Pastur, $\gamma = 1.0$}

\begin{figure}[H]
    \centering
    \includegraphics[width=0.75\textwidth]{{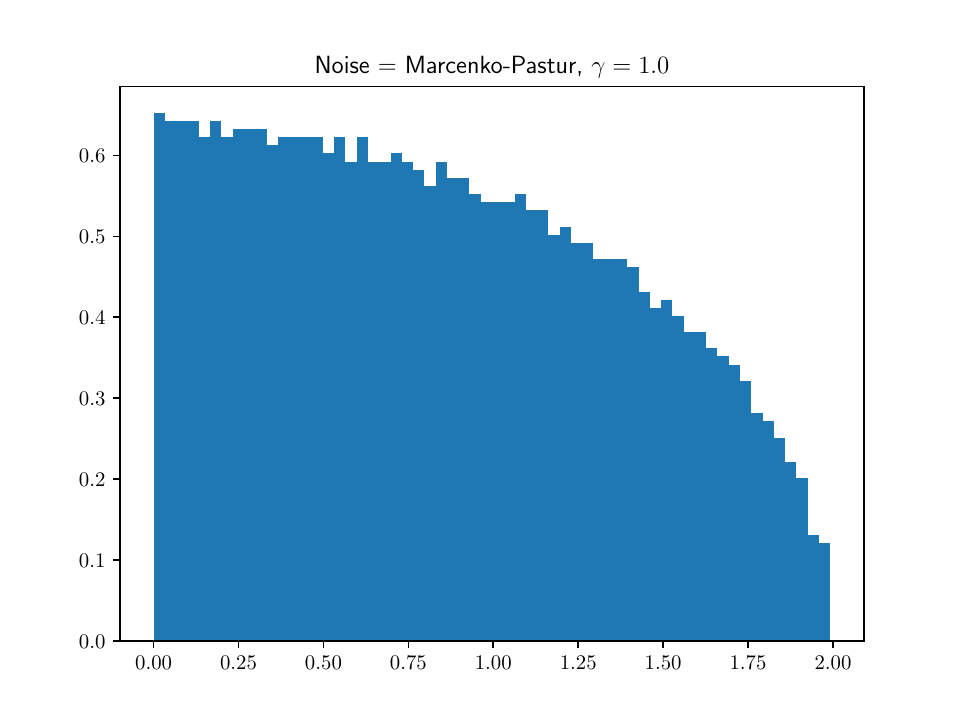}}
    \caption{Experiment: {\bf Hist}}
\end{figure}

\begin{figure}[H]
    \centering
    \includegraphics[width=0.75\textwidth]{{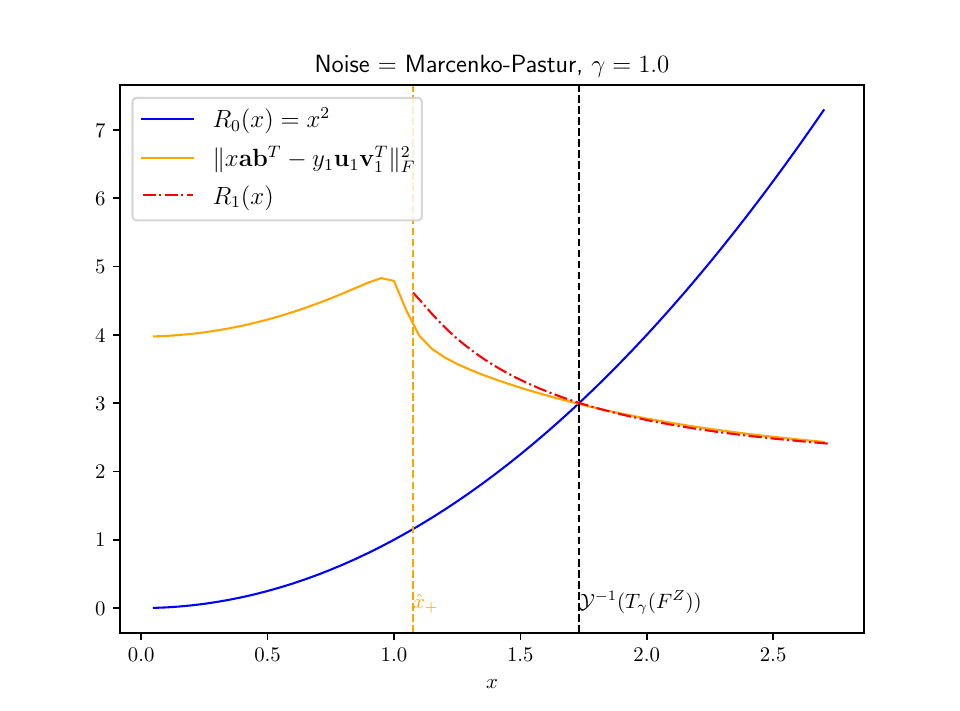}}
    \caption{Experiment: {\bf R0-vs-R1}}
\end{figure}

\begin{figure}[H]
    \centering
    \includegraphics[width=0.75\textwidth]{{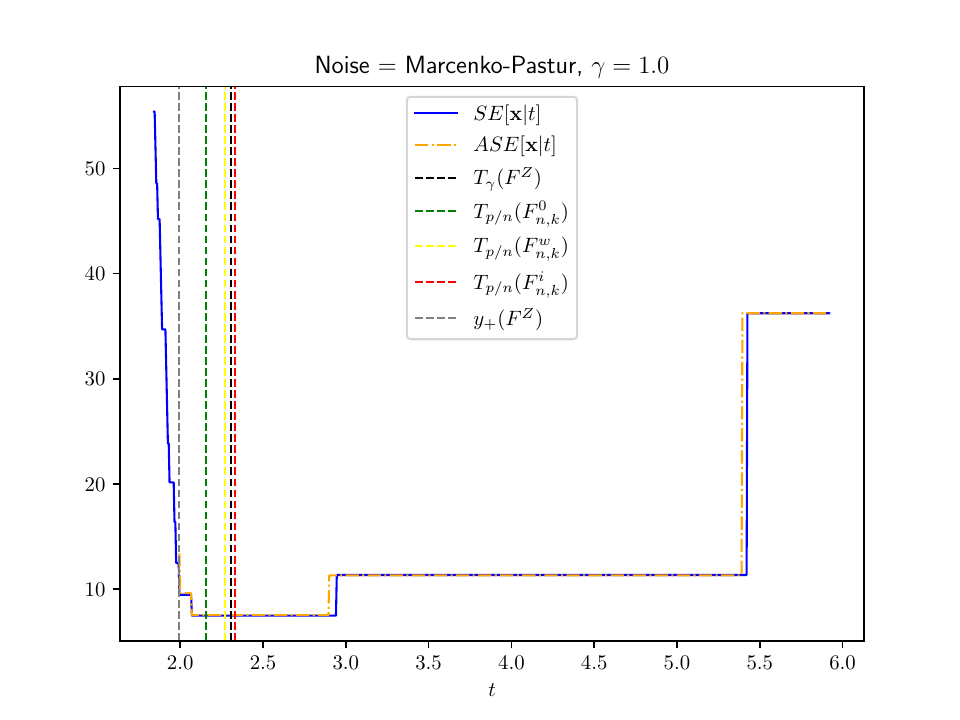}}
    \caption{Experiment: {\bf SE-vs-ASE}}
\end{figure}

\begin{figure}[H]
    \centering
    \includegraphics[width=0.75\textwidth]{{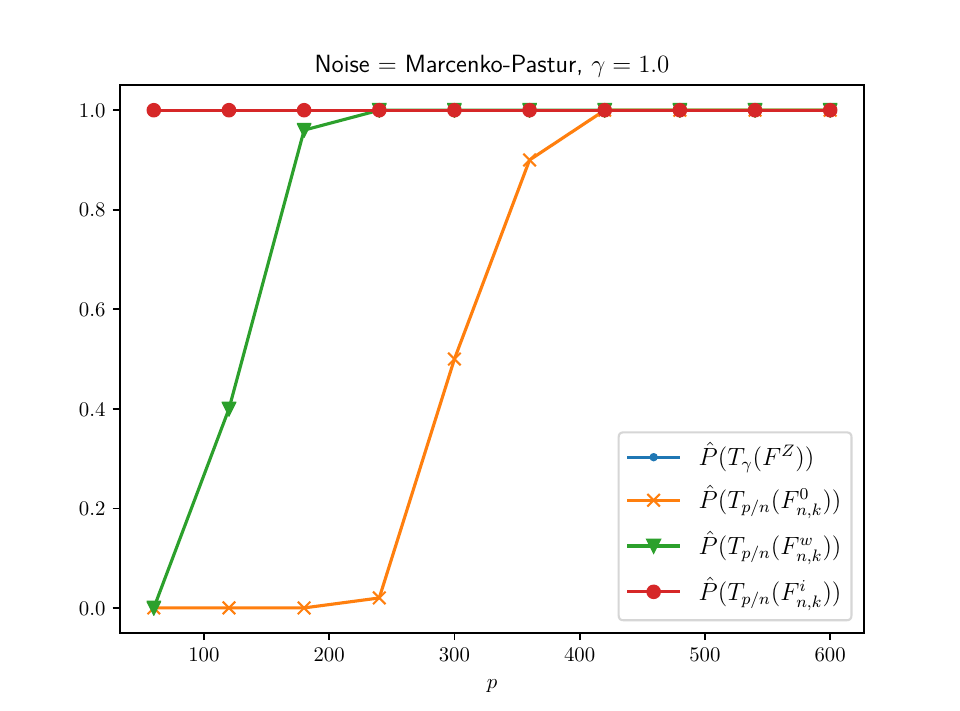}}
    \caption{Experiment: {\bf OracleAttainment}}
\end{figure}

\begin{figure}[H]
    \centering
    \includegraphics[width=0.75\textwidth]{{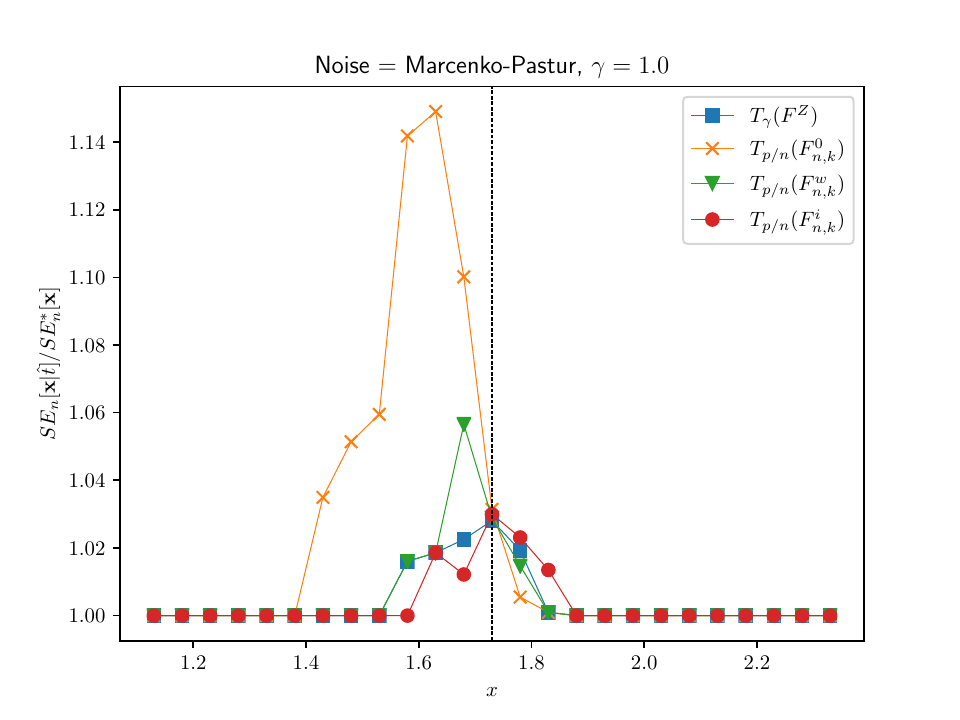}}
    \caption{Experiment: {\bf Regret}}
\end{figure}

\begin{figure}[H]
    \centering
    \includegraphics[width=0.75\textwidth]{{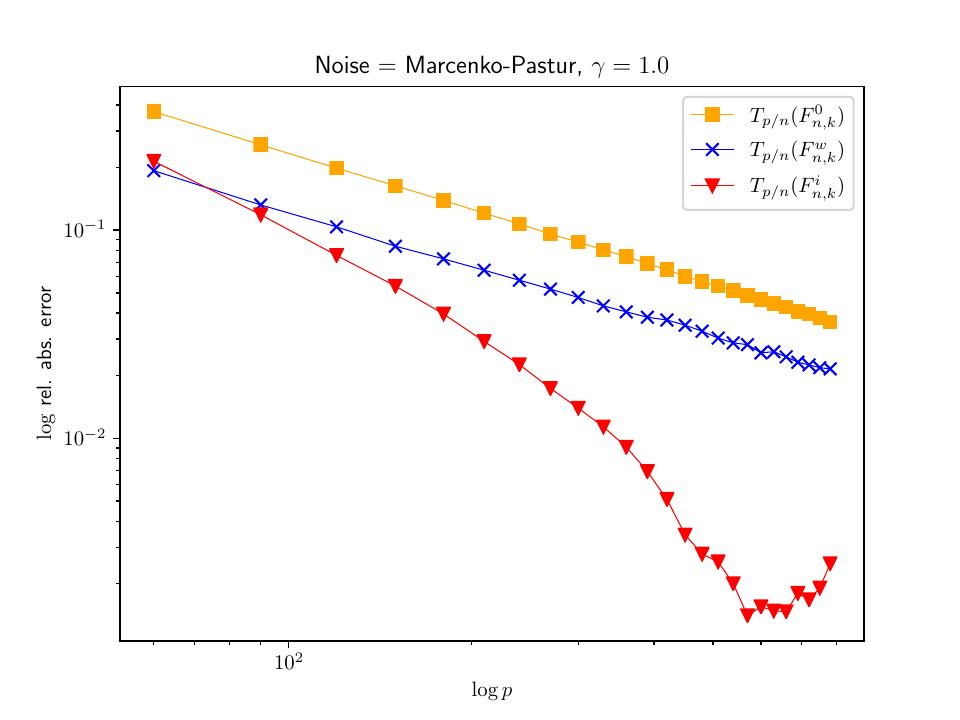}}
    \caption{Experiment: {\bf ConvergenceRate}}
\end{figure}

\subsection{Distribution: Chi10, $\gamma = 0.5$}

\begin{figure}[H]
    \centering
    \includegraphics[width=0.75\textwidth]{{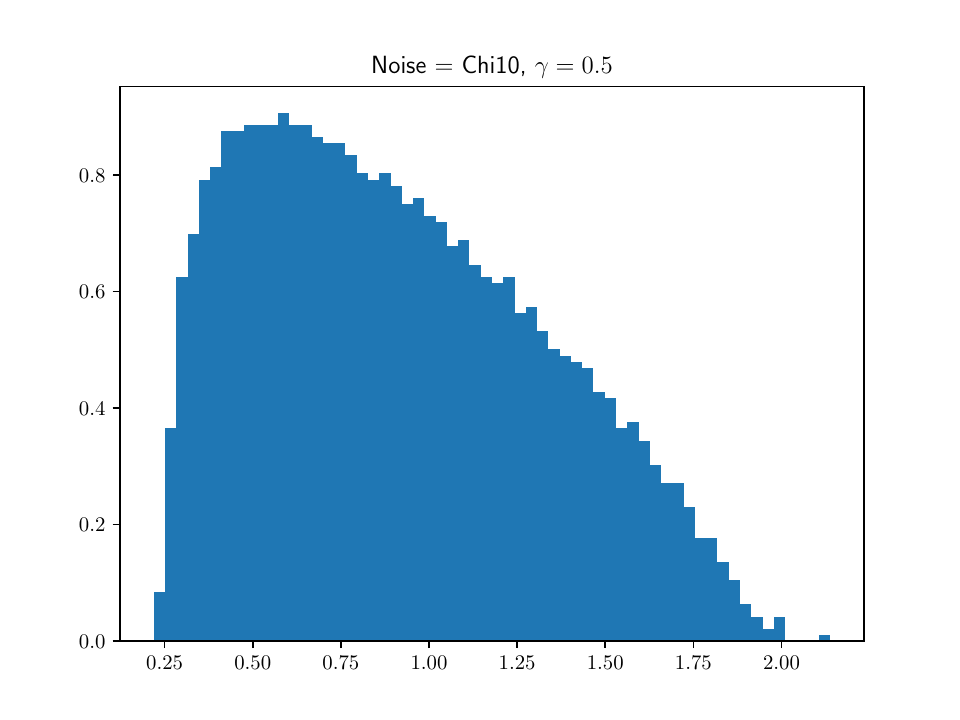}}
    \caption{Experiment: {\bf Hist}}
\end{figure}

\begin{figure}[H]
    \centering
    \includegraphics[width=0.75\textwidth]{{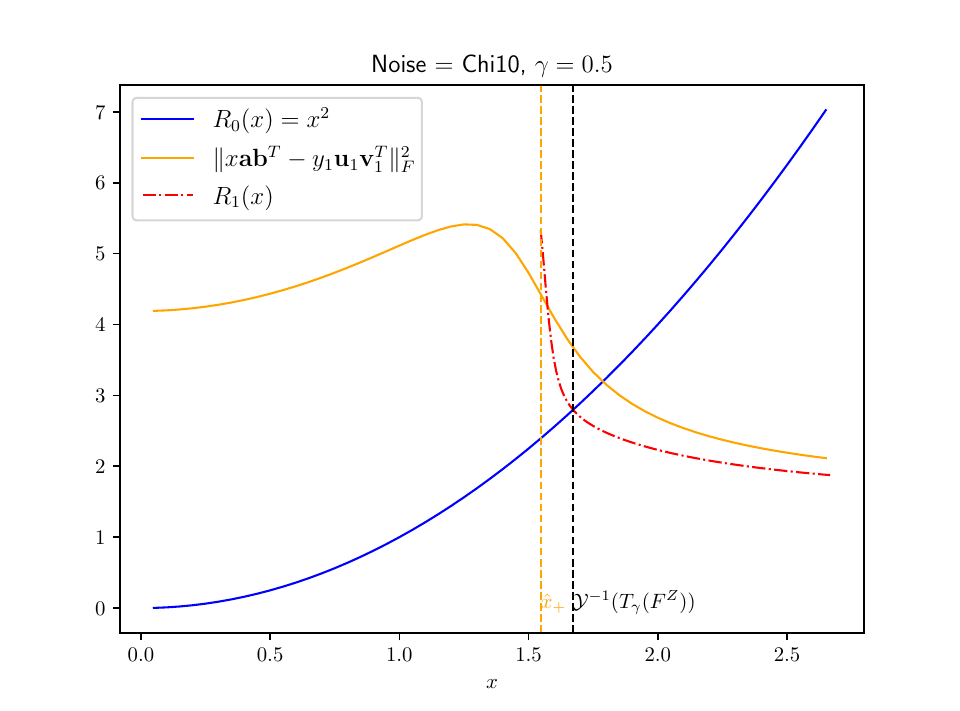}}
    \caption{Experiment: {\bf R0-vs-R1}}
\end{figure}

\begin{figure}[H]
    \centering
    \includegraphics[width=0.75\textwidth]{{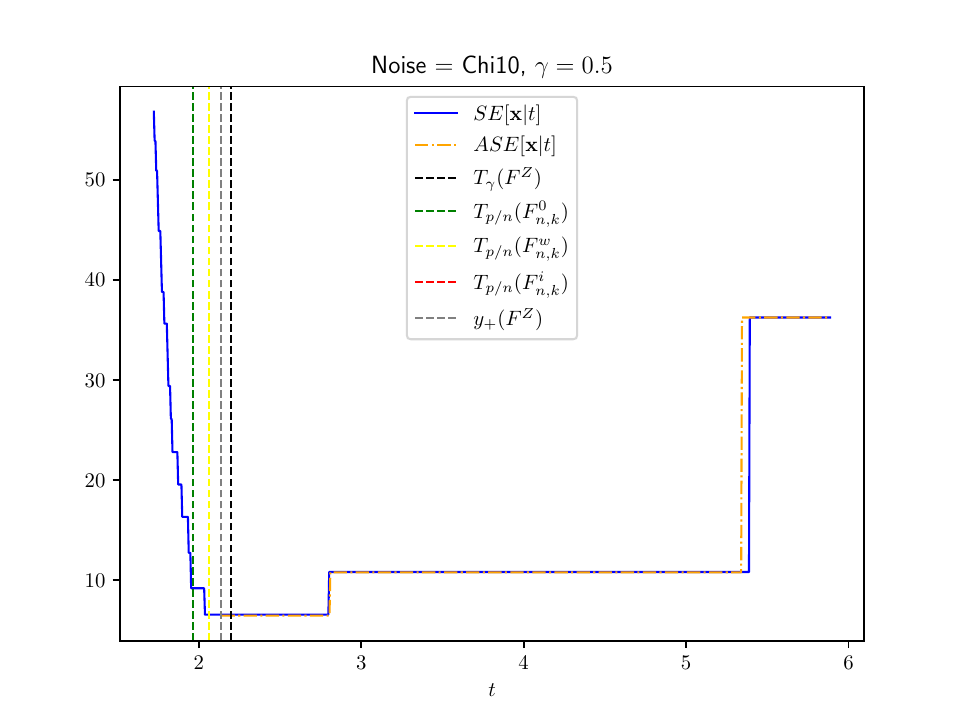}}
    \caption{Experiment: {\bf SE-vs-ASE}}
\end{figure}

\begin{figure}[H]
    \centering
    \includegraphics[width=0.75\textwidth]{{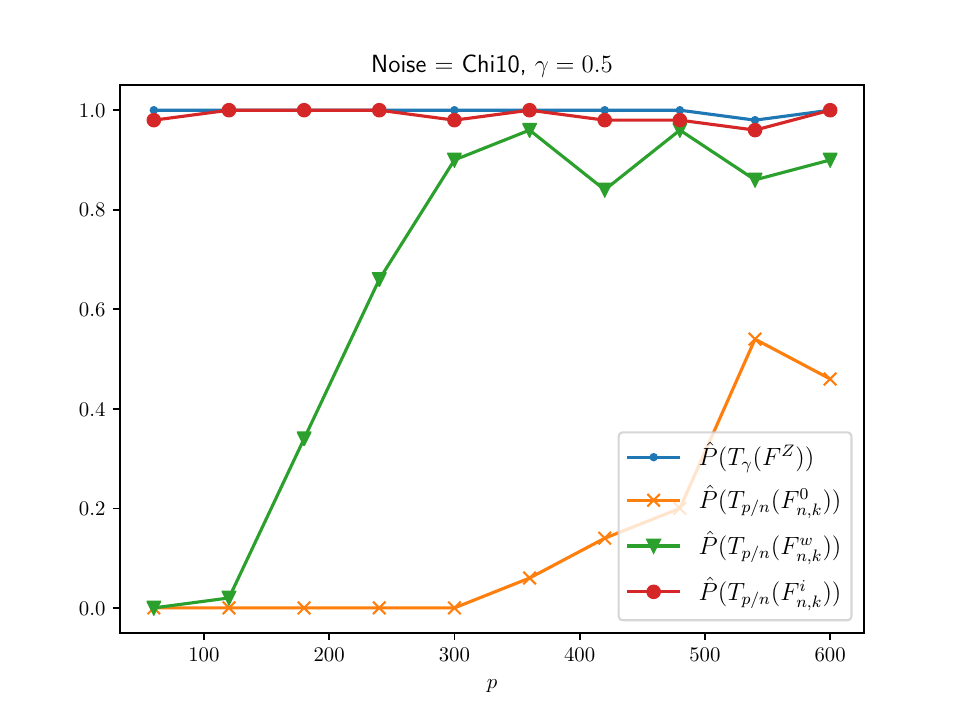}}
    \caption{Experiment: {\bf OracleAttainment}}
\end{figure}

\begin{figure}[H]
    \centering
    \includegraphics[width=0.75\textwidth]{{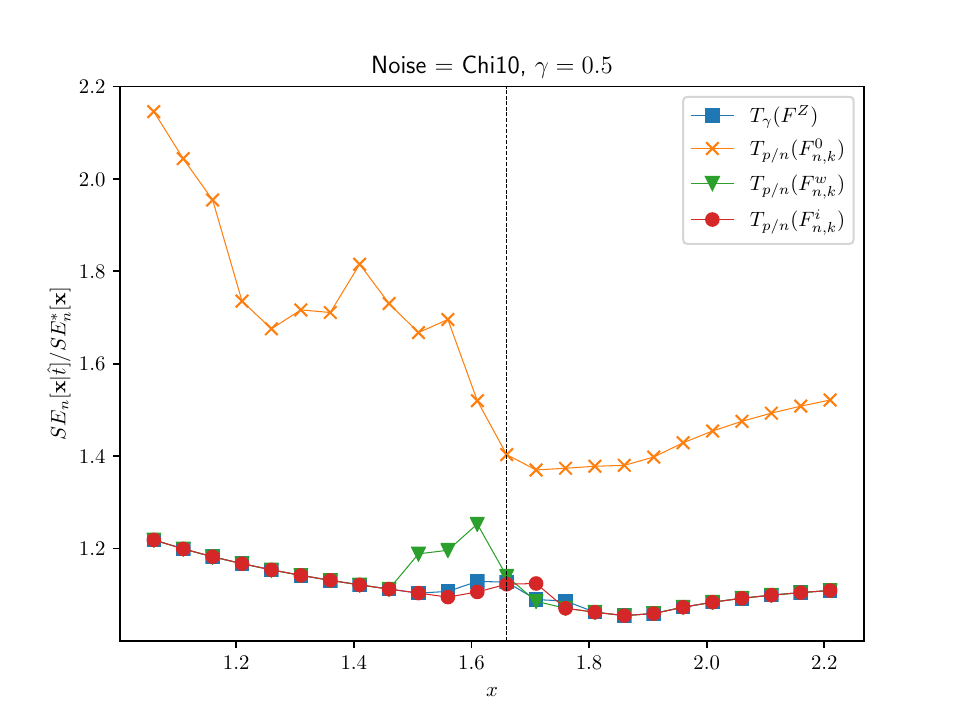}}
    \caption{Experiment: {\bf Regret}}
\end{figure}

\begin{figure}[H]
    \centering
    \includegraphics[width=0.75\textwidth]{{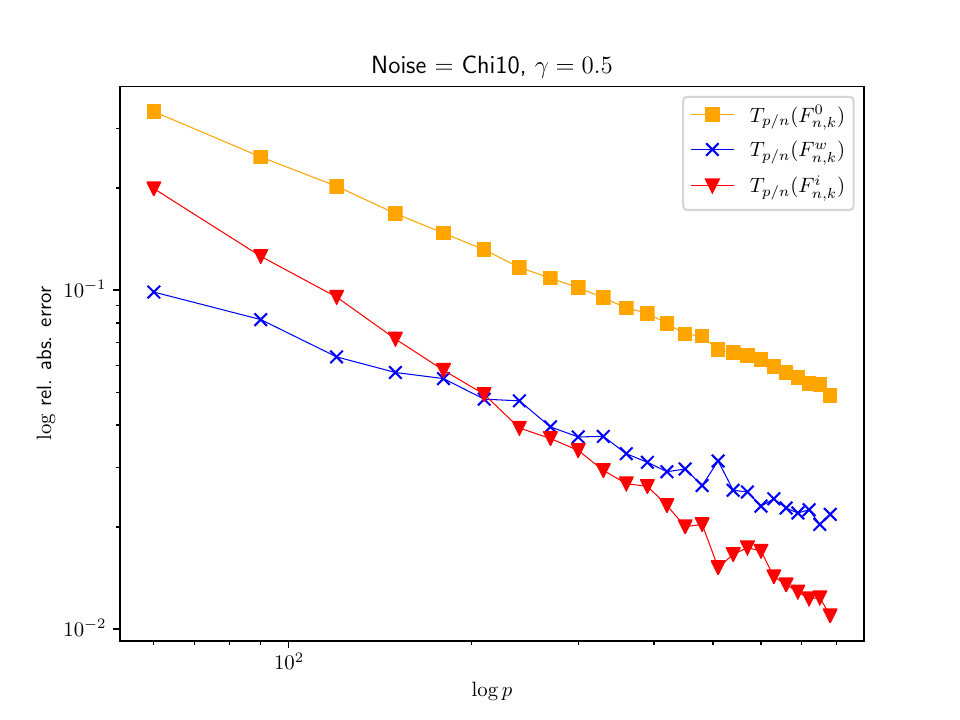}}
    \caption{Experiment: {\bf ConvergenceRate}}
\end{figure}

\subsection{Distribution: Chi10, $\gamma = 1.0$}

\begin{figure}[H]
    \centering
    \includegraphics[width=0.75\textwidth]{{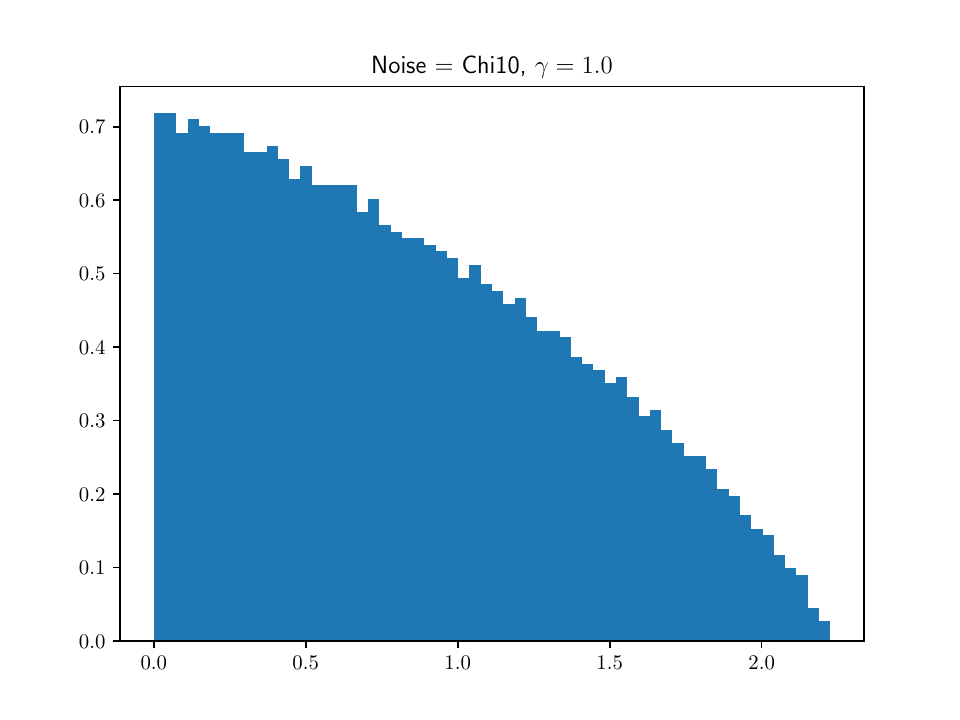}}
    \caption{Experiment: {\bf Hist}}
\end{figure}

\begin{figure}[H]
    \centering
    \includegraphics[width=0.75\textwidth]{{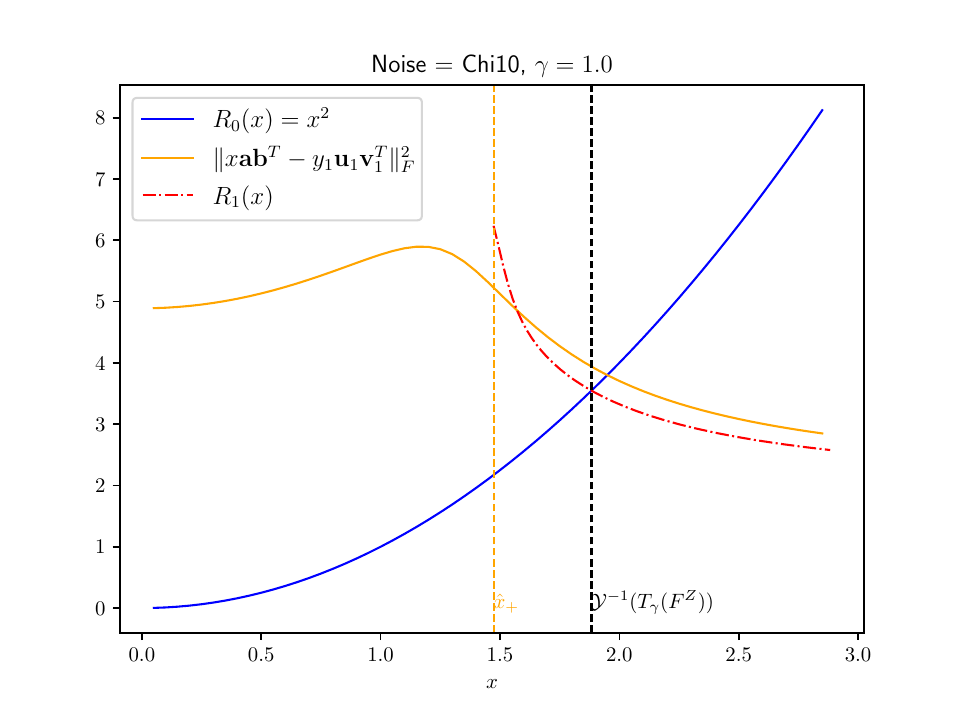}}
    \caption{Experiment: {\bf R0-vs-R1}}
\end{figure}

\begin{figure}[H]
    \centering
    \includegraphics[width=0.75\textwidth]{{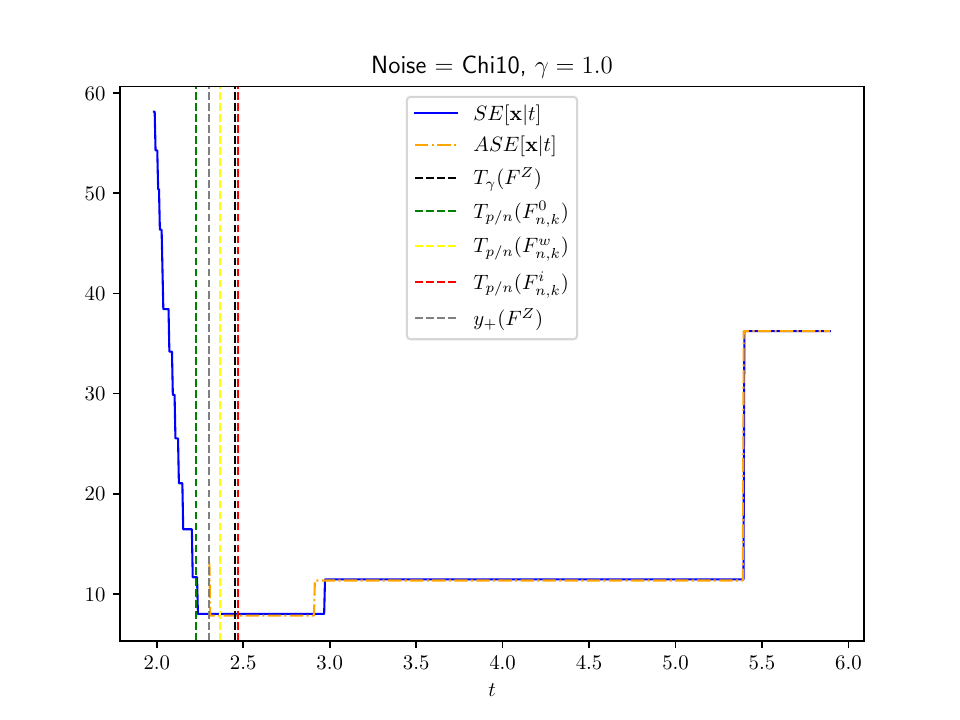}}
    \caption{Experiment: {\bf SE-vs-ASE}}
\end{figure}

\begin{figure}[H]
    \centering
    \includegraphics[width=0.75\textwidth]{{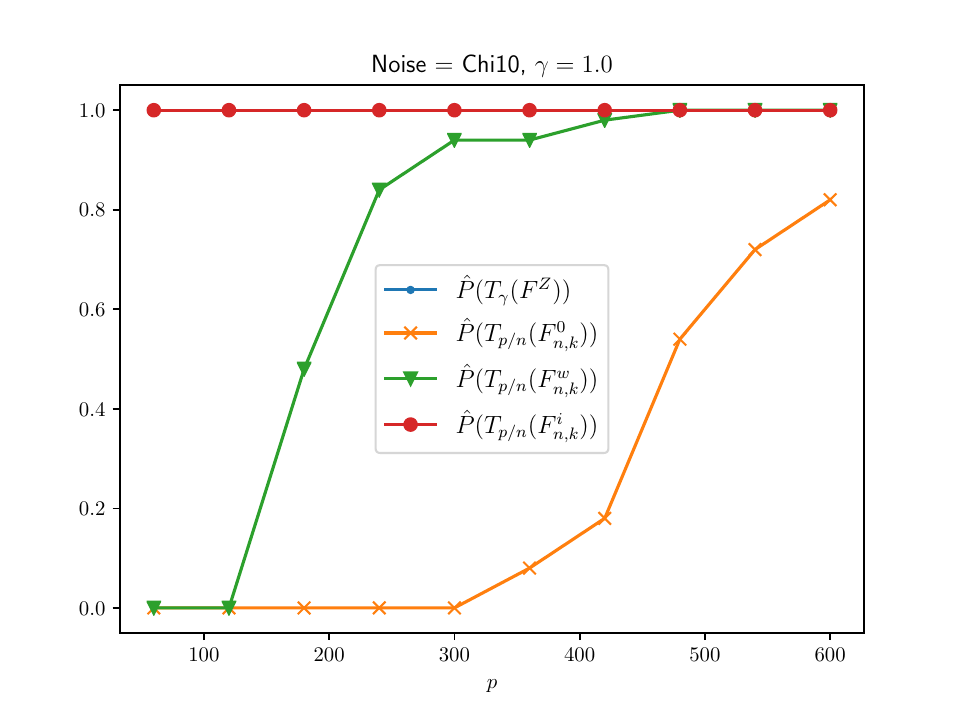}}
    \caption{Experiment: {\bf OracleAttainment}}
\end{figure}

\begin{figure}[H]
    \centering
    \includegraphics[width=0.75\textwidth]{{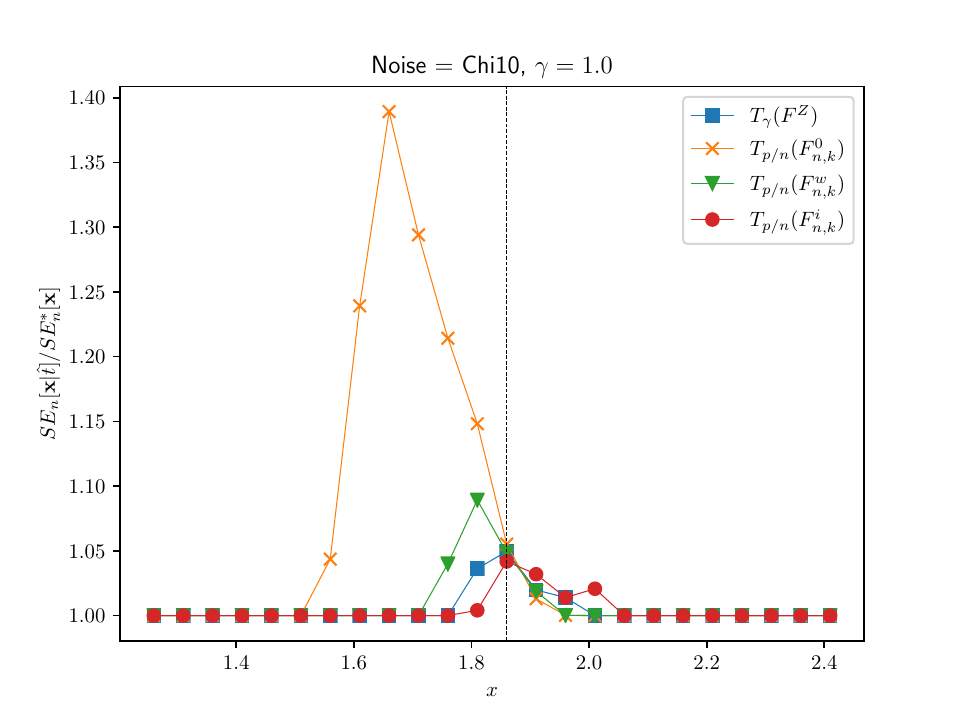}}
    \caption{Experiment: {\bf Regret}}
\end{figure}

\begin{figure}[H]
    \centering
    \includegraphics[width=0.75\textwidth]{{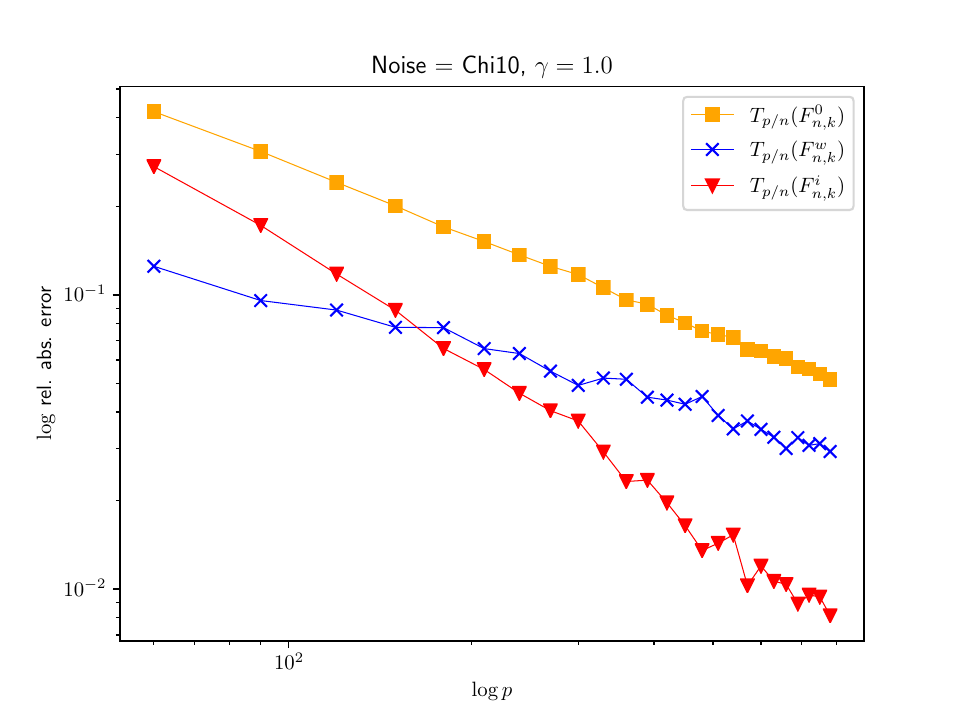}}
    \caption{Experiment: {\bf ConvergenceRate}}
\end{figure}

\subsection{Distribution: Fisher3n, $\gamma = 0.5$}

\begin{figure}[H]
    \centering
    \includegraphics[width=0.75\textwidth]{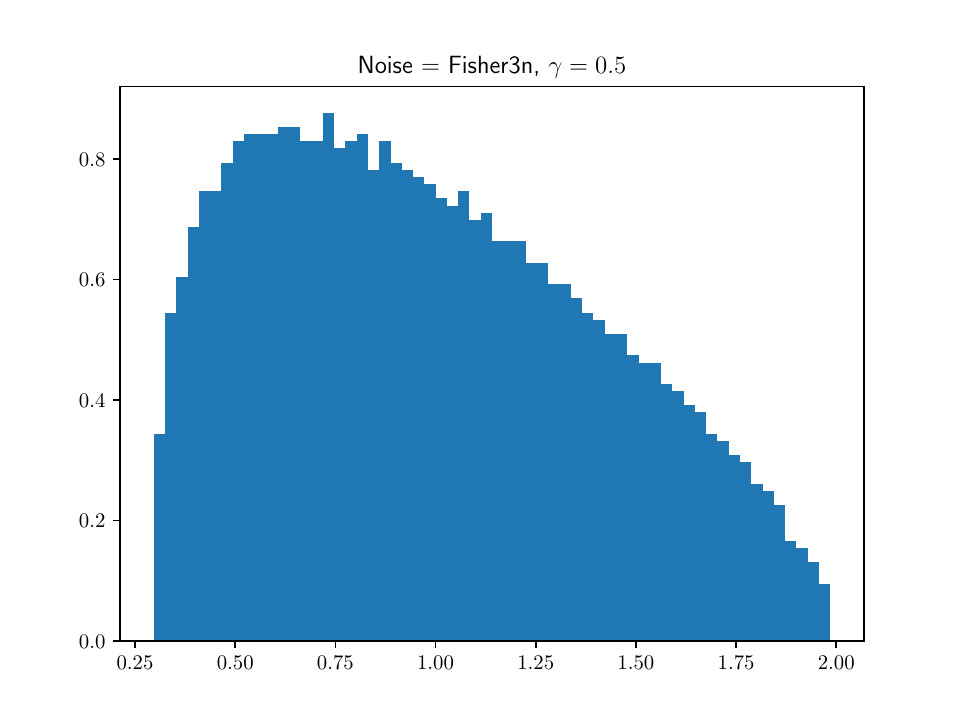}
    \caption{Experiment: {\bf Hist}}
\end{figure}

\begin{figure}[H]
    \centering
    \includegraphics[width=0.75\textwidth]{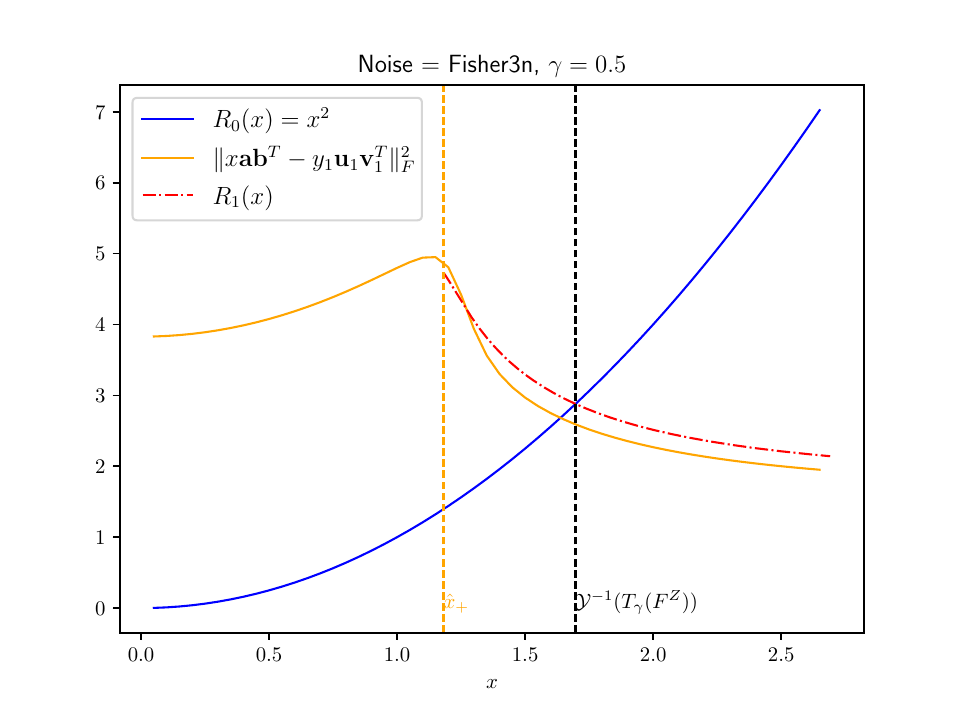}
    \caption{Experiment: {\bf R0-vs-R1}}
\end{figure}

\begin{figure}[H]
    \centering
    \includegraphics[width=0.75\textwidth]{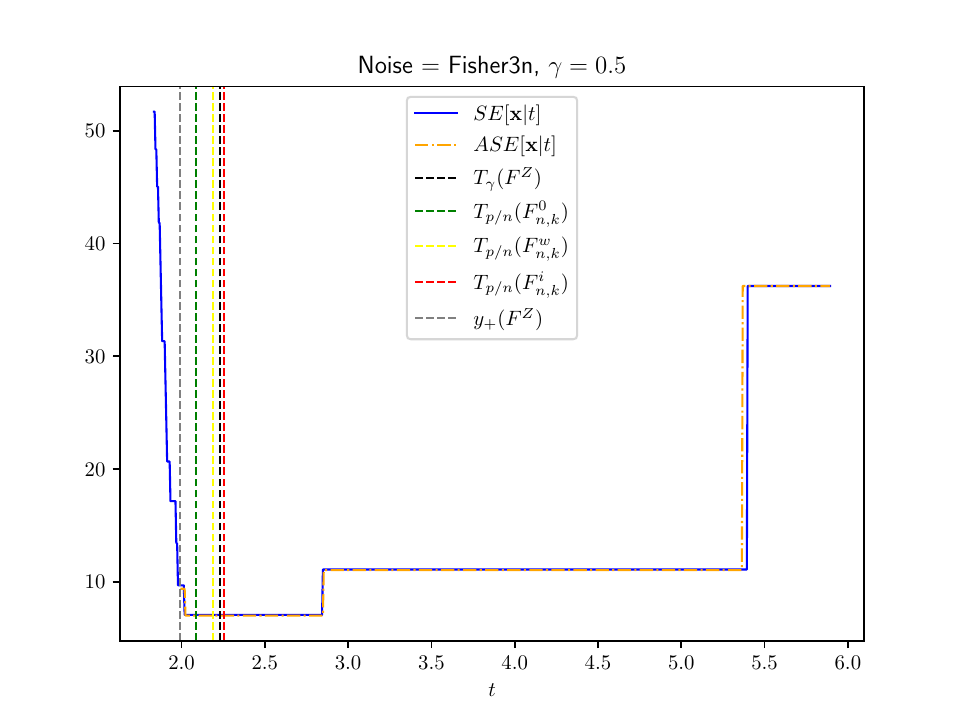}
    \caption{Experiment: {\bf SE-vs-ASE}}
\end{figure}

\begin{figure}[H]
    \centering
    \includegraphics[width=0.75\textwidth]{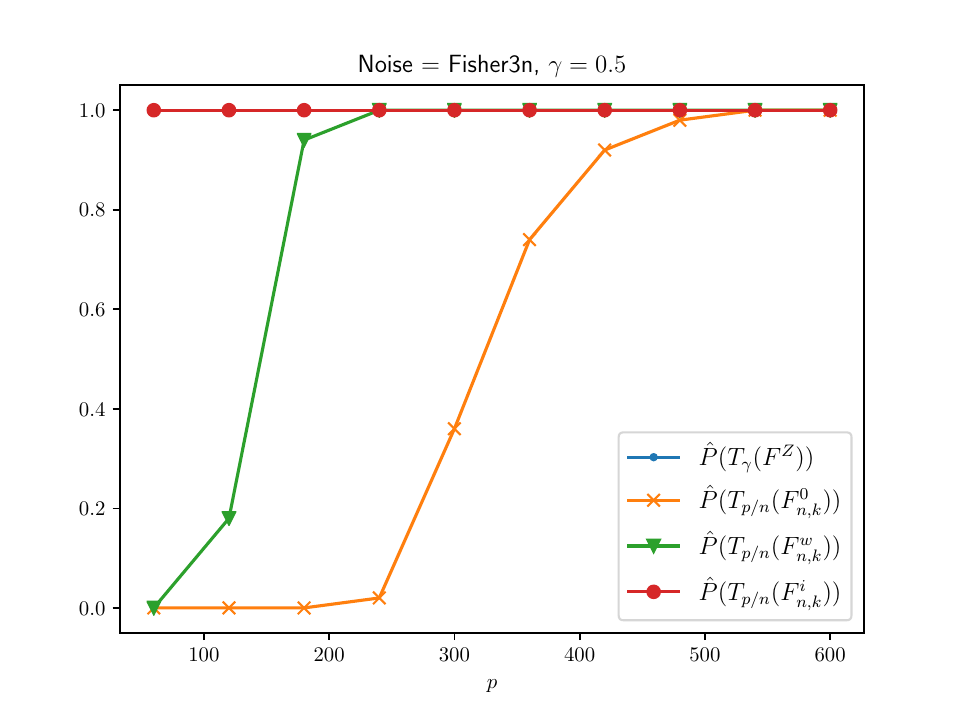}
    \caption{Experiment: {\bf OracleAttainment}}
\end{figure}

\begin{figure}[H]
    \centering
    \includegraphics[width=0.75\textwidth]{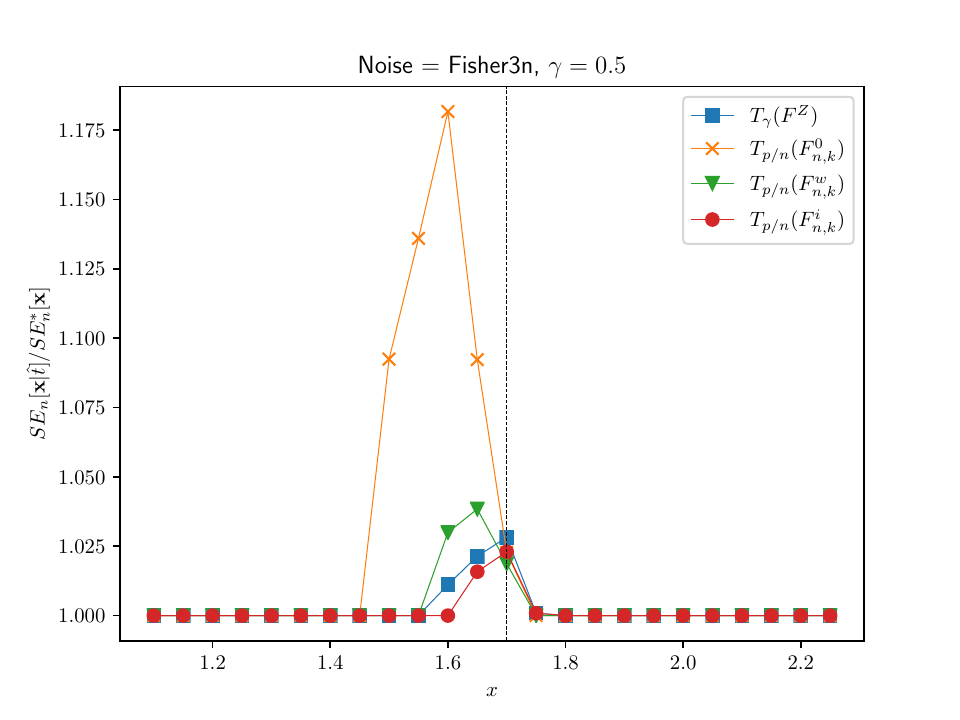}
    \caption{Experiment: {\bf Regret}}
\end{figure}

\begin{figure}[H]
    \centering
    \includegraphics[width=0.75\textwidth]{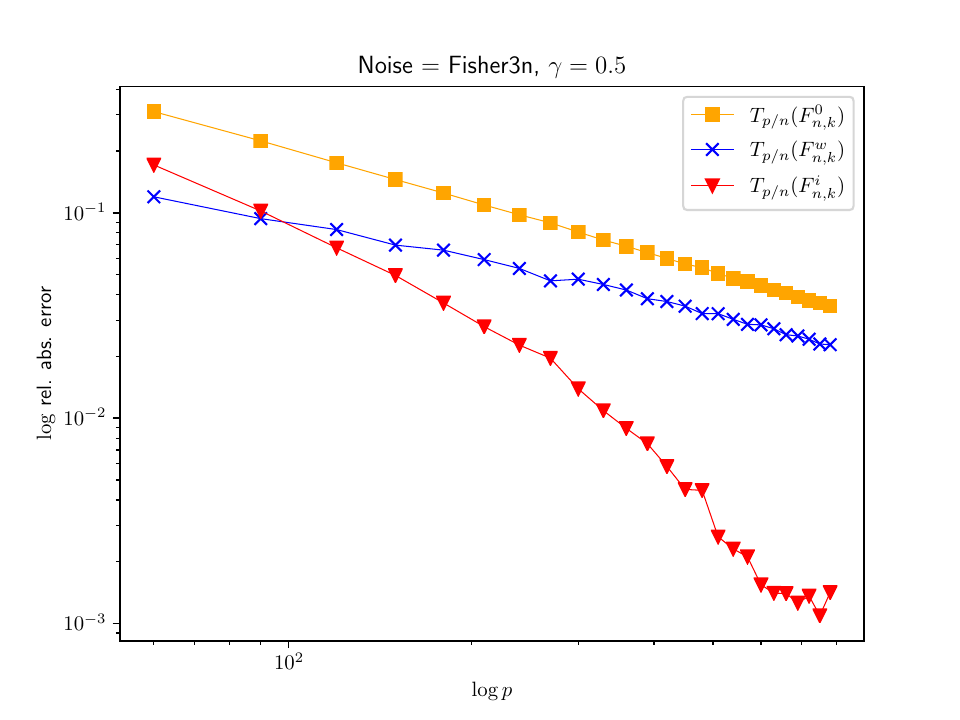}
    \caption{Experiment: {\bf ConvergenceRate}}
\end{figure}

\subsection{Distribution: Fisher3n, $\gamma = 1.0$}

\begin{figure}[H]
    \centering
    \includegraphics[width=0.75\textwidth]{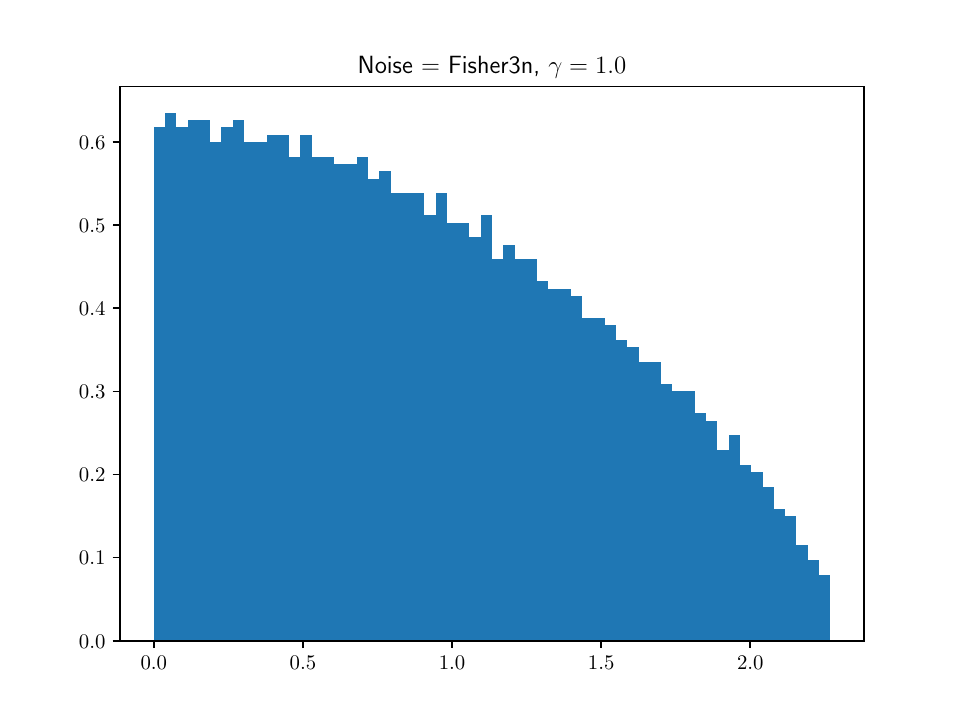}
    \caption{Experiment: {\bf Hist}}
\end{figure}

\begin{figure}[H]
    \centering
    \includegraphics[width=0.75\textwidth]{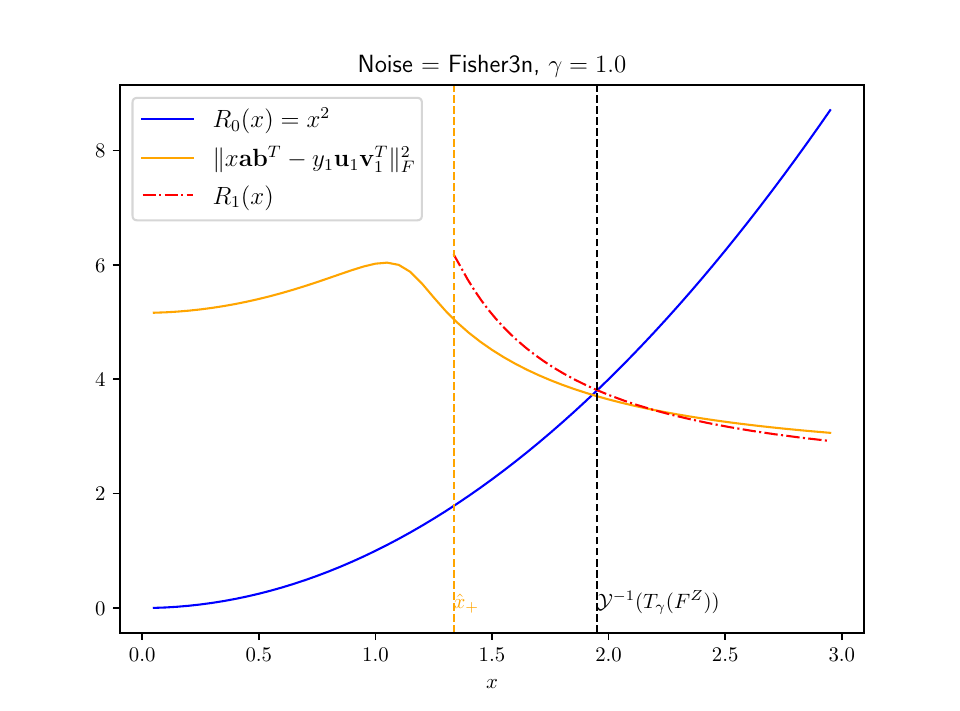}
    \caption{Experiment: {\bf R0-vs-R1}}
\end{figure}

\begin{figure}[H]
    \centering
    \includegraphics[width=0.75\textwidth]{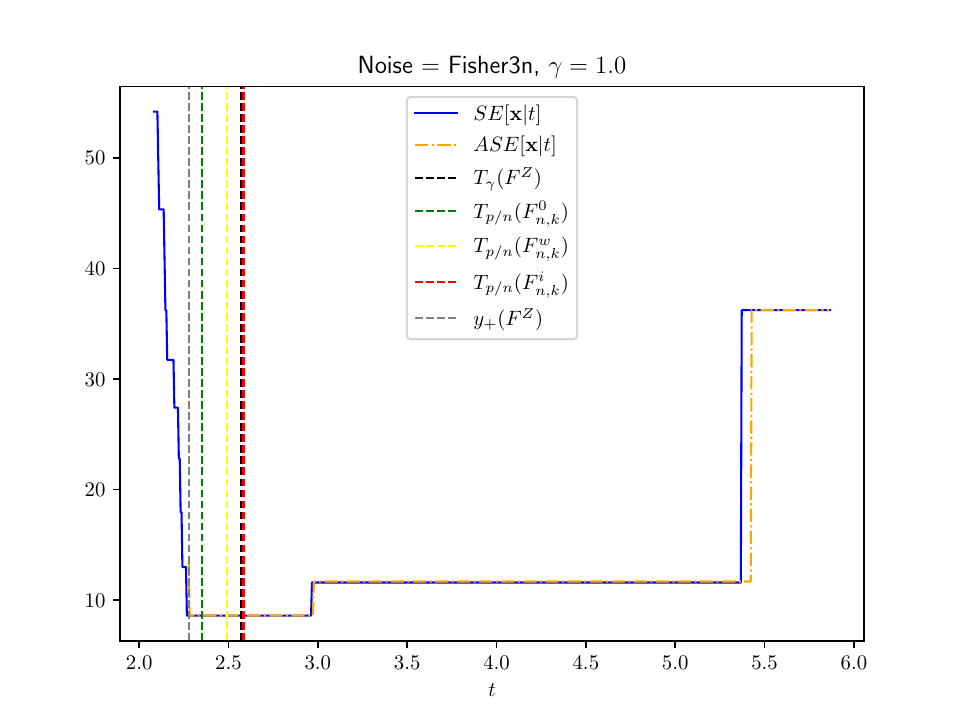}
    \caption{Experiment: {\bf SE-vs-ASE}}
\end{figure}

\begin{figure}[H]
    \centering
    \includegraphics[width=0.75\textwidth]{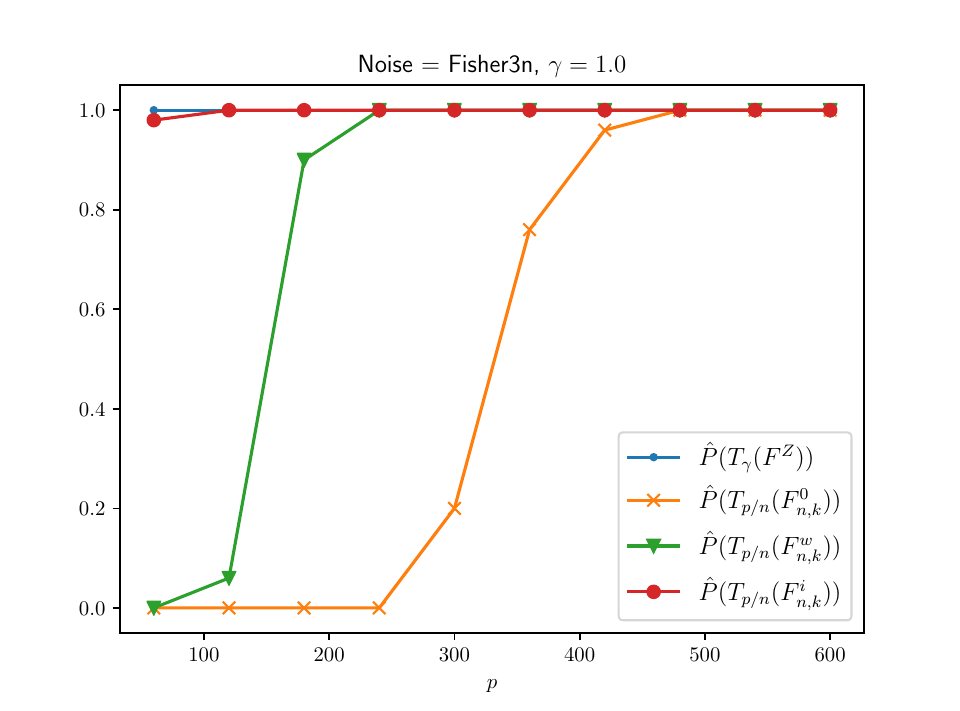}
    \caption{Experiment: {\bf OracleAttainment}}
\end{figure}

\begin{figure}[H]
    \centering
    \includegraphics[width=0.75\textwidth]{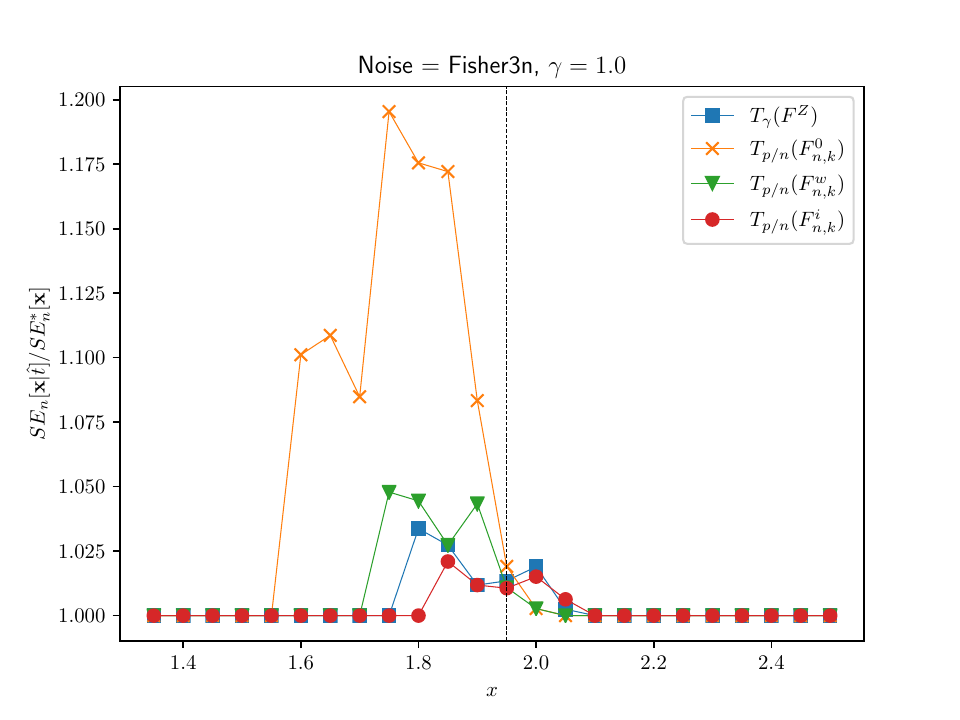}
    \caption{Experiment: {\bf Regret}}
\end{figure}

\begin{figure}[H]
    \centering
    \includegraphics[width=0.75\textwidth]{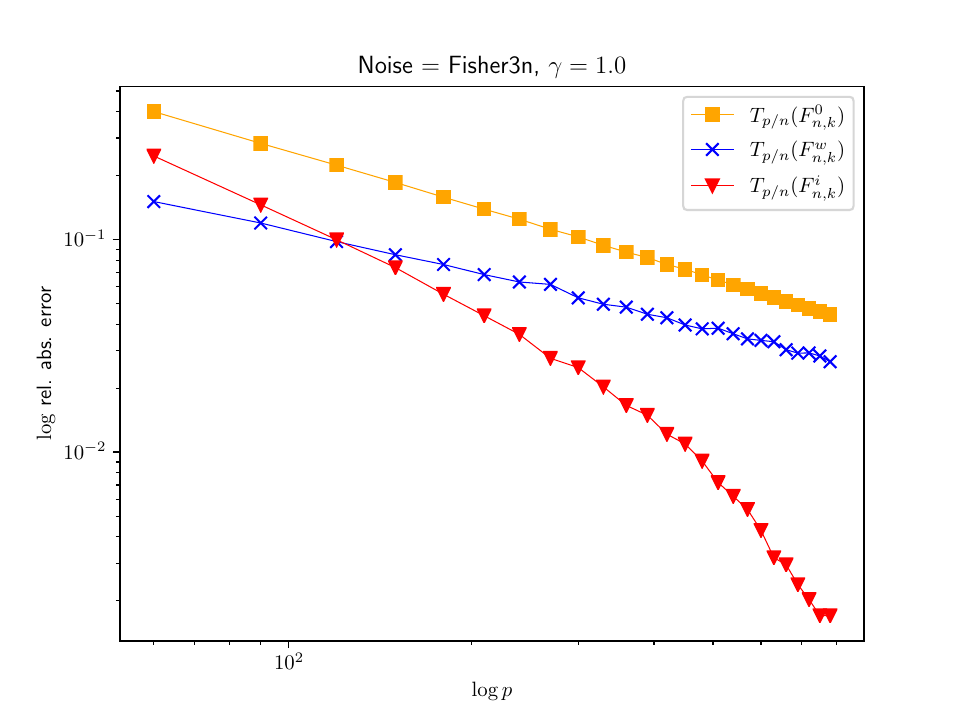}
    \caption{Experiment: {\bf ConvergenceRate}}
\end{figure}

\subsection{Distribution: Mix2, $\gamma = 0.5$}

\begin{figure}[H]
    \centering
    \includegraphics[width=0.75\textwidth]{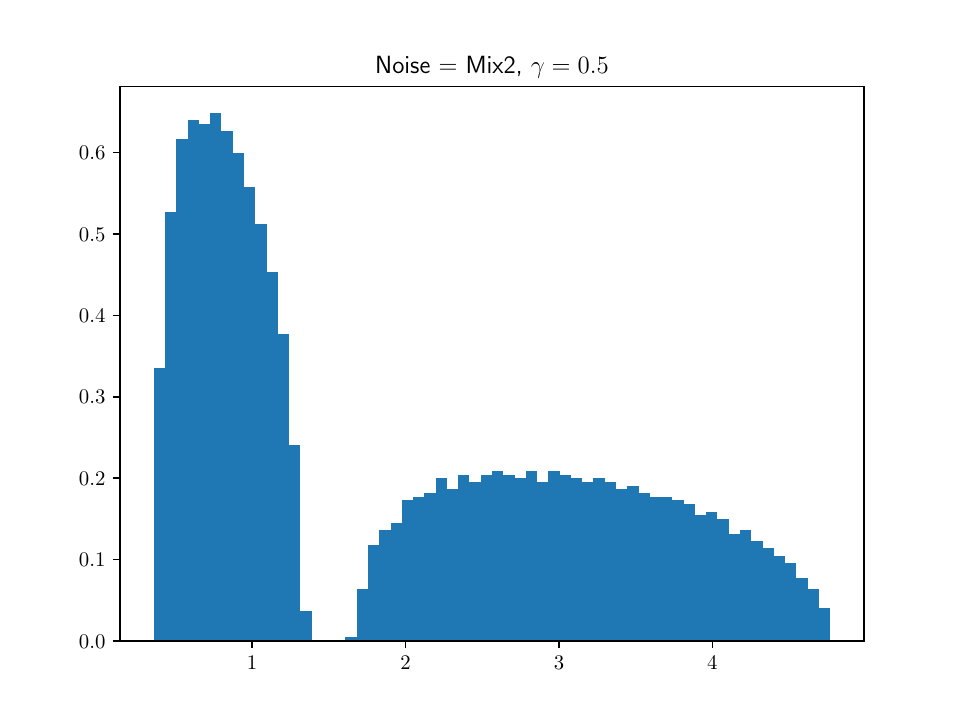}
    \caption{Experiment: {\bf Hist}}
\end{figure}

\begin{figure}[H]
    \centering
    \includegraphics[width=0.75\textwidth]{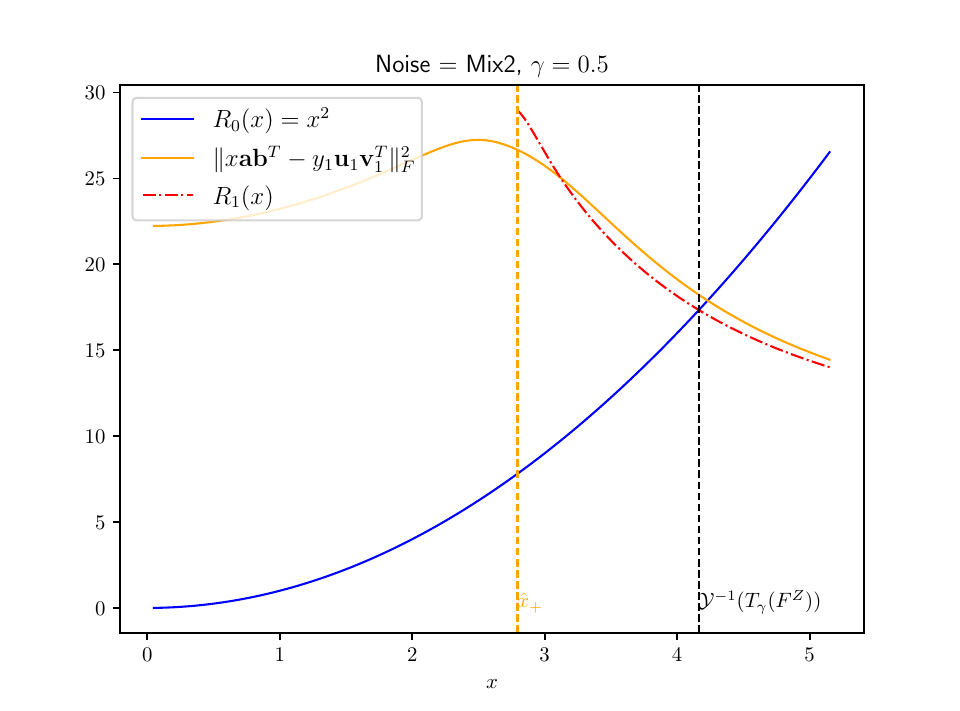}
    \caption{Experiment: {\bf R0-vs-R1}}
\end{figure}

\begin{figure}[H]
    \centering
    \includegraphics[width=0.75\textwidth]{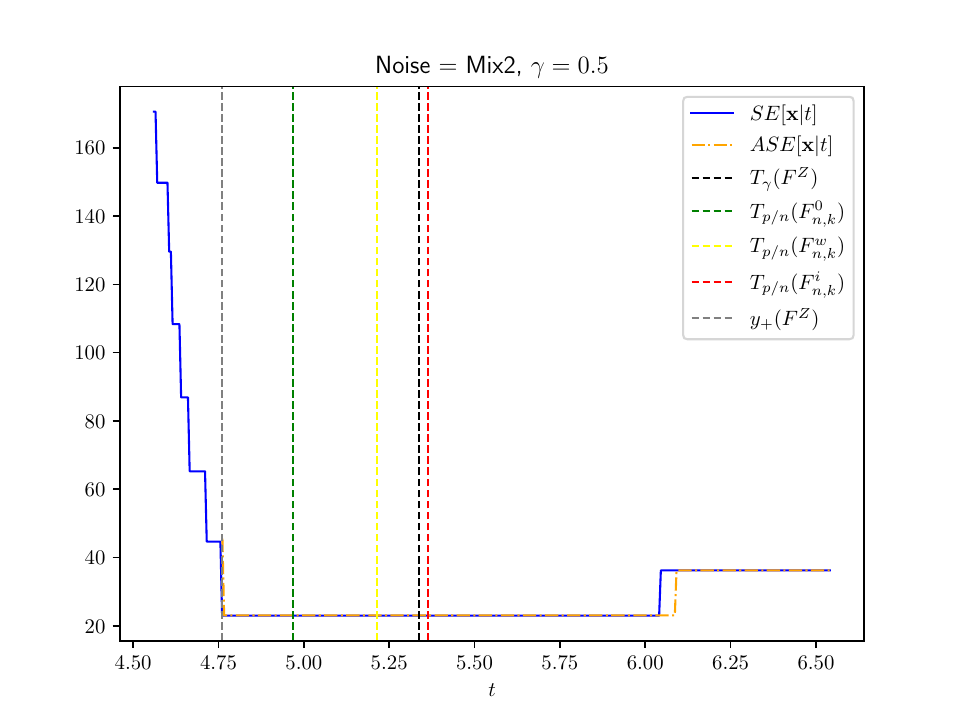}
    \caption{Experiment: {\bf SE-vs-ASE}}
\end{figure}

\begin{figure}[H]
    \centering
    \includegraphics[width=0.75\textwidth]{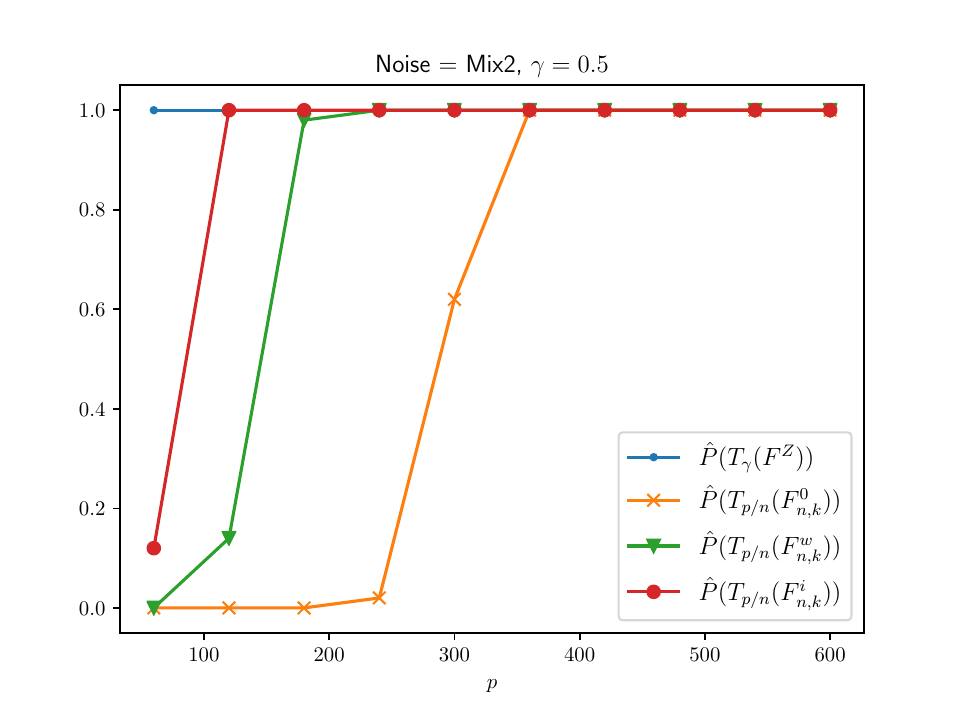}
    \caption{Experiment: {\bf OracleAttainment}}
\end{figure}

\begin{figure}[H]
    \centering
    \includegraphics[width=0.75\textwidth]{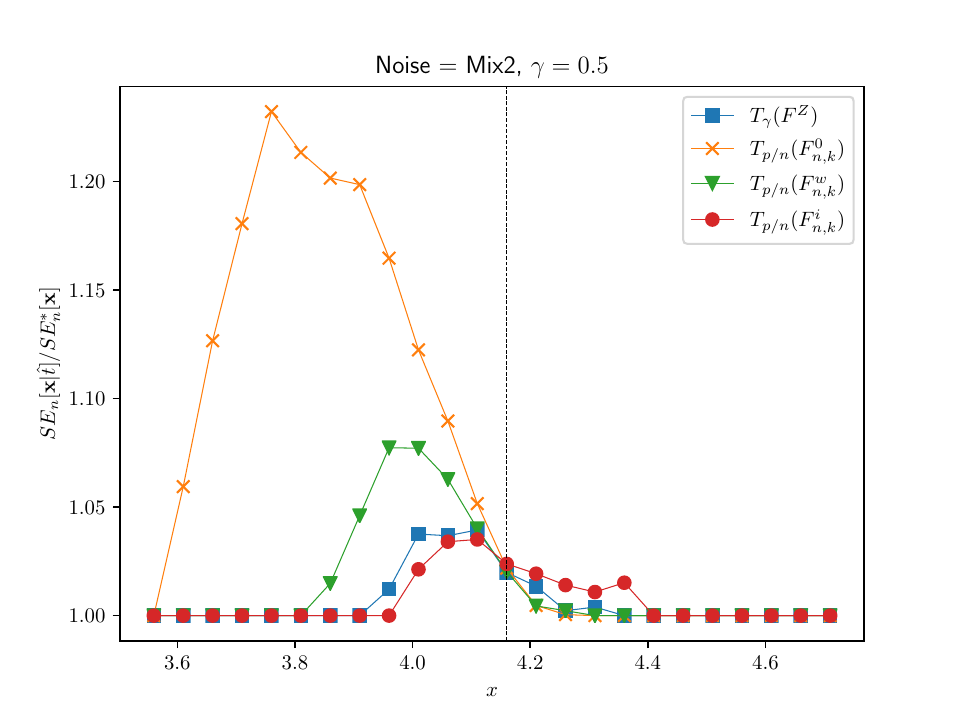}
    \caption{Experiment: {\bf Regret}}
\end{figure}

\begin{figure}[H]
    \centering
    \includegraphics[width=0.75\textwidth]{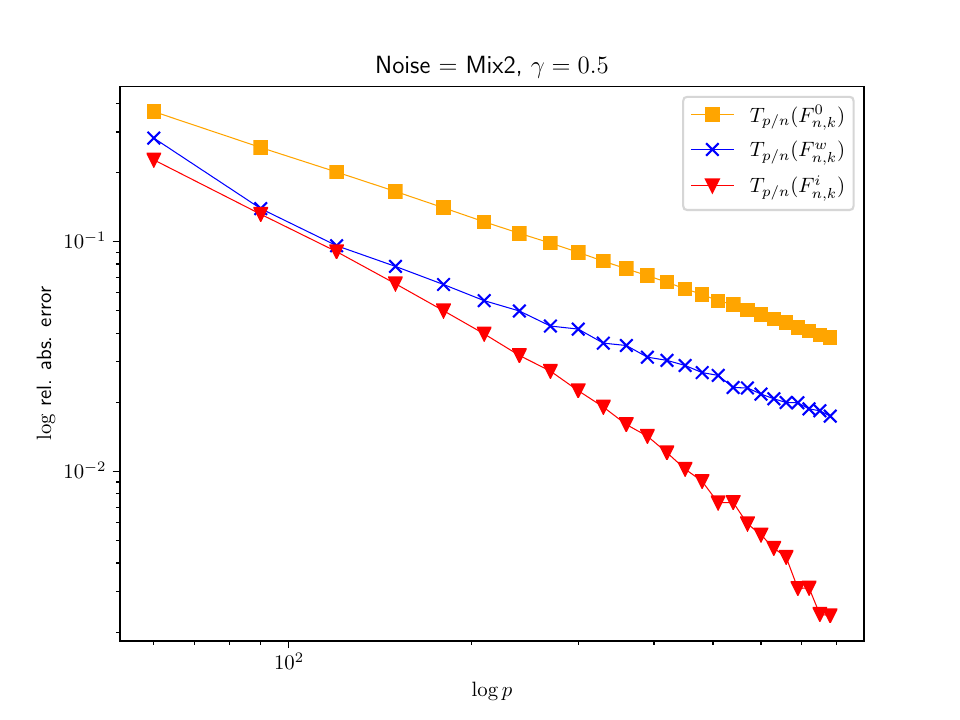}
    \caption{Experiment: {\bf ConvergenceRate}}
\end{figure}

\subsection{Distribution: Mix2, $\gamma = 1.0$}

\begin{figure}[H]
    \centering
    \includegraphics[width=0.75\textwidth]{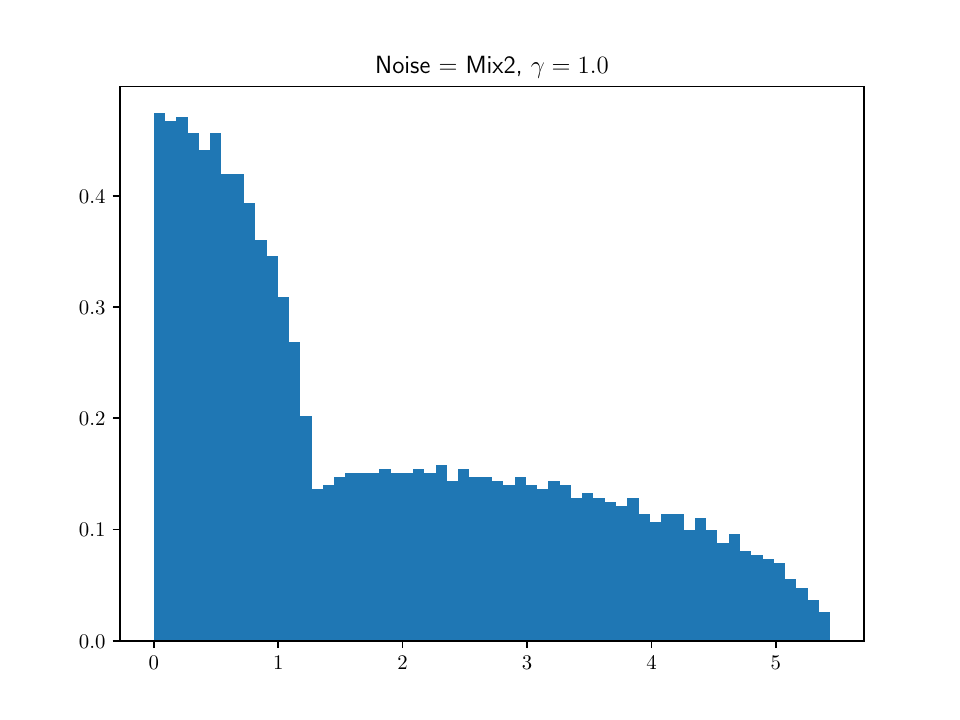}
    \caption{Experiment: {\bf Hist}}
\end{figure}

\begin{figure}[H]
    \centering
    \includegraphics[width=0.75\textwidth]{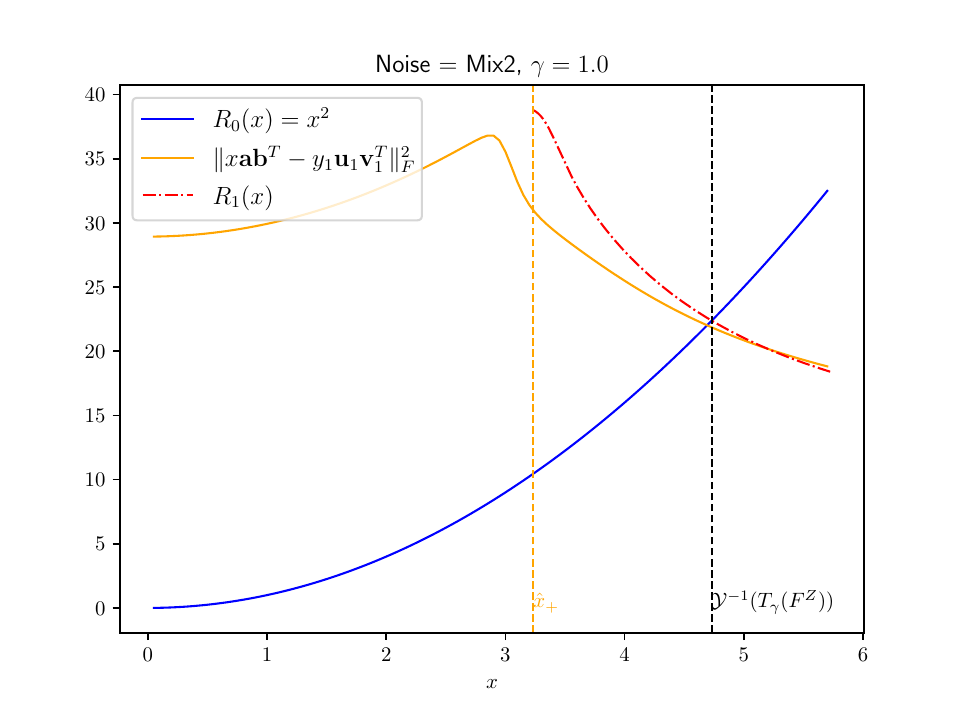}
    \caption{Experiment: {\bf R0-vs-R1}}
\end{figure}

\begin{figure}[H]
    \centering
    \includegraphics[width=0.75\textwidth]{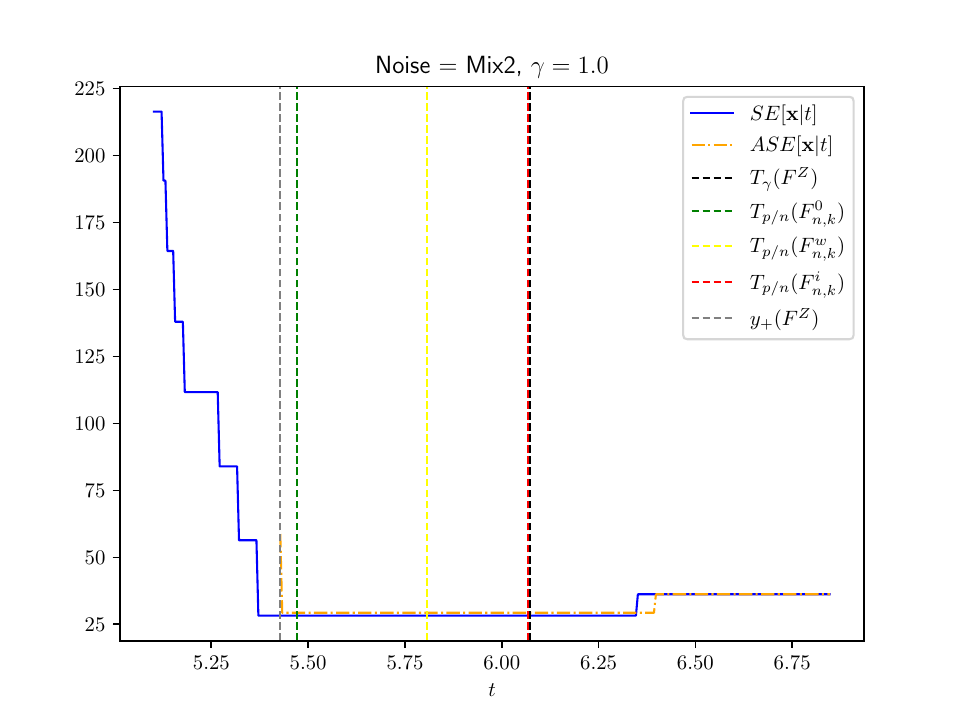}
    \caption{Experiment: {\bf SE-vs-ASE}}
\end{figure}

\begin{figure}[H]
    \centering
    \includegraphics[width=0.75\textwidth]{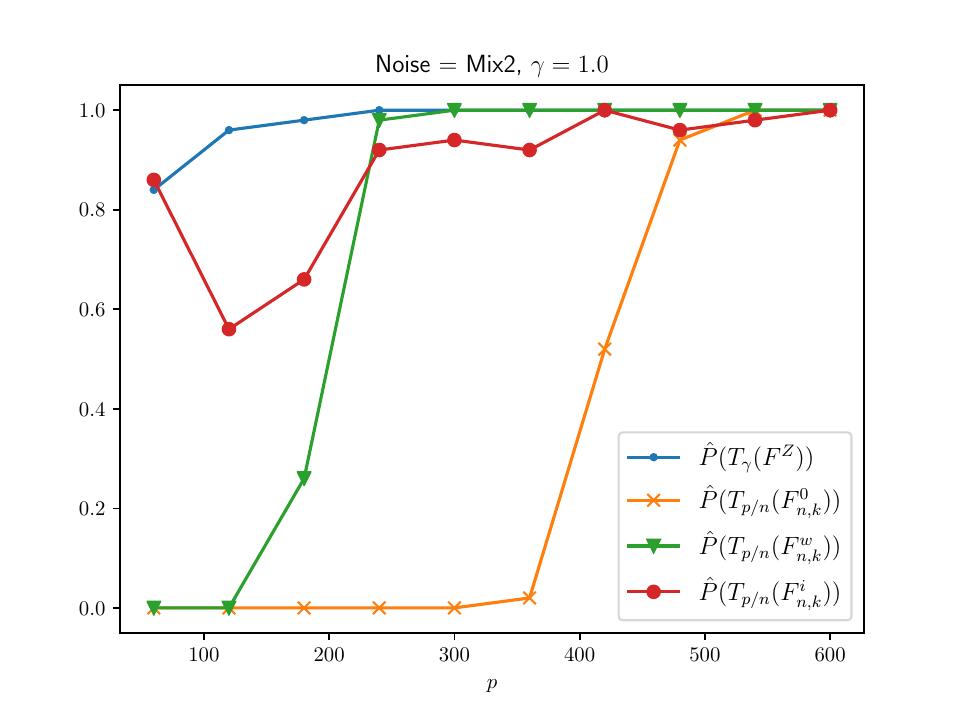}
    \caption{Experiment: {\bf OracleAttainment}}
\end{figure}

\begin{figure}[H]
    \centering
    \includegraphics[width=0.75\textwidth]{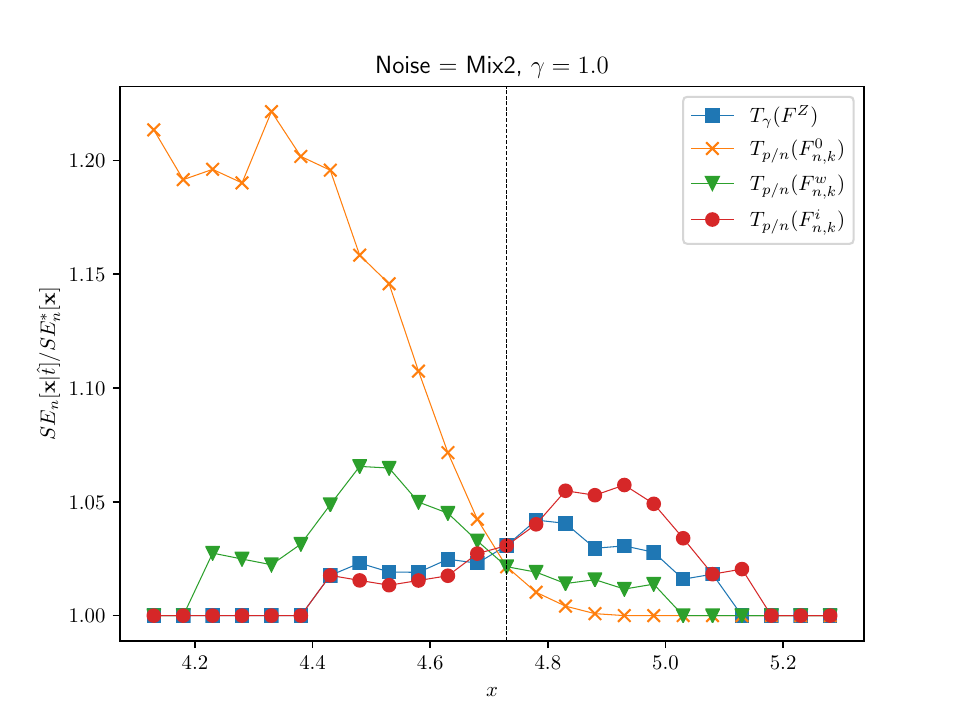}
    \caption{Experiment: {\bf Regret}}
\end{figure}

\begin{figure}[H]
    \centering
    \includegraphics[width=0.75\textwidth]{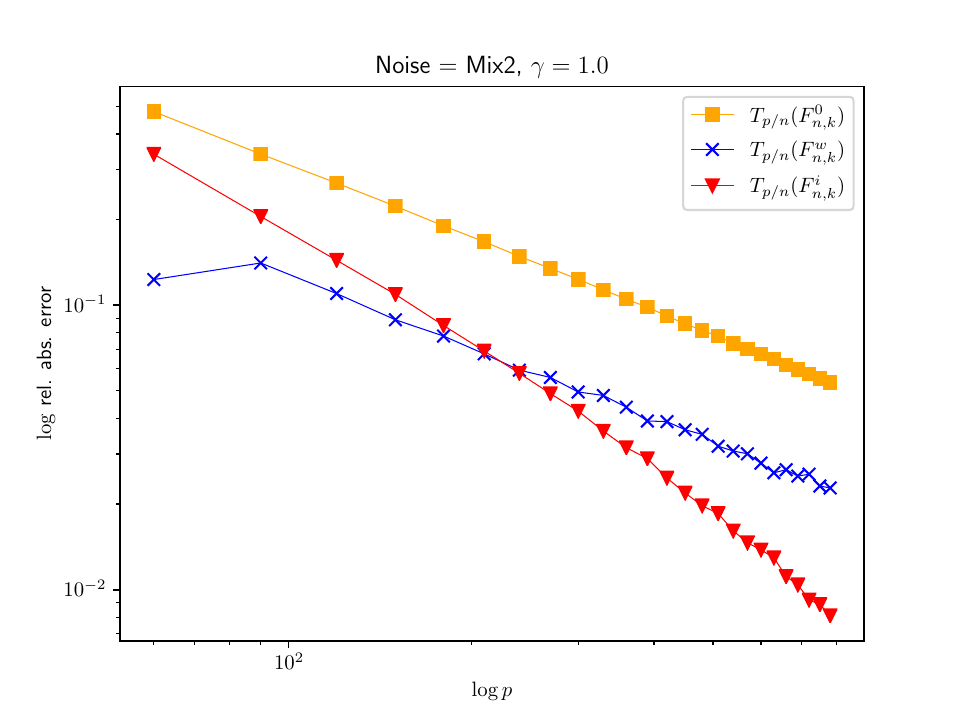}
    \caption{Experiment: {\bf ConvergenceRate}}
\end{figure}

\subsection{Distribution: Unif[1,10], $\gamma = 0.5$}

\begin{figure}[H]
    \centering
    \includegraphics[width=0.75\textwidth]{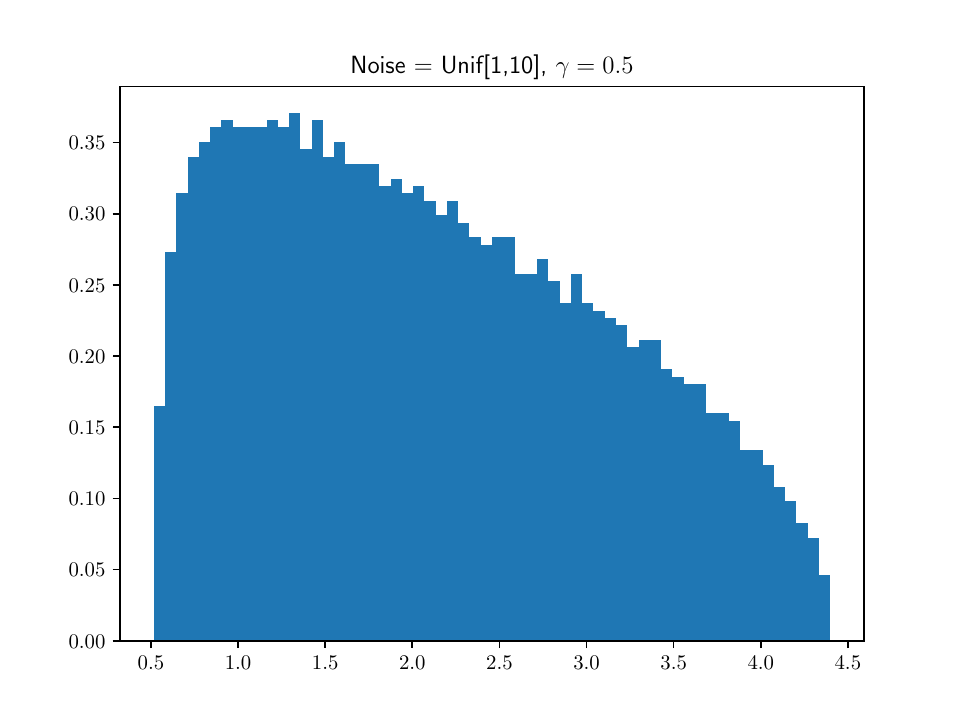}
    \caption{Experiment: {\bf Hist}}
\end{figure}

\begin{figure}[H]
    \centering
    \includegraphics[width=0.75\textwidth]{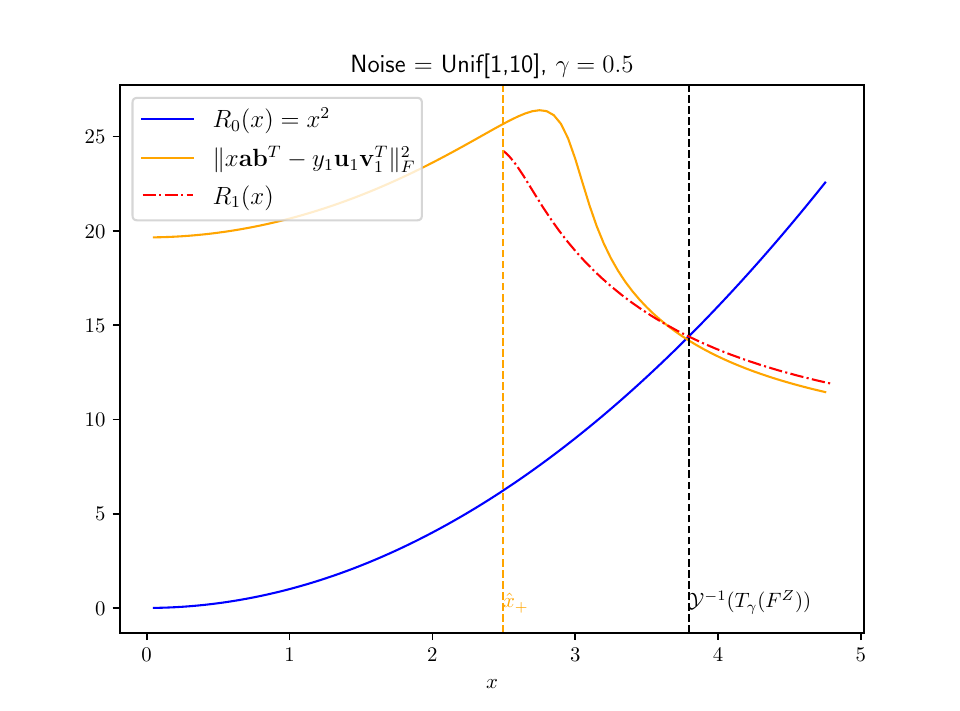}
    \caption{Experiment: {\bf R0-vs-R1}}
\end{figure}

\begin{figure}[H]
    \centering
    \includegraphics[width=0.75\textwidth]{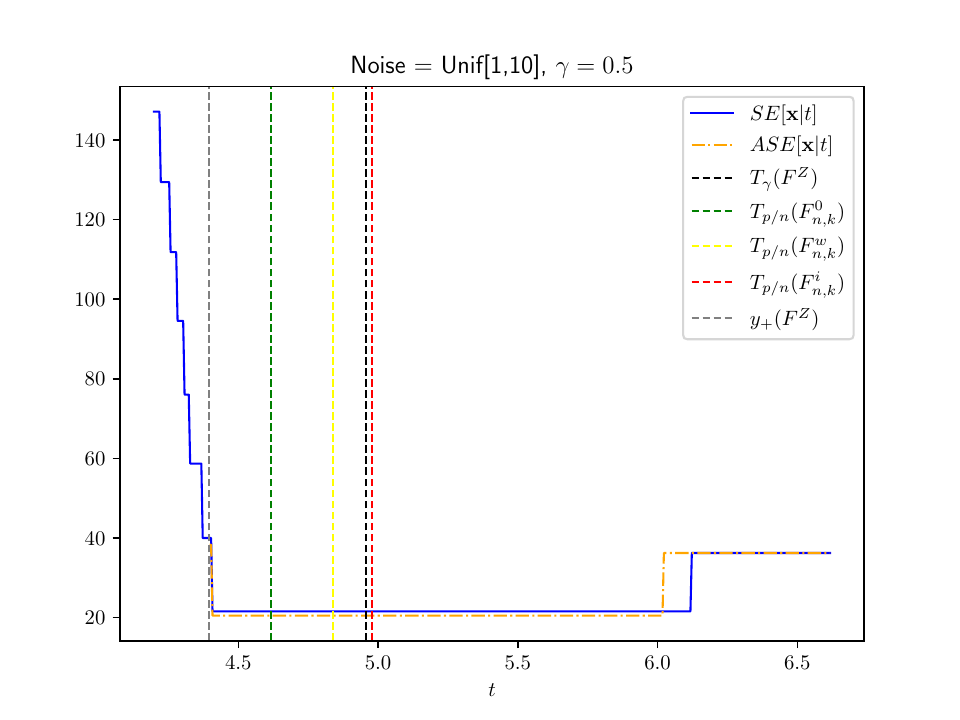}
    \caption{Experiment: {\bf SE-vs-ASE}}
\end{figure}

\begin{figure}[H]
    \centering
    \includegraphics[width=0.75\textwidth]{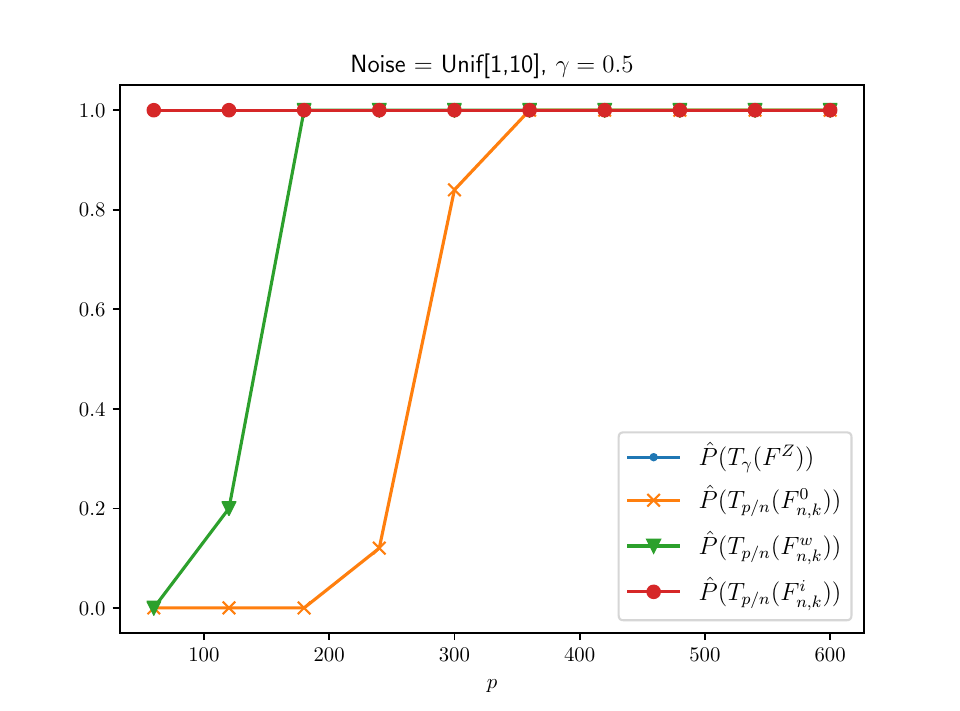}
    \caption{Experiment: {\bf OracleAttainment}}
\end{figure}

\begin{figure}[H]
    \centering
    \includegraphics[width=0.75\textwidth]{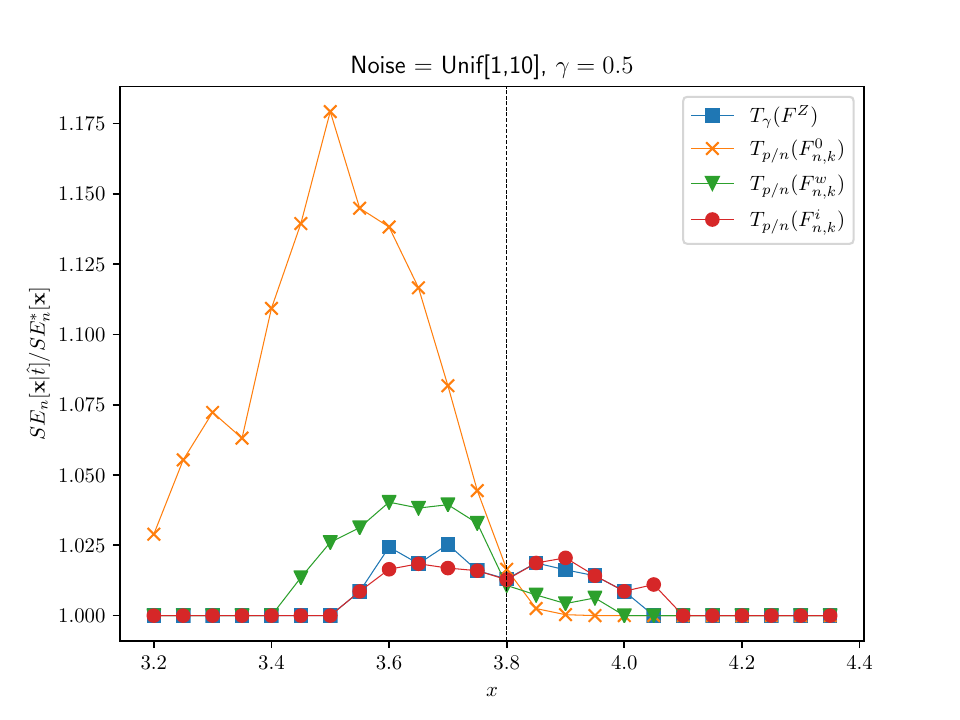}
    \caption{Experiment: {\bf Regret}}
\end{figure}

\begin{figure}[H]
    \centering
    \includegraphics[width=0.75\textwidth]{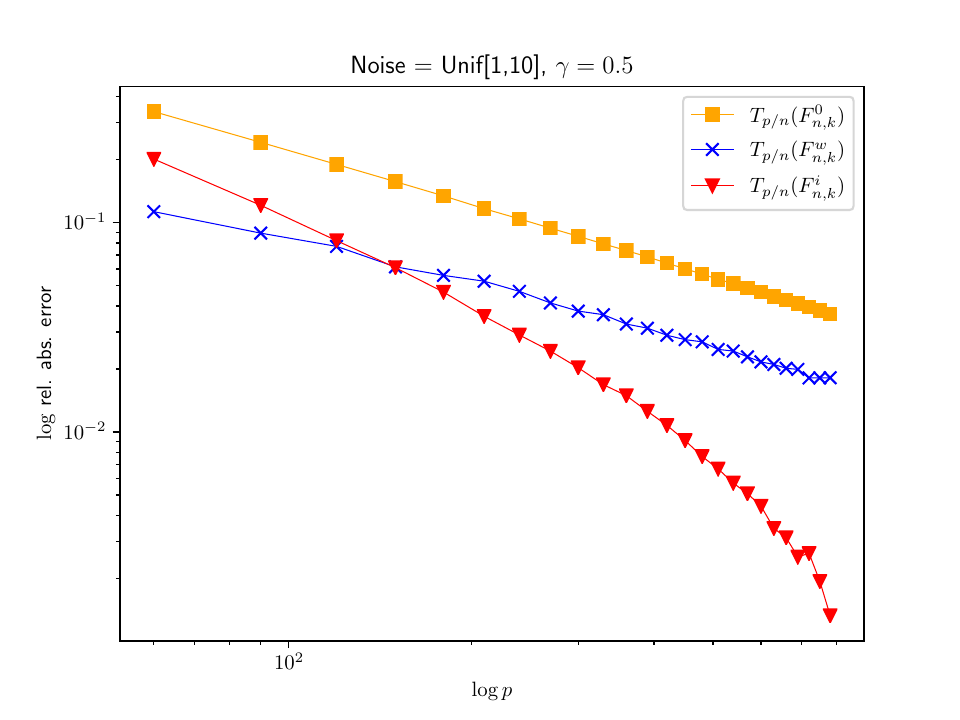}
    \caption{Experiment: {\bf ConvergenceRate}}
\end{figure}

\subsection{Distribution: Unif[1,10], $\gamma = 1.0$}

\begin{figure}[H]
    \centering
    \includegraphics[width=0.75\textwidth]{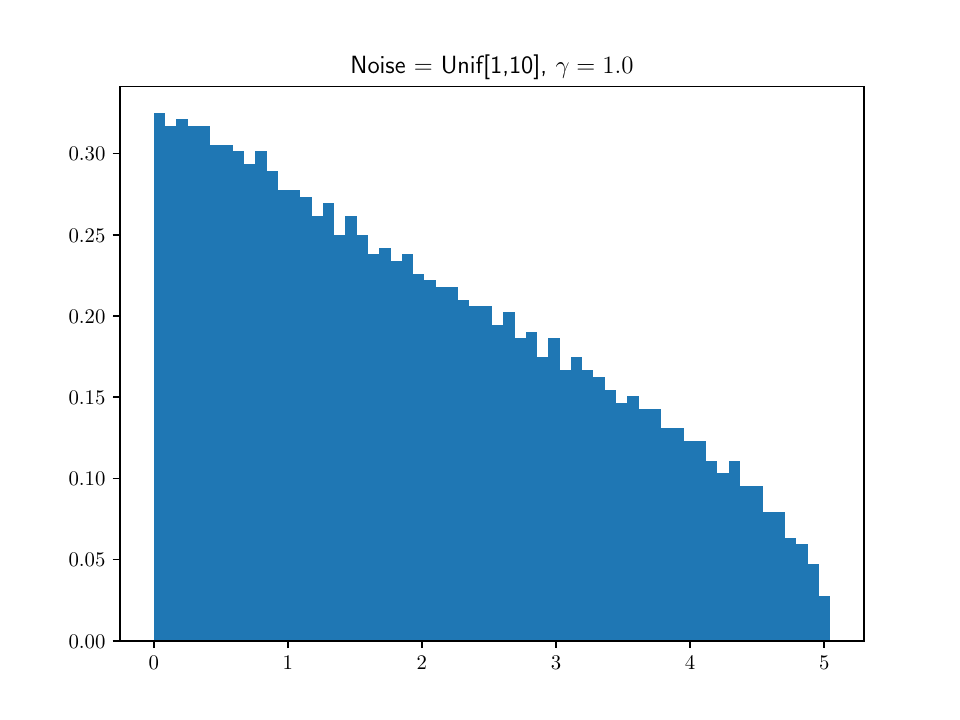}
    \caption{Experiment: {\bf Hist}}
\end{figure}

\begin{figure}[H]
    \centering
    \includegraphics[width=0.75\textwidth]{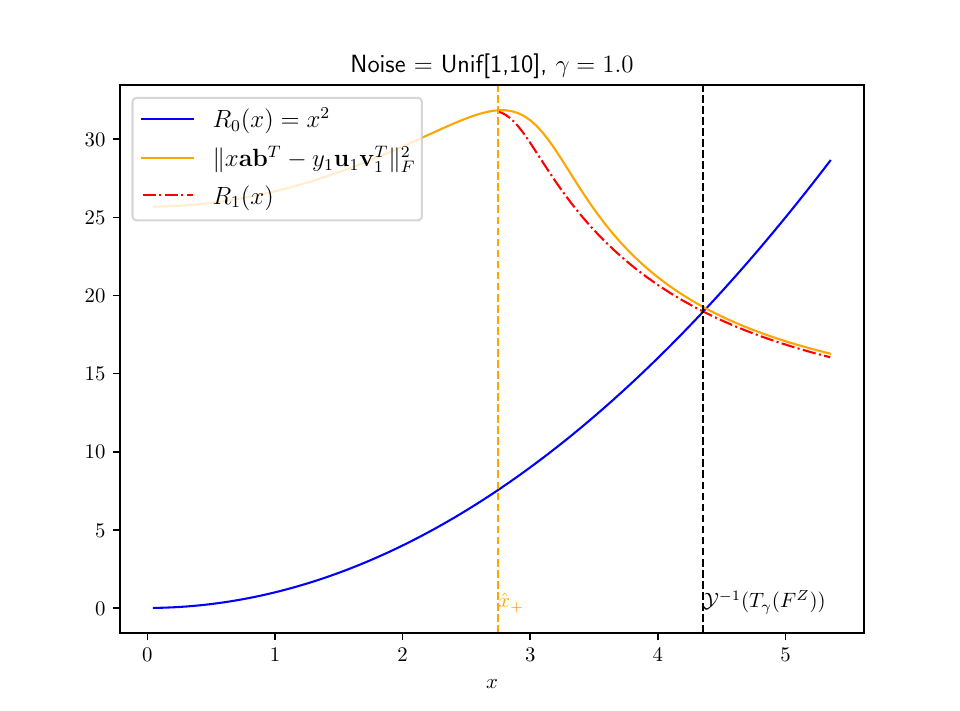}
    \caption{Experiment: {\bf R0-vs-R1}}
\end{figure}

\begin{figure}[H]
    \centering
    \includegraphics[width=0.75\textwidth]{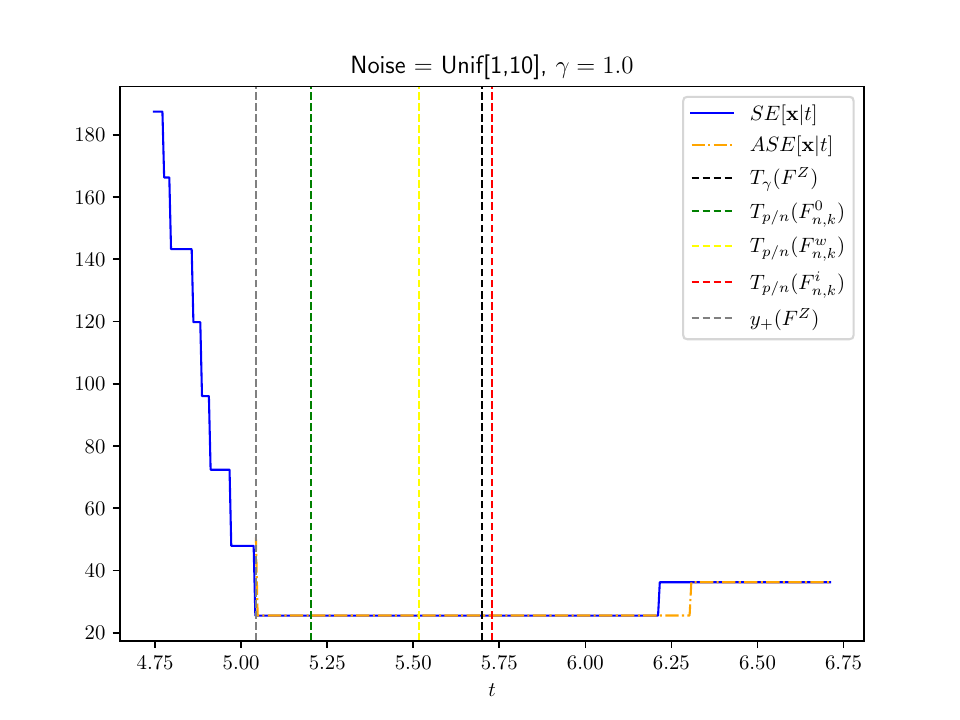}
    \caption{Experiment: {\bf SE-vs-ASE}}
\end{figure}

\begin{figure}[H]
    \centering
    \includegraphics[width=0.75\textwidth]{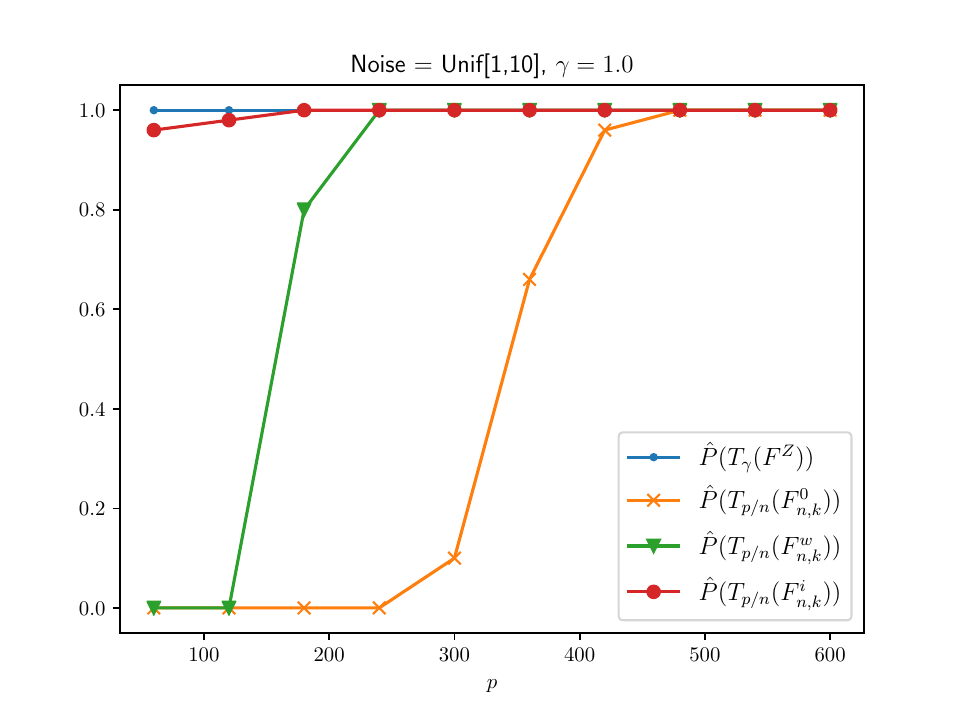}
    \caption{Experiment: {\bf OracleAttainment}}
\end{figure}

\begin{figure}[H]
    \centering
    \includegraphics[width=0.75\textwidth]{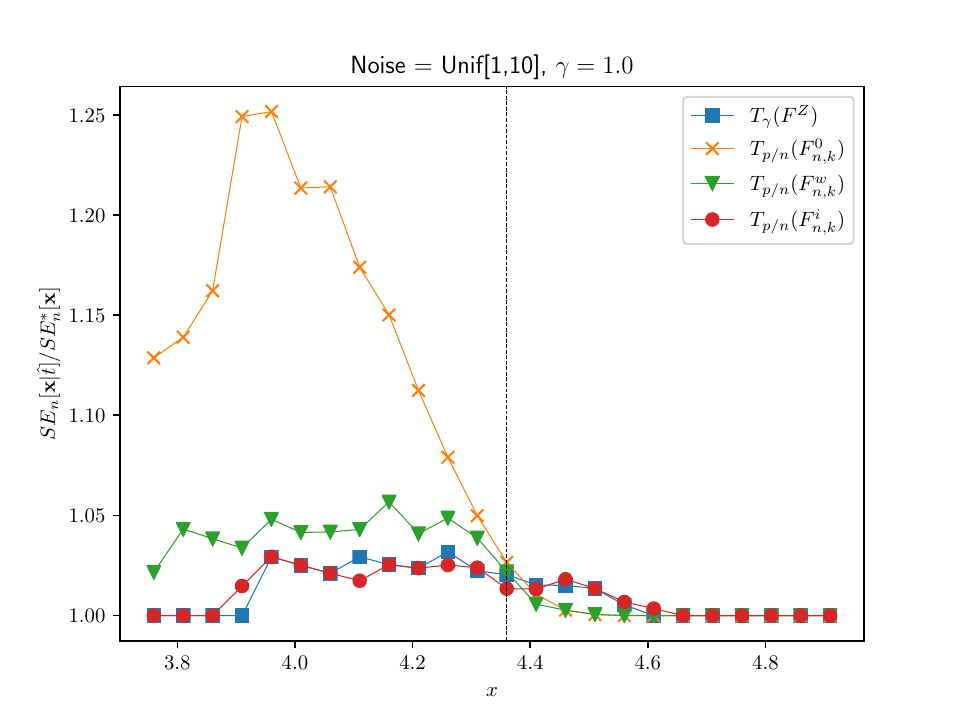}
    \caption{Experiment: {\bf Regret}}
\end{figure}

\begin{figure}[H]
    \centering
    \includegraphics[width=0.75\textwidth]{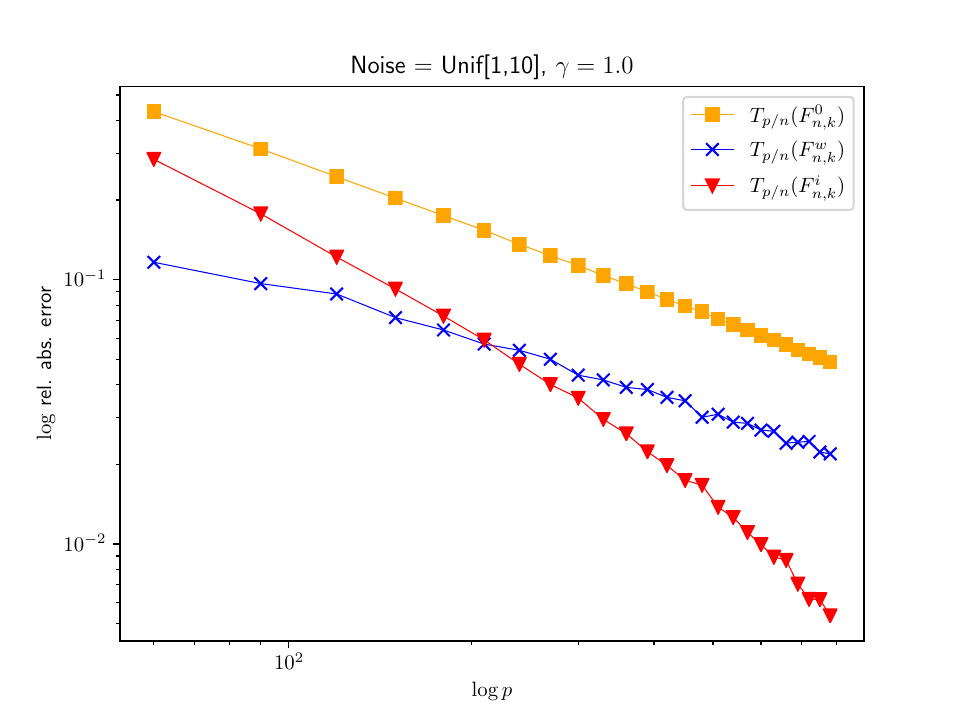}
    \caption{Experiment: {\bf ConvergenceRate}}
\end{figure}

\subsection{Distribution: PaddedIdentity, $\gamma = 0.5$}

\begin{figure}[H]
    \centering
    \includegraphics[width=0.75\textwidth]{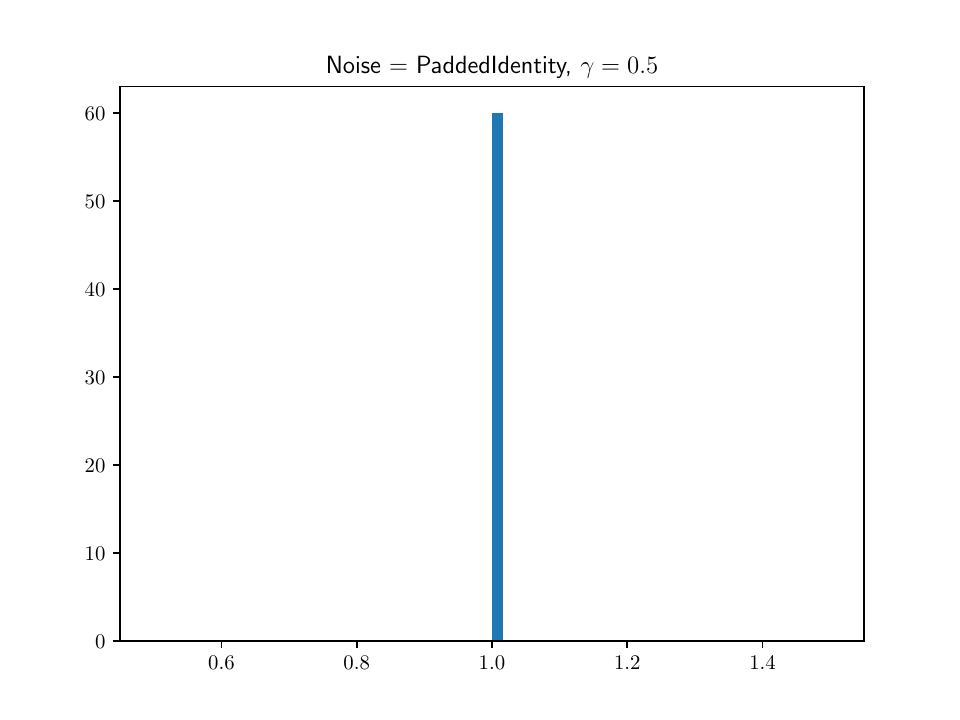}
    \caption{Experiment: {\bf Hist}}
\end{figure}

\begin{figure}[H]
    \centering
    \includegraphics[width=0.75\textwidth]{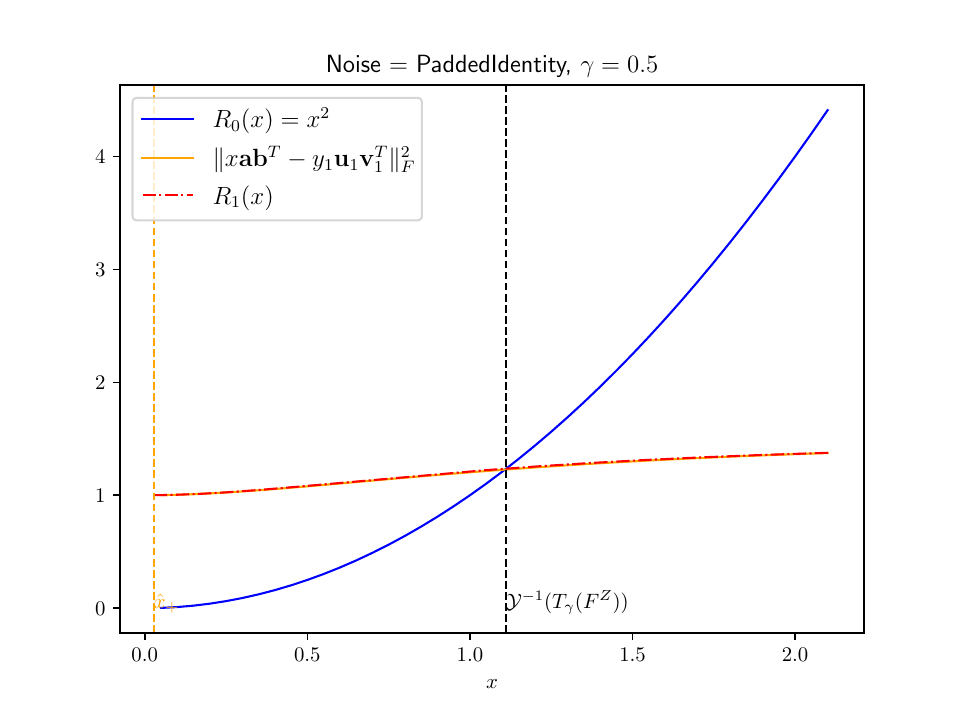}
    \caption{Experiment: {\bf R0-vs-R1}}
\end{figure}

\begin{figure}[H]
    \centering
    \includegraphics[width=0.75\textwidth]{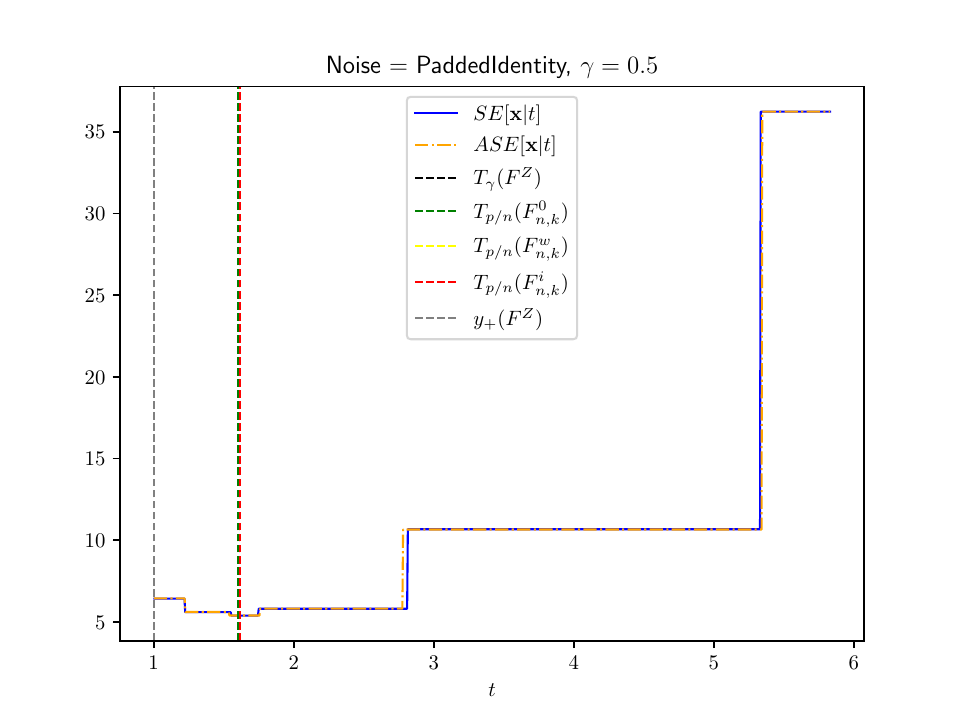}
    \caption{Experiment: {\bf SE-vs-ASE}}
\end{figure}

\begin{figure}[H]
    \centering
    \includegraphics[width=0.75\textwidth]{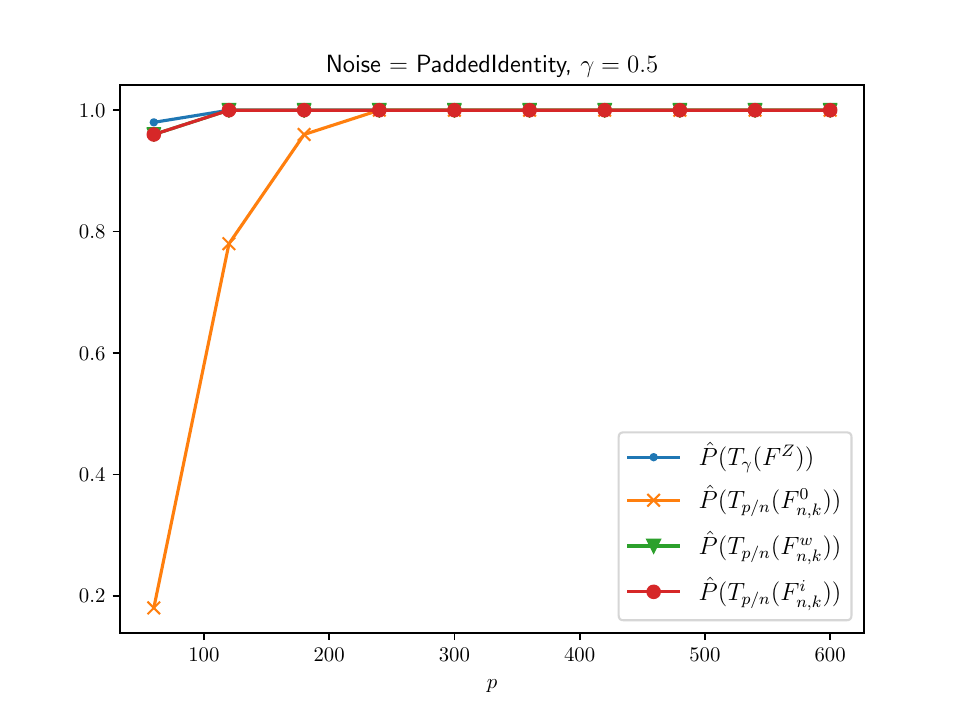}
    \caption{Experiment: {\bf OracleAttainment}}
\end{figure}

\begin{figure}[H]
    \centering
    \includegraphics[width=0.75\textwidth]{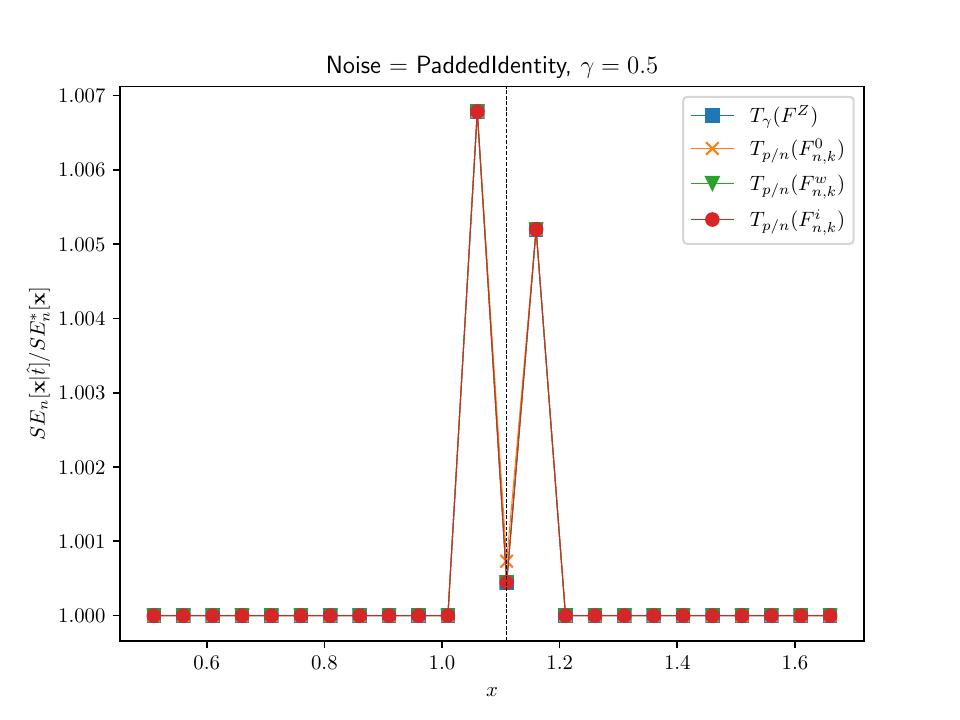}
    \caption{Experiment: {\bf Regret}}
\end{figure}

\begin{figure}[H]
    \centering
    \includegraphics[width=0.75\textwidth]{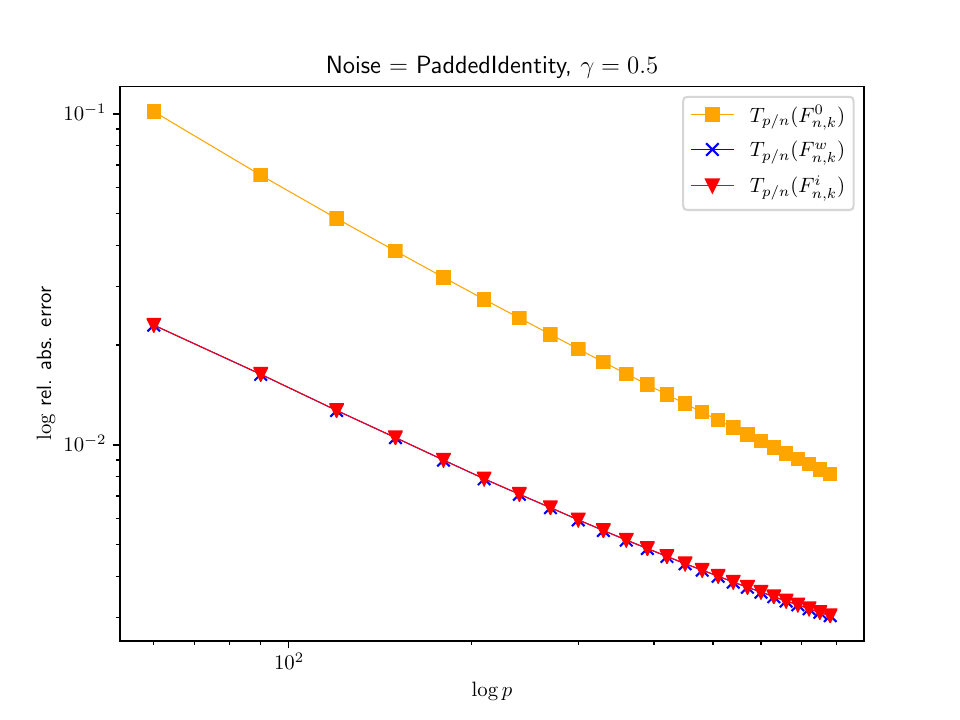}
    \caption{Experiment: {\bf ConvergenceRate}}
\end{figure}

\subsection{Distribution: PaddedIdentity, $\gamma = 1.0$}

\begin{figure}[H]
    \centering
    \includegraphics[width=0.75\textwidth]{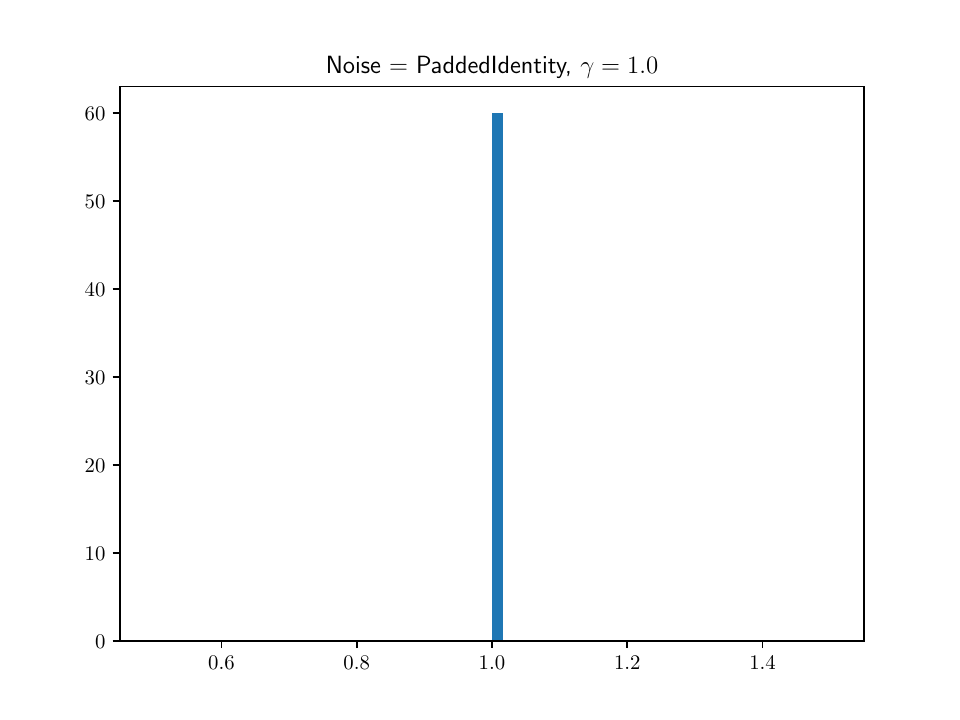}
    \caption{Experiment: {\bf Hist}}
\end{figure}

\begin{figure}[H]
    \centering
    \includegraphics[width=0.75\textwidth]{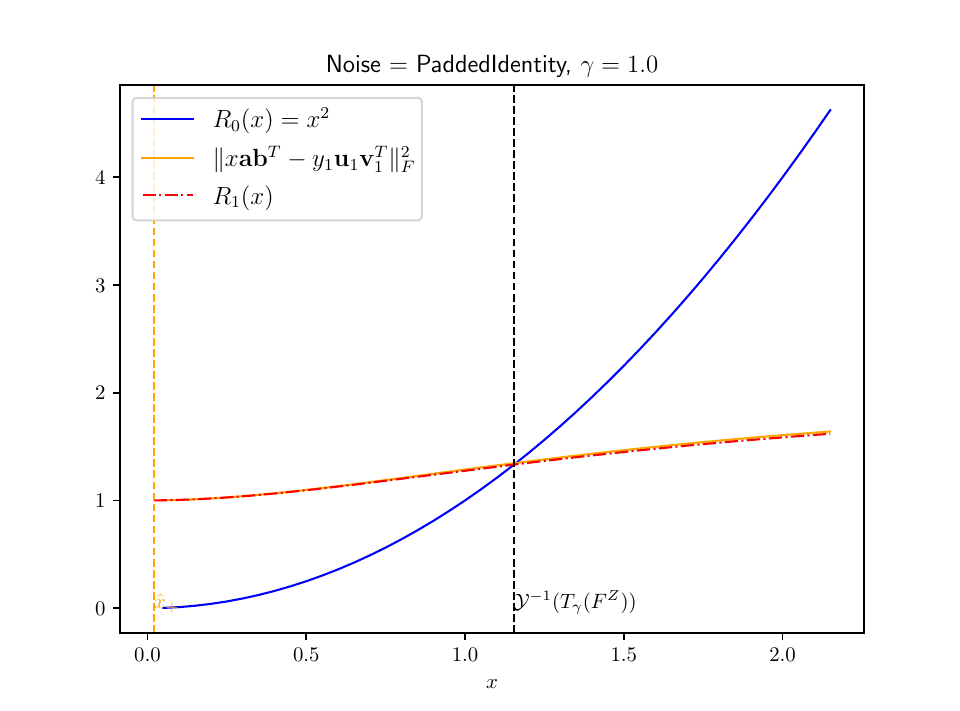}
    \caption{Experiment: {\bf R0-vs-R1}}
\end{figure}

\begin{figure}[H]
    \centering
    \includegraphics[width=0.75\textwidth]{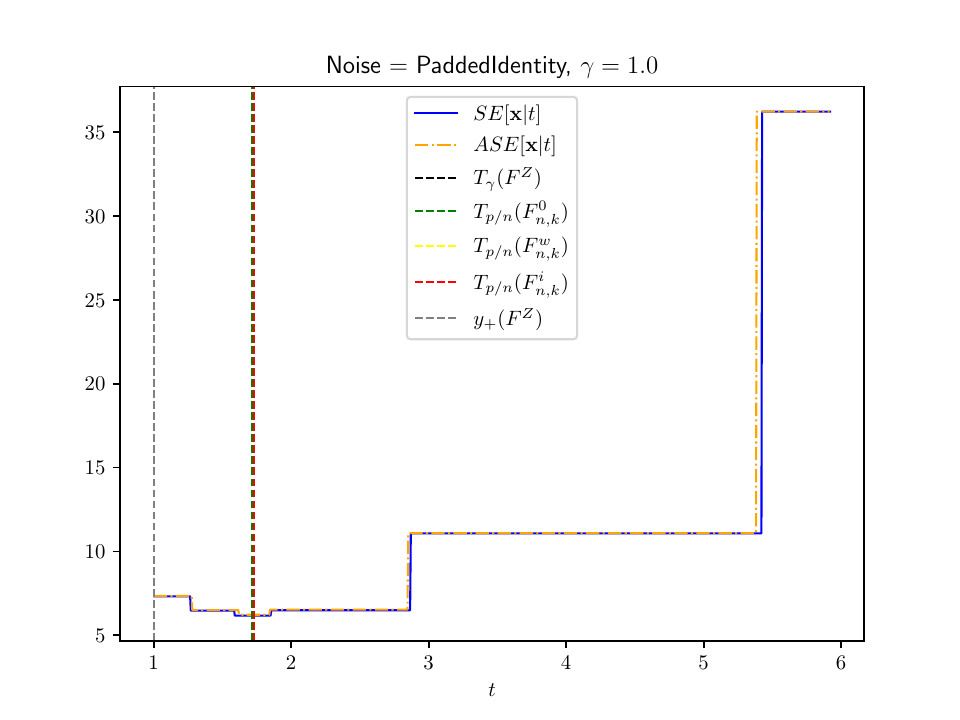}
    \caption{Experiment: {\bf SE-vs-ASE}}
\end{figure}

\begin{figure}[H]
    \centering
    \includegraphics[width=0.75\textwidth]{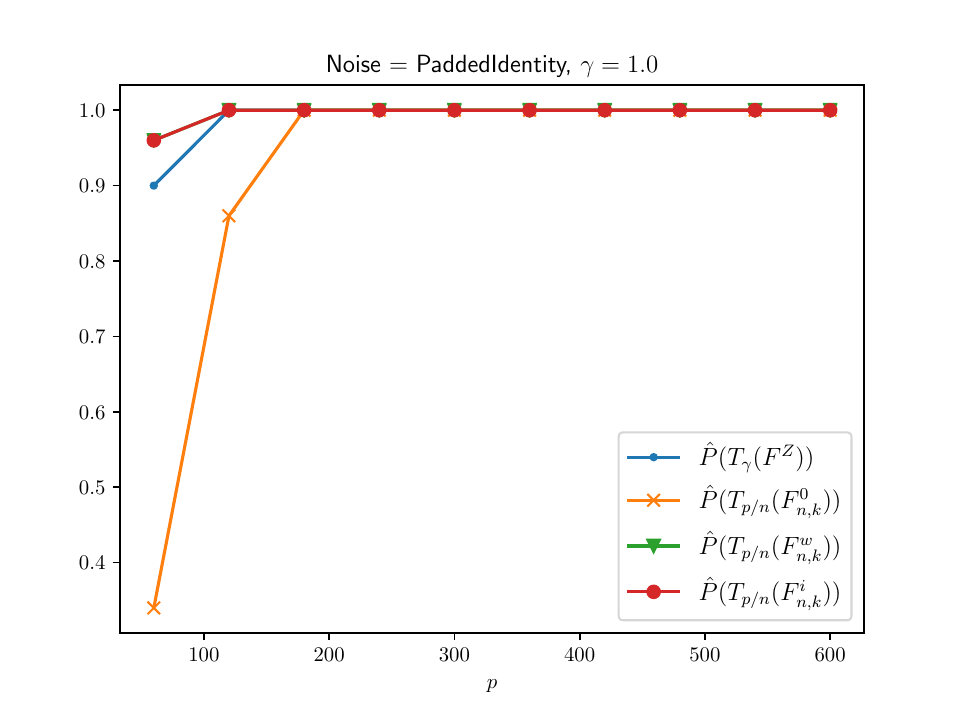}
    \caption{Experiment: {\bf OracleAttainment}}
\end{figure}

\begin{figure}[H]
    \centering
    \includegraphics[width=0.75\textwidth]{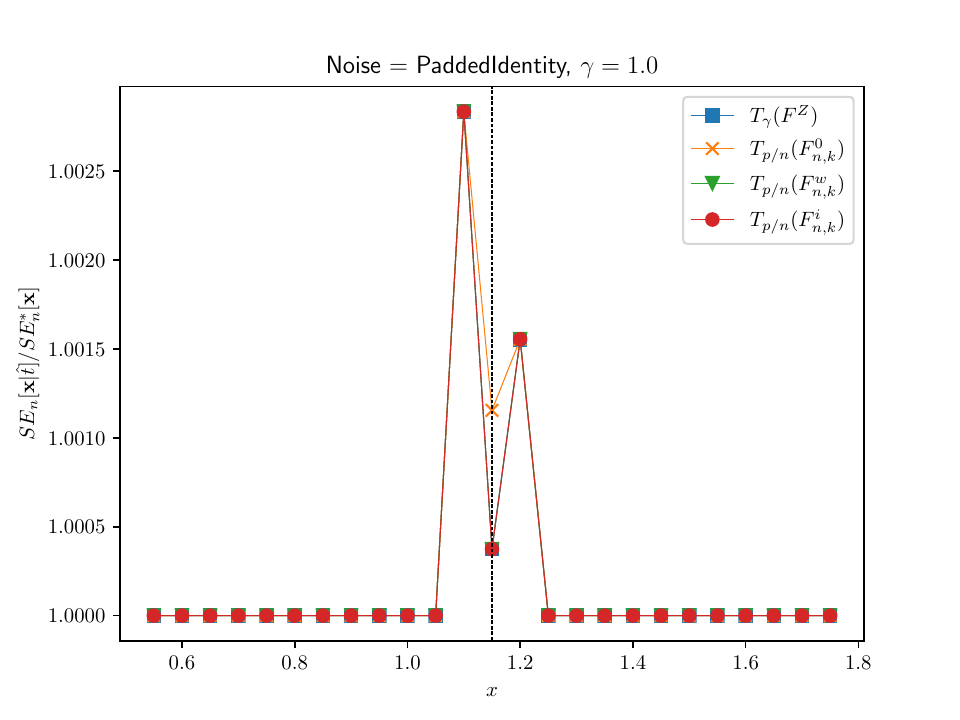}
    \caption{Experiment: {\bf Regret}}
\end{figure}

\begin{figure}[H]
    \centering
    \includegraphics[width=0.75\textwidth]{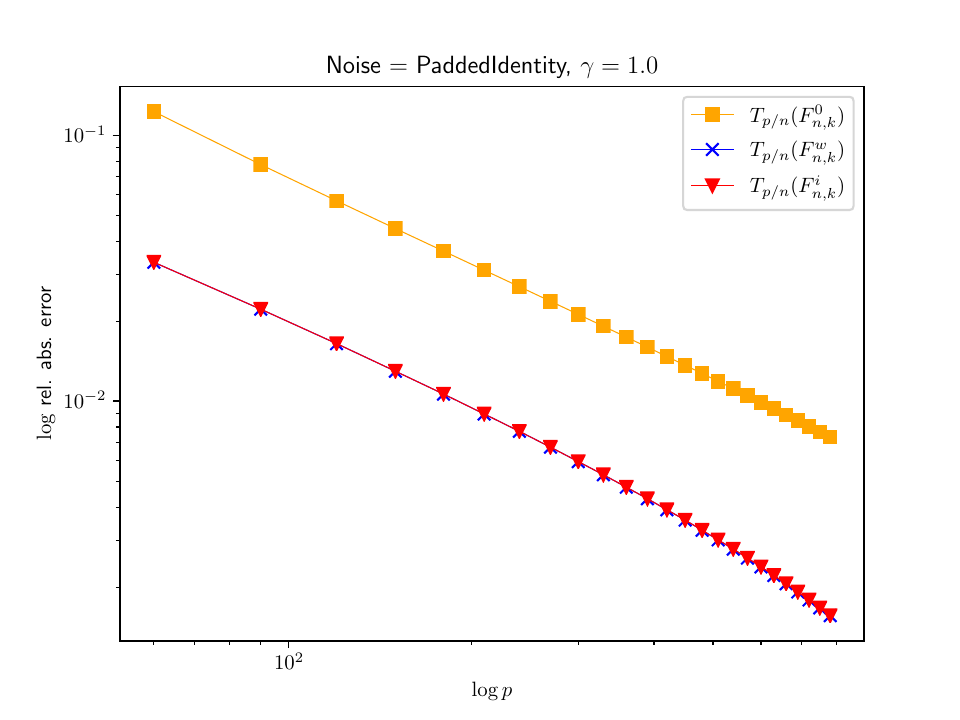}
    \caption{Experiment: {\bf ConvergenceRate}}
\end{figure}

\end{document}